\documentclass[final,reqno]{siamart1116}

\usepackage{amsfonts,amsmath,extarrows,amssymb}
\usepackage{graphicx,epstopdf}
\usepackage{mathrsfs}
\usepackage{mathtools}
\usepackage{bm}
\usepackage{multirow}
\usepackage{color}
\usepackage{hyperref} 
\hypersetup{colorlinks=false} %

\usepackage{cleveref}

\usepackage{caption, subcaption}%
\usepackage[font=normal]{caption} 
\usepackage{array}
\usepackage{enumitem}   
\usepackage{booktabs}

\usepackage{xparse} 
\usepackage{stmaryrd} 
\usepackage{rotfloat}

\usepackage{algorithm}
\usepackage{algpseudocode}

\usepackage{tikz}
\usetikzlibrary{positioning}
\usetikzlibrary{shapes, arrows}
\usetikzlibrary{arrows.meta}
\usetikzlibrary{calc}
\tikzstyle{terminator} = [rectangle, draw, text centered, rounded corners]
\tikzstyle{process} = [rectangle, draw, text centered]
\tikzstyle{decision} = [diamond, draw, text centered]
\tikzstyle{data}=[trapezium, draw, text centered, trapezium left angle=60, trapezium right angle=120]
\tikzstyle{connector} = [draw, -{latex[length=2mm]}, thick, black]
\tikzstyle{connector2} = [draw, -{latex[length=2mm]}, thick, green]
\tikzstyle{connector3} = [draw, -{latex[length=2mm]}, thick, blue]
\pgfdeclarelayer{background}
\pgfsetlayers{background,main}

\usepackage{geometry}
\newsiamthm{claim}{Claim}
\newsiamremark{remark}{Remark}
\newsiamremark{example}{Example}
\newsiamthm{problem}{Problem}
\newsiamthm{algorithm1}{Algorithm}
\newsiamthm{conjecture}{Conjecture}
\newsiamthm{corollary1}{Corollary}
\numberwithin{theorem}{section}

\def\vec#1{{\bf #1}}

\makeatletter
\def\widebreve{\mathpalette\wide@breve}
\def\wide@breve#1#2{\sbox\z@{$#1#2$}%
	\mathop{\vbox{\m@th\ialign{##\crcr
				\kern0.08em\brevefill#1{0.8\wd\z@}\crcr\noalign{\nointerlineskip}%
				$\hss#1#2\hss$\crcr}}}\limits}
\def\brevefill#1#2{$\m@th\sbox\tw@{$#1($}%
	\hss\resizebox{#2}{\wd\tw@}{\rotatebox[origin=c]{90}{\upshape(}}\hss$}

\makeatletter

\makeatletter
\newcommand\footnoteref[1]{\protected@xdef\@thefnmark{\ref{#1}}\@footnotemark}
\makeatother

\newcommand{\TheTitle}
{On Local Minimum Entropy Principle of High-Order Schemes for Relativistic Euler Equations}

\newcommand{\BU}{\mathbf{U}}
\newcommand{\BA}{\mathbf{A}}
\newcommand{\BF}{\mathbf{F}}

\newcommand{\BV}{\mathbf{V}}
\newcommand{\Bn}{{\bm n}}
\newcommand{\cn}{\cdot\Bn_*}
\newcommand{\Be}{\mathbf{e}}

\newcommand{\tp}{\varphi}
\newcommand{\tD}{\widetilde{D}}
\newcommand{\tv}{\tilde{v}}

\newcommand{\BQ}{\mathbf{Q}}
\newcommand{\dBQ}{\widetilde{\BQ}}

\newcommand{\oBv}{\overline\Bv}
\newcommand{\oD}{\overline D}
\newcommand{\ov}{\overline v}

\newcommand{\oBU}{\overline\BU}
\newcommand{\tBU}{\widetilde\BU}

\newcommand{\oBUk}{\oBU^n_K}
\newcommand{\oDk}{\oD^n_K}
\newcommand{\oBvk}{\oBv^n_K}

\newcommand{\oBUj}{\oBU^n_{K_j}}
\newcommand{\oDj}{\oD^n_{K_j}}
\newcommand{\oBvj}{\oBv^n_{K_j}}

\newcommand{\C}{\mu}
\newcommand{\Cs}{\tilde\mu}

\newcommand{\Bm}{{\bm m}}
\newcommand{\Bv}{{\bm v}}
\newcommand{\Bx}{{\bm x}}

\newcommand{\Bxi}{{\bm \xi}}
\newcommand{\Bxij}{\Bxi^{(j)}_K}

\newcommand{\SE}{\mathscr E}
\newcommand{\SF}{\mathscr F}

\newcommand{\CH}{\mathcal{H}}
\newcommand{\CE}{\mathcal{E}}
\newcommand{\CEj}{\big|\CE_{K}^{(j)}\big|}

\newcommand{\BALL}{\mathbb{B}_1(\mathbf{0})}

\newcommand{\diag}{\operatorname{diag}}

\newcommand{\dt}{\Delta t}
\newcommand{\dx}{\Delta x}
\newcommand{\dy}{\Delta y}

\title{
	{\TheTitle}
	\thanks{
		This work is partially supported by Shenzhen Science and Technology Program (Grant No. RCJC20221008092757098) and National Natural Science Foundation of China (Grant No. 12171227).
		The work of S. Cui is also partially supported by NSF of Guangdong Province (Grant No. 2024A1515012329) and National Natural Science Foundation of China (Grant No. 12471382).
	}
}

\author{
	Shumo Cui 
	\thanks{
		Shenzhen International Center for Mathematics and Department of Mathematics, Southern University of Science and Technology, Shenzhen 518055, China.
		({\tt cuism@sustech.edu.cn})
	}
	\and
	Kailiang Wu
	\thanks{
		Corresponding author. Department of Mathematics and Shenzhen International Center for Mathematics, Southern University of Science and Technology, Shenzhen 518055, China.
		({\tt wukl@sustech.edu.cn})
	}
	\and
	Linfeng Xu
	\thanks{
		Department of Mathematics, Southern University of Science and Technology, Shenzhen 518055, P.R.~China. 
		({\tt xulf2022@mail.sustech.edu.cn})
	}
} 

\allowdisplaybreaks

\begin{document}
\large 
	
\maketitle


\begin{abstract}
\normalsize The minimum entropy principle (MEP), first established in [E.~Tadmor, {\em Appl.~Numer.~Math.}, 2:211--219, 1986] for the nonrelativistic Euler system, states that the minimum of the initial specific entropy serves as a lower bound for the specific entropy at all future times. This fundamental principle provides the best-known a priori estimate for entropy in the (nonrelativistic) Euler equations and has been successfully incorporated into the design of stable numerical schemes (see, e.g., [B.~Khobalatte \& B.~Perthame, {\em Math. Comp.}, 62:119-131, 1994; X.~Zhang \& C.-W.~Shu, {\em Numer. Math.}, 121:545-563, 2012]). However, compared to the nonrelativistic case, the understanding of entropy in relativistic Euler equations remains far more limited, due to the complexities of nonlinear entropy and the intricate mathematical structure of the relativistic Euler system.

This paper first establishes the MEP for the relativistic Euler equations with a broad class of general equations of state (EOSs) that satisfy relativistic causality. Furthermore, we address the challenge of preserving the local version of the discovered MEP in high-order numerical schemes. At the continuous level, we find out a family of entropy pairs for the relativistic Euler equations with a general EOS and provide rigorous analysis to prove the strict convexity of entropy under a necessary and sufficient condition. At the numerical level, we develop a rigorous framework for designing provably entropy-preserving high-order schemes that ensure both physical admissibility and the discovered MEP. The relativistic effects, coupled with the abstract and general EOS formulation, introduce significant challenges not encountered in the nonrelativistic case or with the ideal EOS. In particular, entropy is a highly nonlinear and implicit function of the conservative variables, making it particularly difficult to enforce entropy preservation. To address these challenges, we establish a series of auxiliary theories via highly technical inequalities. Another key innovation is the use of geometric quasi-linearization (GQL), which reformulates the nonlinear constraints into equivalent linear ones by introducing additional free parameters. These advancements form the foundation of our entropy-preserving analysis. We propose novel, robust, locally entropy-preserving high-order frameworks. A central challenge is accurately estimating the local minimum of entropy, particularly in the presence of shock waves at unknown locations. To address this, we introduce two new approaches for estimating local lower bounds of specific entropy, which prove effective for both smooth and discontinuous problems. Numerical experiments demonstrate that our entropy-preserving methods maintain high-order accuracy while effectively suppressing spurious oscillations, outperforming existing local entropy minimum estimation techniques in the literature. Moreover, our approach is not limited to the relativistic Euler equations but can also be applied to other hydrodynamic models that admit an MEP.

\end{abstract}
	


\section{Introduction}

The principle of entropy non-decrease is a fundamental law of nature. For hyperbolic conservation laws, a key concept is the entropy condition. Due to the finite speed of wave propagation, solutions to nonlinear hyperbolic equations can develop discontinuities, such as shock waves, even from smooth initial data. This necessitates the use of weak solutions, which are typically non-unique. The entropy condition provides an additional criterion to single out the physically relevant solution. It plays a crucial role in ensuring both the uniqueness and stability of weak solutions, making it indispensable for guaranteeing that the solutions are both mathematically well-posed and physically meaningful.

In the 1980s, researchers began to recognize the importance of the entropy principle in numerical computations and introduced entropy-stable schemes that satisfy discrete entropy conditions  \cite{harten1976finite,crandall1980monotone,osher1984riemann,osher1988convergence,tadmor1987numerical,jiang1994cell,tadmor2003entropy,berthon2015entropy,berthon2016fully,berthon2023artificial}. Over the past few decades, significant advancements have been made in the development of high-order entropy-stable schemes \cite{lefloch2002fully, fjordholm2012arbitrarily, fjordholm2013eno, chen2017entropy, zhao2022strictly, duan2021, berthon2014entropy, hiltebrand2014entropy}. However, in the context of fluid dynamics, the study of entropy stability has typically been limited to a single special entropy pair derived from thermodynamics. Furthermore, most of the high-order schemes that can be proven to be entropy-stable are semi-discrete, meaning they are spatially discrete but continuous in time. Moreover, the proofs rely on the assumption that they preserve the positivity of density $\rho$ and pressure $p$. Developing fully discrete high-order schemes that can be rigorously proven to be entropy-stable remains an ongoing challenge.

Simultaneously, attention has also been paid to another important entropy law: the minimum entropy principle (MEP), originally discovered by Eitan Tadmor for the (non-relativistic) gas dynamics equations \cite{TADMOR1986211}:
\begin{equation} \label{LMEP}
	S({\bm x},t) 
	\geq
    \operatorname*{ess~inf}_{\|{\bm z} - {\bm x} \| \leq \tau  v_{\max} } S({\bm z},t-\tau) =: S^{\rm L}_{\min}({\bm x},t,\tau),
\end{equation}
where $\tau \in [0,t]$, $v_{\max}$ denotes the maximum propagation speed, and $S$ is the specific entropy, defined as 
	$S({\bf U}) = \log \big( {p}{\rho^{-\Gamma}} \big)$ with $\Gamma$ being the adiabatic index. The MEP was also shown by Guermond and
	Popov \cite{Guermond2014} via viscous regularization of the non-relativistic Euler equations. 
This principle asserts that the spatial minimum of the specific entropy is a non-decreasing function of time, meaning that the minimum value of specific entropy at the initial time serves as a lower bound for all future times: 
\begin{equation} \label{GMEP}
    S({\bm x},t) 
	\geq
    \operatorname*{ess~inf}_{\bm x}S({\bm x},0)=:S^{\rm G}_{\min} 
    \qquad
	\forall ~ t > 0.
\end{equation}
The MEP is recognized as the best known pointwise (a priori) entropy estimate for gas dynamics equations \cite{zhang2012minimum}. It is crucial to develop numerical methods that preserve the MEP, referred to as {\it{entropy-preserving}} schemes in this paper.

The analysis and design of entropy-preserving schemes have garnered considerable attention, as preserving the MEP enhances numerical stability and reduces spurious numerical oscillations \cite{TADMOR1986211, khobalatte1994maximum, zhang2012minimum, JIANG2018}. 
The preservation of the MEP can be viewed as a key property that characterizes the nonlinear stability of a numerical scheme.  
 Tadmor \cite{TADMOR1986211} demonstrated that Godunov-type and Lax–Friedrichs schemes preserve the MEP for (non-relativistic) gas dynamics equations. Coquel and Perthame \cite{coquel1998relaxation} studied first-order relaxation schemes satisfying the MEP for non-relativistic Euler equations. 
 As noted by Perthame and Shu \cite{perthame1996positivity}, {\em the MEP is ``extremely difficult to preserve for second- and higher-order schemes."} Khobalatte and Perthame \cite{khobalatte1994maximum} developed second-order entropy-preserving kinetic schemes, while Zhang and Shu \cite{zhang2012minimum} proposed a novel approach to enforcing the MEP in high-order finite volume and discontinuous Galerkin (DG) schemes. Guermond et al.~\cite{guermond2016invariant, guermond2018second, guermond2019invariant} studied continuous finite element methods that preserve the MEP. Jiang and Liu \cite{JIANG2018} introduced a new limiter designed to preserve the MEP for DG schemes. Gouasmi et al.~\cite{gouasmi2020minimum} proved the MEP for compressible multicomponent Euler equations. 
All of these works focused on numerically preserving the {\bf global} MEP \eqref{GMEP}, rather than a discrete version of the {\bf local} MEP \eqref{LMEP}. For convenience, we refer to the numerical schemes that preserve the local MEP as {\em locally entropy-preserving} schemes. Lv and Ihme \cite{LV2015715} observed the notable advantages of enforcing the local MEP over the global MEP, particularly for controlling spurious oscillations more effectively. They also proposed a high-order DG scheme for the (non-relativistic) Euler equations that preserves the local MEP \cite{LV2015715}. The key challenge in designing locally entropy-preserving schemes is providing a suitable estimate of the local minimum of the specific entropy. More recently, Ching, Johnson, and Kercher \cite{ching2024positivity, ching2024positivity2,ching2025positivity} developed locally entropy-preserving schemes for chemically reacting compressible Euler and Navier--Stokes equations. The local MEP of approximate Riemann solvers was used to obtain the first-order entropy-stable schemes for non-relativistic Euler equations \cite{berthon2012local}. 
It is also important to note that the specific entropy $S = \log ( {p}{\rho^{-\Gamma}} )$ is well-defined only when the density $\rho$ and pressure $p$ are positive. Therefore, to numerically preserve the MEP \cite{zhang2012minimum}, it is essential to first ensure the positivity of $\rho$ and $p$ throughout the computation. Numerical methods designed to preserve such positivity have been extensively studied; see, e.g., \cite{zhang2010, zhang2010b,  Xu2014,WuTang2015,  QinShu2016, WuTang2017ApJS,guermond2018second}.

The above efforts on the MEP and entropy-preserving schemes have mainly focused on (non-relativistic) hydrodynamic equations. However, in many important astrophysical and high-energy physical phenomena, the flow of matter and energy occurs at speeds close to the speed of light or in strong gravitational fields, necessitating the incorporation of relativistic effects. In such cases, the governing equations take the relativistic form of the Euler system:
\begin{equation}\label{eq:covariant}
	\begin{cases}
		\partial_{\alpha}(\rho u^{\alpha}) = 0, \\ 
		\partial_{\alpha} T^{\alpha\beta} = 0,
	\end{cases}
\end{equation}
where $\rho$ denotes the rest-mass density, $u^{\alpha}$ is the four-velocity (with Einstein summation convention over $\alpha$), and $T^{\alpha\beta}$ is the stress-energy tensor. For an ideal fluid, the stress-energy tensor is expressed as
$T^{\alpha\beta} = \rho h u^{\alpha} u^{\beta} + p g^{\alpha\beta},$ 
where $p$ is the pressure, $h = 1 + e + \frac{p}{\rho}$ is the specific enthalpy, and $e$ is the specific internal energy. We adopt the geometrized unit system where the speed of light $c = 1$. 
In special relativity, the spacetime metric $(g^{\alpha\beta})_{4 \times 4}$ is the Minkowski tensor $\text{diag}\{-1, 1, 1, 1\}$. In the $d$-dimensional case, the special relativistic Euler system can then be rewritten in the conservative form:
\begin{equation}\label{eq:RHD3D}
	\frac{\partial \vec{U}}{\partial t} + \sum_{i=1}^{d} \frac{\partial \vec{F}_i(\vec{U})}{\partial x_i} = \mathbf{0},
\end{equation}
where the conservative vector $\textbf{U}$ and flux function $\textbf{F}_i$ are defined as
\begin{align} 
	&\mathbf{U} = 
	\left( D, {\bm m}^\top, E \right)^\top 
	= 
	\left( \rho \gamma, \rho h \gamma^2 {\bm v}^\top, \rho h \gamma^2 - p \right)^\top, 
	\label{eq:86} \\ 
	&\mathbf{F}_i = 
	\left( D v_i, v_i {\bm m}^\top + p {\bf e}_i^\top, m_i \right)^\top 
	= 
	\left( \rho \gamma v_i, \rho h \gamma^2 v_i {\bm v}^\top + p {\bf e}_i^\top, \rho h \gamma^2 v_i \right)^\top.
	\label{eq:93} 
\end{align}
Here, $D$ is the mass density, ${\bm m}$ is the momentum density, and $E$ is the energy density, referred to as the {\em conservative quantities}. The rest-mass density $\rho$, velocity ${\bm v} = (v_1, v_2, \dots, v_d)^\top$, and pressure $p$ are often termed the {\em primitive quantities}. Additionally, the Lorentz factor 
$\gamma = (1 - |{\bm v}|_2^2)^{-\frac12},$   
and ${\bf e}_i$ represents the $i$-th column of the $d \times d$ identity matrix.


However, in the relativistic case, the understanding of the entropy principle remains rather limited. The primary challenge arises from the strong nonlinearity introduced by relativistic effects, which complicates the relationship between entropy and conserved quantities to the point of making it impossible to explicitly characterize. Specifically, the function $S(\mathbf{U})$ is strongly nonlinear and can only be defined implicitly, requiring the solution of a complex nonlinear equation. This represents a notable departure from the non-relativistic case, posing substantial difficulty to the study of the MEP and the development of entropy-preserving schemes.
\begin{center} {\em Does a similar MEP hold in the relativistic case?} \end{center}
Until 2021, this question remained unanswered. In \cite{WuMEP2021}, the MEP was first proven for the special relativistic Euler equations \eqref{eq:RHD3D} coupled with a simple ideal equation of state (EOS), where the numerical preservation of the global MEP was also explored. However, the ideal EOS is often a poor approximation for many relativistic flows due to its incompatibility with relativistic kinetic theory \cite{taub1948relativistic}. Therefore, it is important to explore the MEP for the relativistic Euler equations with more physically accurate and practical EOSs and locally entropy-preserving schemes that can control spurious oscillations.

The aim of this paper is to explore the MEP for the relativistic Euler equations with a general EOS, which includes a broad range of widely used EOSs. Additionally, we address, for the first time, the challenge of preserving the local version of the discovered MEP in high-order schemes for relativistic hydrodynamics.  
The contributions and innovations of this paper are outlined as follows:
\begin{itemize}
	
	\item We find a family of entropy pairs for the relativistic Euler equations with a general EOS. Rigorous analysis is presented to prove that the entropy is strictly convex under a sufficient and necessary condition. Based on these entropy pairs, we theoretically establish both global and local MEPs for the relativistic Euler system with a general EOS. Unlike the non-relativistic case, there are no explicit formulas for either the flux ${\bf F}_i$ or the primitive quantities (and the specific entropy $S$) in terms of ${\bf U}$, making our entropy analysis highly nontrivial. Furthermore, the general and abstract form of the EOS introduces greater challenges compared to the case of a simple EOS, such as the ideal EOS.
	
	\item We present a rigorous framework for developing provably entropy-preserving high-order schemes for the relativistic Euler equations. Relativistic effects present essential difficulties in this study that are not encountered in the non-relativistic case or with the simple ideal EOS. Specifically, in special relativity, the specific entropy $S$ is a highly nonlinear implicit function of the conservative variables $\mathbf{U}$, which poses significant challenges for proving entropy preservation. To overcome these difficulties, we establish a series of auxiliary theories in Section \ref{sec:3} using technical estimates. We also apply the geometric quasi-linearization (GQL) technique \cite{wu2023geometric} to convert the otherwise unmanageable nonlinear constraints into linear ones by introducing extra free parameters. These methodologies lay the foundation for our entropy-preserving analysis.

	\item We propose novel, robust, locally entropy-preserving high-order frameworks. The critical challenge is properly estimating the local minimum of entropy, which is particularly difficult due to the potential occurrence of shock waves at unknown locations. To address this, we propose two new methods for estimating local lower bounds of the specific entropy. Our methods prove effective for both smooth and discontinuous solutions. Numerical experiments show that these approaches maintain high-order accuracy while effectively controlling spurious oscillations, outperforming existing local entropy minimum estimation techniques, such as those in \cite{LV2015715, ching2024positivity, ching2024positivity2,ching2025positivity}. Notably, our methods are not limited to relativistic Euler equations but are versatile enough to be applied to other hydrodynamic models that admit the MEP. Furthermore, our approach offers fresh insights into suppressing spurious oscillations through local MEP preservation.

	\item We implement our locally entropy-preserving high-order DG schemes coupled with high-order strong-stability-preserving multi-step (SSP MS) discretization. Several one-dimensional (1D) and two-dimensional (2D) numerical experiments, using various EOSs, demonstrate the importance of preserving the discovered MEP in numerical simulations and validate the accuracy and effectiveness of the proposed schemes. We also conduct comparisons with existing entropy minimum estimation techniques to highlight the advantages of our proposed methods.
\end{itemize}

This paper is organized as follows: Section \ref{sec:2} establishes both local and global MEP for the relativistic Euler equations with a general EOS. Section \ref{sec:3} develops the auxiliary theories for the rigorous entropy-preserving analysis. Sections \ref{sec:1Dscheme} and \ref{sec:2Dschemes} present provably entropy-preserving high-order finite volume and DG schemes on 1D meshes and multidimensional unstructured meshes, respectively. Section \ref{sec:6} introduces two novel methods for estimating local bounds of the specific entropy. Finally, Section \ref{section:5} provides several numerical examples to demonstrate the importance of preserving the discovered MEP and showcases the effectiveness and robustness of the proposed high-order entropy-preserving schemes.

\section{Minimum Entropy Principle} \label{sec:2}

This section proves the MEP for the relativistic Euler equations with a general EOS. In subsequent \Cref{sec:3,sec:1Dscheme,sec:2Dschemes,sec:6}, we will explore how to incorporate the discovered MEP into the design of robust high-order numerical schemes.

\subsection{General Equation of State}

To close the relativistic Euler system \eqref{eq:RHD3D}, an EOS is required, which relates the enthalpy $h$ to the pressure $p$ and rest-mass density $\rho$. 
The relativistic EOS for a single-component perfect gas is given by
\begin{equation}\label{eq:PEOS}
	h(\theta) = \frac{ K_3(1/\theta) }{K_2(1/\theta)},
\end{equation}
where $\theta := p/\rho$, and $K_2$ and $K_3$ are the modified Bessel functions of the second kind, of orders two and three, respectively. Due to the presence of these complicated Bessel functions, the EOS \eqref{eq:PEOS} is computationally expensive and, therefore, rarely used in practice.

As alternatives, several approximate EOSs have been proposed. Ryu, Chattopadhyay, and Choi \cite{ryu2006equation} introduced the following RC-EOS\footnote{\label{fn:EOS}These abbreviations follow the conventions in \cite{mignone2005piecewise,ryu2006equation}.}:
\begin{equation}\label{hEOS1}
	h(\theta) = \frac{2(6\theta^2 + 4\theta + 1)}{3\theta + 2}.
\end{equation}
Sokolov, Zhang, and Sakai \cite{sokolov2001simple} proposed the IP-EOS\footnoteref{fn:EOS}:
\begin{equation}\label{hEOS2}
	h(\theta) = 2\theta + \sqrt{1 + 4\theta^2}.
\end{equation}
Mathews \cite{mathews1971hydromagnetic} suggested the TM-EOS\footnoteref{fn:EOS}:
\begin{equation}\label{hEOS3}
	h(\theta) = \frac{5}{2}\theta + \sqrt{1 + \frac{9}{4}\theta^2},
\end{equation}
which was later used in numerical simulations by Mignone, Plewa, and Bodo \cite{mignone2005piecewise}. In addition to the above approximate relativistic EOSs, 
another special EOS is the ideal EOS:
\begin{equation}\label{ID-EOS}
	h(\theta) = 1 + \frac{\Gamma}{\Gamma -1} \theta
	\qquad \textrm{with} \qquad
	\Gamma \in (1,2],
\end{equation}
which is commonly used in non-relativistic fluid dynamics and has also been borrowed to the study of relativistic flows. However, for most relativistic hydrodynamic problems, the ideal EOS \eqref{ID-EOS} is a poor approximation due to its inconsistency with relativistic kinetic theory \cite{taub1948relativistic}. Moreover, when the adiabatic index $\Gamma > 2$, the ideal EOS \eqref{ID-EOS} can result in superluminal wave propagation, violating the fundamental principles (causality) of special relativity.

To ensure the relativistic causality, specifically that the local sound speed $c_s < 1$, it has been shown in \cite{WuTang2017ApJS,xu2024high} that the function $h(\theta)$ must satisfy the following inequality:
\begin{equation}\label{eq:122}
	h'(\theta) > \frac{h(\theta)}{h(\theta) - \theta} \qquad \forall \, \theta \in (0, + \infty),
\end{equation}
or equivalently, the function $e(\theta) := h(\theta) - \theta - 1$ must satisfy:
\begin{equation}\label{eq:206}
	e'(\theta) > \frac{\theta}{e(\theta) + 1} \qquad \forall \, \theta \in (0, + \infty).
\end{equation}
In the remainder of this paper, we will consider a general class of EOSs of the form:
\begin{equation}\label{eq:175}
	h = h(\theta) = e(\theta) + \theta + 1,
\end{equation}
where the functions $h(\theta)$ and $e(\theta)$ satisfy the conditions \eqref{eq:122}--\eqref{eq:206}. Following \cite{xu2024high}, we refer to such general EOSs as {\em Synge-type} EOSs. It can be verified that all the aforementioned EOSs \eqref{eq:PEOS}--\eqref{ID-EOS} belong to this class.

\subsection{Convex Entropy Pairs}

Given any smooth function $\mathcal{H}: \mathbb{R}_+ \mapsto \mathbb{R}$, we define 
\begin{equation}\label{eq:409}
	\SE(\mathbf{U}) := - D \mathcal{H}(S(\mathbf{U})), 
	\qquad
	\SF_i(\mathbf{U}) := - D v_i \mathcal{H}(S(\mathbf{U})),
	\quad i = 1, \dots, d,
\end{equation}
where $S(\mathbf{U})$ is defined as
\begin{equation}\label{eq:405}
	S(\mathbf{U}) := -\ln{\rho} + \int_1^{\theta} \frac{e'(\xi)}{\xi} \, \mathrm{d}\xi.
\end{equation}
For example, for the ideal EOS \eqref{ID-EOS}, RC-EOS \eqref{hEOS1}, IP-EOS \eqref{hEOS2}, TM-EOS \eqref{hEOS3}, the specific forms of the corresponding $S(\mathbf{U})$ are respectively given as follows: 
		\begin{align*} 
	S_{\rm ID}&=-\ln{\rho}+\frac{1}{\Gamma-1}\int_1^\theta\frac{1}{\xi}\mathrm{d}\xi=-\ln{\rho}+\frac{1}{\Gamma-1}\ln{\theta}=\frac{1}{\Gamma-1}\ln{\frac{p}{\rho^\Gamma}}.
		\\
	S_{\rm RC}&=-\ln{\rho}+\int_1^\theta \left(\frac{3}{2\xi}+\frac{9}{2(3\xi+2)}+\frac{9}{(3\xi+2)^2}\right)\mathrm{d}\xi\\
	&=-\ln{\rho}+\frac{3}{2}\ln{\theta}+\frac{3}{2}\ln{(3\theta+2)}-\frac{3}{3\theta+2}+\frac{3}{5}-\frac{3}{2}\ln{5}.
		\\
	S_{\rm IP}&=-\ln{\rho}+\int_1^\theta \left(\frac{1}{\xi}+\frac{4}{\sqrt{1+4\xi^2}}\right)\mathrm{d}\xi\\
	&=-\ln{\rho}+\ln{\theta}+2\ln{\left(2\theta+\sqrt{1+4\theta^2}\right)}-2\ln{(2+\sqrt{5})}.
		\\
	{S}_{\rm TM}&=-\ln{\rho}+\int_1^\theta \left(\frac{3}{2\xi}+\frac{\frac{9}{4}\xi}{\sqrt{1+\frac{9}{4}\xi^2}}\right)\mathrm{d}\xi
	\\
	&=-\ln{\rho}+\frac{3}{2}\ln{\theta}+\frac{3}{2}\ln{\left(\frac{3}{2}\theta+ \sqrt{1+\frac{9}{4}\theta^2}\right)}-\frac{3}{2}\ln{\left(\frac{3}{2}+\sqrt{\frac{13}{4}}\right)}.
	\end{align*}

As proven below, the pair $(\SE, \SF)$ defined in \eqref{eq:409} forms an entropy pair for the relativistic Euler system if  the function $\mathcal{H}(\cdot)$ satisfies a sufficient and necessary condition. By appropriately choosing $\mathcal{H}$, we will construct a sequence of entropy pairs, which lays the foundation for establishing the MEP for the relativistic Euler system \eqref{eq:RHD3D}.


\begin{lemma}\label{lem:convexentropy}
	For a general EOS satisfying condition \eqref{eq:122} and $h(\theta) - \theta h'(\theta) \neq 0$, the function $\mathcal{E}(\mathbf{U})$ defined in \eqref{eq:409} is strictly convex if and only if the following conditions hold:
	\begin{equation} \label{eq:410}
		\mathcal{H}'(S(\mathbf{U})) > 0, \quad 
		\mathcal{H}'(S(\mathbf{U})) - (1 + e'(\theta)) \mathcal{H}''(S(\mathbf{U})) > 0 \quad \forall\ \mathbf{U} \in \mathcal{G},
	\end{equation}
	where $\mathcal{G}$ denotes the physically admissible state set defined as
	\begin{equation}\label{eq:DefG}
		\mathcal{G} = \left\{ \mathbf{U} = ( D, {\bm m}^\top, E )^\top : \rho(\mathbf{U}) > 0, p(\mathbf{U}) > 0, | {\bm v}(\mathbf{U}) | < 1 \right\}.
	\end{equation}
\end{lemma}

\begin{proof}
	The strict convexity of the function $\mathcal{E}(\mathbf{U})$ is equivalent to the positive definiteness of its Hessian matrix:
	\begin{align}\label{eta_UU}
		\mathcal{E}_{\BU\BU} 
		&= 
		-\CH'(S) 
		(
		\Be_1 S_{\BU}^\top + S_{\BU}\Be_1^\top + DS_{\BU\BU}
		)
		-D \CH''(S) S_{\BU} S_{\BU}^\top
		\\ \nonumber
		&=-\CH'(S)\ \BA_1+\frac{D}{1+e'(\theta)}
		\left(
		\CH'(S)-(1+e'(\theta))\CH''(S)
		\right)
		S_{\BU}S_{\BU}^\top, 
	\end{align}
	where 
	$$
	\BA_1
	:=
	\Be_1 S_{\BU}^\top
	+
	S_{\BU}\Be_1^\top
	+
	D S_{\BU\BU}
	+
	\frac{D}{1+e'(\theta)} S_{\BU} S_{\BU}^\top.
	$$
	Since $S({\bf U})$ does not have explicit formulas,  to express $S_{\mathbf{U}}^\top$ and $S_{\mathbf{UU}}$, we apply the chain rule
	\begin{equation*}
		S_{\BU}^\top 
		= 
		\left(\frac{\partial S}{\partial \BV}\right)^\top
		\frac{\partial\BV}{\partial \BU},
		\qquad
		S_{\BU\BU}
		=
		\frac{\partial^2 S}{\partial \BU^2} = \frac{\partial^2 S}{\partial \BV\partial\BU}\frac{\partial\BV}{\partial \BU}=S_{\mathbf{UV}}\frac{\partial\BV}{\partial \BU},
	\end{equation*}
	where $\mathbf{V} := ( \rho, \bm{v}^\top, p )^\top$ represents the vector of primitive variables.
	Using results from \cite{xu2024high}, the matrix $S_{\mathbf{UV}}$ is explicitly given by
	\begin{equation*}
		S_{\mathbf{UV}}=\begin{pmatrix}
			\frac{1+e'(\theta)}{\rho^2\gamma} & \frac{h\gamma}{\rho\theta}\Bv^\top & \frac{h-\theta(1+e'(\theta))}{\rho^2\theta^2\gamma} \\
			\mathbf{0}_d & -\frac{1}{\rho\theta}\mathbf{I}_d & \frac{1}{\rho^2\theta^2}\Bv\\
			0 & \mathbf{0}_d^\top & -\frac{1}{\rho^2\theta^2}
		\end{pmatrix},
	\end{equation*}
	where $\mathbf{I}_d$ denotes the $d \times d$ identity matrix and $\mathbf{0}_{d}$ is the $d \times 1$ zero vector. 
	Substituting these expressions, the matrix $\mathbf{A}_1$ becomes
	\begin{align*}
		\BA_1 &=\mathbf{e}_1S_{\BU}^\top+S_{\BU}\mathbf{e}_1^\top+DS_{\BU\BV}\frac{\partial \BV}{\partial \BU}+\frac{D}{1+e'(\theta)}S_{\BU}S_{\BU}^\top
		\\
		&= \frac{-1}{\rho h\theta^2e'(\theta) \left(1+e'(\theta)\right)\ \left(1 - c_s^2|\Bv|^2\right)}
		\begin{pmatrix}
			a_1 & a_2\Bv^\top & a_3 \\ 
			a_2\Bv & \BA_2 & a_4\Bv \\
			a_3 & a_4\Bv^\top & a_5 
		\end{pmatrix},
	\end{align*}
	where   the sound speed 
	\begin{equation*}
		c_s =  \sqrt{\frac{\theta(1+e'(\theta))}{h(\theta) \,e'(\theta)}}
	\end{equation*}
	satisfies $c_s \in (0,1)$ under condition \eqref{eq:122}. 
	Additionally, 
	\begin{equation} \label{Delta_theta}
		\begin{aligned} 
			a_1 &:= h\gamma^{-1} \Delta_{\theta}^2, 
			&\quad
			a_2 &:= \Big(h+\theta(1+e'(\theta))\Big)\Delta_{\theta},
			\\
			a_3 &:= -\Big(h+\theta(1+e'(\theta))|\Bv|^2\Big)\Delta_{\theta},
			&\quad 
			a_4 &:= -\gamma\Big(h+\theta(1+e'(\theta))|\Bv|^2+\theta(1+e'(\theta))^2\Big),
			\\
			a_5 &:= \gamma\Big(h+\theta(1+e'(\theta))(2+e'(\theta))|\Bv|^2\Big),
			& \quad
			\Delta_{\theta} &:= h-\theta(1+e'(\theta)) = h(\theta)-\theta h'(\theta ),
		\end{aligned}
	\end{equation}
	and 
	\begin{equation*}
		\begin{aligned}
			\BA_2
			:=
			& \frac{\theta(1+e'(\theta))}{\gamma}
			\left(
			e'(\theta)-\frac{\theta|\Bv|^2(1+e'(\theta))}{h}
			\right)
			\mathbf{I}_d \\
			& +
			\left(
			\frac{\theta^2(1+e'(\theta))^2}{h\gamma}+\gamma\Big(h+\theta(1+e'(\theta))(2+e'(\theta))\Big)
			\right)
			\Bv\Bv^\top.
		\end{aligned}
	\end{equation*}
	Noting $a_1>0$, we  
	define an invertible matrix
	\begin{equation*}
		{\bf P}_1:= 
		\begin{pmatrix}
			1 & \mathbf{0}_d^\top & 0 \\
			-\frac{a_2}{a_1}\Bv & \mathbf{I}_d & \mathbf{0}_d \\
			-\frac{a_3}{a_1} & \mathbf{0}_d^\top & 1 
		\end{pmatrix}.
	\end{equation*}
	Using this transformation, we have
	\begin{equation} \label{P1A1P1}
		{\bf P}_1 \BA_1 {\bf P}_1^\top 
		=  
		-\frac{1}
		{
			\rho h\theta^2e'(\theta) \left(1+e'(\theta)\right)
			\left(1 - c_s^2|\Bv|^2\right)
		}
		\begin{pmatrix}
			a_1 & \mathbf{0}_{d+1}^\top \\
			\mathbf{0}_{d+1} & a_6\BA_3
		\end{pmatrix}
	\end{equation}
	with
	\begin{equation*}
		a_6 := \gamma\theta e'(\theta)(1+e'(\theta))(1-c_s^2|\Bv|^2)>0\quad {\rm and} \quad
		\BA_3:=  \begin{pmatrix}
			\left(1-|\Bv|^2\right)\mathbf{I}_d+\Bv\Bv^\top & -\Bv \\
			-\Bv^\top & |\Bv|^2 
		\end{pmatrix}.
	\end{equation*}
	Note that 
	\begin{equation} \label{P1SU}
		\mathbf{P}_1S_{\BU} = 
		\begin{pmatrix}
			-\frac{h\gamma^{-1}}{\rho\theta}, & 
			\frac{2(1+e'(\theta))\Bv^\top}{\rho\Delta_{\theta}}, &
			-\frac{(1+e'(\theta))(1+|\Bv|^2)}{\rho\Delta_{\theta}}
		\end{pmatrix}^\top=:\bm{b}_1.
	\end{equation}
	Combining \eqref{eta_UU}, \eqref{P1A1P1}, and \eqref{P1SU}, we obtain
	\begin{align}
		\nonumber
		\mathbf{P}_1\mathcal{E}_{\BU\BU}\mathbf{P}_1^\top 
		&= \mathbf{P}_1
		\left(
		-\CH'(S)\BA_1+\frac{D}{1+e'(\theta)}
		\left(
		\CH'(S)-(1+e'(\theta))\CH''(S)
		\right)
		S_{\BU}S_{\BU}^\top
		\right)
		\mathbf{P}_1^\top
		\\ \nonumber
		&=
		-\CH'(S) \mathbf{P}_1\BA_1\mathbf{P}_1^\top+\frac{D}{1+e'(\theta)}
		\left(
		\CH'(S)-(1+e'(\theta))\CH''(S)
		\right)
		\bm{b}_1\bm{b}_1^\top
		\\ \label{P1EUUP1}
		&=
		a_7\CH'(S)\BA_4+a_8
		\left(
		\CH'(S)-(1+e'(\theta))\CH''(S)
		\right)
		\bm{b}_1\bm{b}_1^\top,
	\end{align}
	where $a_7:= (\rho h\theta^2e'(\theta) \left(1+e'(\theta)\right)\ \left(1 - c_s^2|\Bv|^2\right))^{-1}>0$, $a_8 := \frac{D}{1+e'(\theta)}>0$, and
	\begin{equation} \label{A_4}
		\BA_4:=
		\begin{pmatrix}
			a_1 & \mathbf{0}_{d+1}^\top \\
			\mathbf{0}_{d+1} & a_6\BA_3
		\end{pmatrix}.
	\end{equation}
	We proceed to show that $\BA_3$ is positive semi-definite. Consider the transformation matrix
	\begin{equation*}
		{\bf P}_2:= 
		\begin{pmatrix}
			\mathbf{I}_d & \frac{\Bv}{|\Bv|^2} \\
			\mathbf{0}_d^\top & 1 
		\end{pmatrix}.
	\end{equation*}
	Applying ${\bf P}_2$ to $\BA_3$, we obtain
	\begin{equation*}
		\mathbf{P}_2\BA_3  \mathbf{P}_2^\top 
		= 
		\begin{pmatrix}
			\left(1-|\Bv|^2\right)\mathbf{I}_d
			+
			\left(1-\frac{1}{|\Bv|^2}\right)\Bv\Bv^\top 
			& \mathbf{0}_d
			\\
			\mathbf{0}_d^\top 
			& |\Bv|^2
		\end{pmatrix}.
	\end{equation*}
	The eigenvalues of the block matrix
	\begin{equation*}
		\mathbf{B}
		:=
		\left(1-|\Bv|^2\right)\mathbf{I}_d+\left(1-\frac{1}{|\Bv|^2}\right)\Bv\Bv^\top
	\end{equation*}
	are
	\begin{align*}
		\lambda_{\mathbf{B}}^{(1)} = 0,\qquad
		\lambda_{\mathbf{B}}^{(2)}=\cdots=\lambda_{\mathbf{B}}^{(d)}=1-|\Bv|^2,\quad \mbox{if}~d\ge 2.
	\end{align*}
	Thus, $\BA_3$ is positive semi-definite, and its rank is ${\rm rank}(\BA_3)=d$. Therefore, $\BA_3$ can be diagonalized using an orthogonal matrix $\BV_3$, such that  $\BA_3=\BV_3\mathbf{\Lambda}_3\BV_3^\top$ with $\mathbf{\Lambda}_3={\rm diag}\{\lambda_0=0,\lambda_1>0,\ldots, \lambda_d>0\}.$ 
	In particular, the orthonormal eigenvector of $\mathbf{A}_3$ corresponding to $\lambda_0=0$ is given by 
	\begin{equation} \label{0vec}
		\mathbf{V}_3^{(0)}=\frac{1}{\sqrt{1+|\Bv|^2}}\left(\Bv^\top,1\right)^\top.
	\end{equation}
	From \eqref{A_4},  $\BA_4$  can be expressed as
	\begin{equation} \label{A4eigdecomp}
		\BA_4=a_6
		\begin{pmatrix}
			\frac{a_1}{a_6} & \mathbf{0}_{d+1}^\top\\
			\mathbf{0}_{d+1} & \BA_3
		\end{pmatrix}=a_6
		\begin{pmatrix}
			1 & \mathbf{0}_{d+1}^\top\\
			\mathbf{0}_{d+1} & \BV_3
		\end{pmatrix}
		\begin{pmatrix}
			\frac{a_1}{a_6} & \mathbf{0}_{d+1}^\top\\
			\mathbf{0}_{d+1} & \mathbf{\Lambda}_3
		\end{pmatrix}
		\begin{pmatrix}
			1 & \mathbf{0}_{d+1}^\top\\
			\mathbf{0}_{d+1} & \BV_3
		\end{pmatrix}^\top=a_6\BV_4\mathbf{\Lambda}_4\BV_4^\top,
	\end{equation}
	where $\mathbf{\Lambda}_4 = \text{diag}\{\frac{a_1}{a_6}, \lambda_0, \lambda_1, \ldots, \lambda_d\}$. Since $a_1 > 0$ and $a_6 > 0$, it follows that $\mathbf{A}_4$ is positive semi-definite.
	
	Since $\mathcal{E}_{\mathbf{UU}}$ and $\mathbf{P}_1 \mathcal{E}_{\mathbf{UU}} \mathbf{P}_1^\top$ are congruent, it suffices to show that $\mathbf{P}_1 \mathcal{E}_{\mathbf{UU}} \mathbf{P}_1^\top$ is positive definite if and only if $\mathcal{H}(S)$ satisfies \eqref{eq:410}.

	\paragraph{Sufficiency} Assume that $\mathcal{H}(S)$ satisfies \eqref{eq:410}. For any $\bm{z} \in \mathbb{R}^{d+2}$, from \eqref{P1EUUP1}, we have
	\begin{equation*}
		\bm{z}^\top \mathbf{P}_1 \mathcal{E}_{\mathbf{UU}} \mathbf{P}_1^\top \bm{z} 
		= 
		a_7 \mathcal{H}'(S) \, \bm{z}^\top \mathbf{A}_4 \bm{z} 
		+ 
		a_8 \big(\mathcal{H}'(S) - (1 + e'(\theta)) \mathcal{H}''(S)\big) \, |\bm{b}_1^\top \bm{z}|^2 \ge 0,
	\end{equation*}
	where the last inequality follows from that 
	$\mathbf{A}_4$ and $\bm{b}_1 \bm{b}_1^\top$ are both positive semi-definite. 
	If $\bm{z}^\top\mathbf{P}_1\mathcal{E}_{\BU\BU}\mathbf{P}_1^\top \bm{z}=0$, since $a_7>0$, $a_8>0$, and $\CH(S)$ satisfies \eqref{eq:410}, we obtain
	\begin{equation} \label{suffcond}
		\bm{z}^\top\BA_4\bm{z} = 0, \qquad
		\bm{b}_1^\top\bm{z} = 0.
	\end{equation}
	Let $\bm{z}:=\left(z^{(1)},(\bm{z}^{(2)})^\top\right)^\top$ with $\bm{z}^{(2)}\in\mathbb{R}^{d+1}$. Using \eqref{A4eigdecomp}, we have
	\begin{equation*}
		0=\bm{z}^\top\mathbf{A}_4\bm{z}=a_6\left(\mathbf{V}_4^\top\bm{z}\right)^\top\mathbf{\Lambda}_4\left(\mathbf{V}_4^\top\bm{z}\right)=a_6\left(\frac{a_1}{a_6}\left|z^{(1)}\right|^2+\sum\limits_{k=0}^d\lambda_k\left|\left({\mathbf{V}_3^\top}\bm{z}^{(2)}\right)_k\right|^2\right).
	\end{equation*}
	This implies that $\mathbf{V}_4^\top\bm{z}=\mathbf{0}$, except for the possibility that $\left({\mathbf{V}_3^\top}\bm{z}^{(2)}\right)_0$ may not be zero for $\lambda_0=0$. Suppose that $\left(\mathbf{V}_3^\top\bm{z}^{(2)}\right)_0=c\neq0.$ Then, we have
	\begin{equation} \label{z2}
		\bm{z}^{(2)}=\mathbf{V}_3^\top\left(c,\mathbf{0}_d^\top\right)^\top=c\mathbf{V}_3^{(0)}.
	\end{equation}
	Additionally, the condition $\bm{b}_1^\top\bm{z} = 0$ combined with \eqref{P1SU} gives
	\begin{equation} \label{b1z2=0}
		\begin{pmatrix}
			2\bm{v}^\top, & -(1 + |\bm{v}|^2)
		\end{pmatrix}
		\cdot
		\bm{z}^{(2)} = 0.
	\end{equation}
	Substituting \eqref{z2} and \eqref{0vec} into \eqref{b1z2=0}, we perform a direct calculation:
	\[
	\begin{pmatrix}
		2\bm{v}^\top, & -(1 + |\bm{v}|^2)
	\end{pmatrix}
	\cdot
	\bm{v}_3^{(0)} = \frac{1}{\sqrt{1 + |\bm{v}|^2}}
	\begin{pmatrix}
		2\bm{v}^\top, & -(1 + |\bm{z}|^2)
	\end{pmatrix}
	\cdot
	\begin{pmatrix}
		\bm{z} \\ 1
	\end{pmatrix}.
	\]
	Simplifying the above, we find:
	\[
	\frac{1}{\sqrt{1 + |\bm{z}|^2}} \left( 2|\bm{z}|^2 - (1 + |\bm{z}|^2) \right) = \frac{1}{\sqrt{1 + |\bm{z}|^2}} (-1 + |\bm{z}|^2).
	\]
	For $|\bm{z}|^2 < 1$, this results in a nonzero value, contradicting \eqref{b1z2=0} if $c \neq 0$. Therefore, we must have $c = 0$, implying:
	\[
	\left({\mathbf{V}_3^\top} \bm{z}^{(2)}\right)_0 = 0.
	\]
	From \eqref{suffcond}, it follows that ${\mathbf{V}_4^\top} \bm{z} = \mathbf{0}$. Since $\mathbf{V}_4$ is orthogonal, this implies $\bm{z} = \mathbf{0}$. 
	Thus, the matrix $\mathbf{P}_1 \mathcal{E}_{\mathbf{UU}} \mathbf{P}_1^\top$ is positive definite, completing the proof for sufficiency.

	\paragraph{Necessity} 	Next, we assume that $\mathbf{P}_1\mathcal{E}_{\BU\BU} \mathbf{P}_1^\top$ is positive definite. Then,  all $\bm{z}\in \mathbb{R}^{d+2}\backslash\{\boldsymbol{0}\}$, we have
	\begin{equation} \label{necessary:spd}
		\bm{z}^\top \mathbf{P}_1\mathcal{E}_{\BU\BU}\mathbf{P}_1^\top \bm{z} = -\CH'(S)\bm{z}^\top \BA_4 \bm{z}+\frac{D}{1+e'(\theta)}\bigg(\CH'(S)-(1+e'(\theta))\CH''(S)\bigg)\ \left|\bm{b}_1^\top \bm{z}\right|^2 > 0. 
	\end{equation}
	Since $\BA_4$ does not have full rank, there exists a nonzero vector $\bm{z}_1\in \mathbb{R}^{d+2}\backslash\{\boldsymbol{0}\}$ such that $\BA_4\bm{z}_1=\boldsymbol{0}$. Additionally, we can find another nonzero vector $\bm{z}_2\in \mathbb{R}^{d+2}\backslash\{\boldsymbol{0}\}$ satisfying $\boldsymbol{b}_1^\top \bm{z}_2=0$. Because $\mathbf{P}_1\mathcal{E}_{\BU\BU} \mathbf{P}_1^\top$ is positive definite, it follows that $\bm{z}_1\neq \bm{z}_2$.  We observe from \eqref{necessary:spd} that 
	\begin{equation*}
		\begin{aligned}
			0 &< \bm{z}_1^\top \mathbf{P}_1\mathcal{E}_{\BU\BU}\mathbf{P}_1^\top \bm{z}_1 = \frac{D}{1+e'(\theta)}\bigg(\CH'(S)-(1+e'(\theta))\CH''(S)\bigg)\ \left|\bm{b}_1^\top \bm{z}_1\right|^2 \\
			0 &< \bm{z}_2^\top \mathbf{P}_1\mathcal{E}_{\BU\BU}\mathbf{P}_1^\top \bm{z}_2= -\CH'(S)\bm{z}_2^\top \BA_4 \bm{z}_2,
		\end{aligned}
	\end{equation*}
	from which, we deduce that the conditions in \eqref{eq:410} are satisfied. Thus, the proof of necessity is complete. 
\end{proof}

\begin{lemma} \label{thm:400}
	For any function $\mathcal{H}$ satisfying \eqref{eq:410},  $(\SE,\bm{\SF})$ defined in \eqref{eq:409} forms a convex entropy pair for the relativistic Euler system with a general EOS of the form \eqref{eq:175}.
\end{lemma}

\begin{proof}

	We only present the proof for the one-dimensional case ($d = 1$); the extension to higher spatial dimensions follows similarly. 
	Assume the solution is smooth, and we aim to derive an additional conservation law. 
	By combining the momentum and energy equations of the 1D relativistic Euler equations, we obtain
	\begin{equation} \label{mEeq}
		\rho h\gamma^2\frac{\partial v_1}{\partial t} + \rho h\gamma^2v_1\frac{\partial v_1}{\partial x_1}+v_1\frac{\partial p}{\partial t}+\frac{\partial p}{\partial x_1} = 0.
	\end{equation}
	where $\gamma = \frac{1}{\sqrt{1-v_1^2}}$ is the Lorentz factor. Using the definition of $\gamma$, we have
	\begin{equation} \label{Lorentzeq}
		\frac{\partial \gamma}{\partial t} = \gamma^3 v_1 \frac{\partial v_1}{\partial t}, 
		\qquad 
		\frac{\partial \gamma}{\partial x_1} = \gamma^3 v_1 \frac{\partial v_1}{\partial x_1}.
	\end{equation}
	Multiplying \eqref{mEeq} by $v_1$ and using \eqref{Lorentzeq}, we obtain 
	\begin{equation} \label{LM23}
		\frac{1}{\gamma^2}\frac{\partial p}{\partial t} = \frac{\partial p}{\partial t}+v_1\frac{\partial p}{\partial x_1}+\frac{\rho h}{\gamma}\left(\frac{\partial \gamma}{\partial t}+\frac{\partial (\gamma v_1)}{\partial x_1}\right)-\rho h\frac{\partial v_1}{\partial x_1}.
	\end{equation}
	The energy equation implies that
	\begin{align*}
		\frac{\partial p}{\partial t} &= \frac{\partial (\rho h\gamma^2)}{\partial t}+\frac{\partial (\rho h\gamma^2v_1)}{\partial x_1}\\
		&=\rho h\gamma\left(\frac{\partial\gamma}{\partial t}+\frac{\partial (\gamma v_1)}{\partial x_1}\right)+\gamma\frac{\partial (\rho h\gamma)}{\partial t} +\gamma v_1\frac{\partial (\rho h\gamma)}{\partial x_1}.
	\end{align*}
	Using the product rule and the continuity equation, we simplify
	\begin{align*}
		\frac{\partial p}{\partial t}&=\big(\rho\gamma(1+e)+p\gamma\big)\left(\frac{\partial \gamma}{\partial t}+\frac{\partial (\gamma v_1)}{\partial x_1}\right)+\gamma\left(-\frac{\partial (\rho\gamma v_1)}{\partial x_1}(1+e)+\rho\gamma\frac{\partial e}{\partial t}+\frac{\partial (p\gamma)}{\partial t}\right)\\
		&\quad+\gamma v_1\left(\frac{\partial (\rho\gamma)}{\partial x_1}(1+e)+\rho\gamma\frac{\partial e}{\partial x_1} + \frac{\partial (p\gamma)}{\partial x_1}\right)\\
		&=\big(\rho\gamma(1+e)+p\gamma\big)\left(\frac{\partial \gamma}{\partial t}+\frac{\partial (\gamma v_1)}{\partial x_1}\right)-\rho\gamma^2(1+e)\frac{\partial v_1}{\partial x_1}+\rho\gamma^2\frac{\partial e}{\partial t}\\
		&\quad+\gamma^2\frac{\partial p}{\partial t}+p\gamma\frac{\partial \gamma}{\partial t}+\rho\gamma^2 v_1\frac{\partial e}{\partial x_1}+\gamma^2v_1\frac{\partial p}{\partial x_1}+p\gamma v_1\frac{\partial \gamma}{\partial x_1}.
	\end{align*}
Rearranging terms, we deduce
\begin{align*}
	\frac{\partial p}{\partial t} 
	&= \rho \gamma (1 + e) \left( \frac{\partial \gamma}{\partial t} 
	+ \frac{\partial (\gamma v_1)}{\partial x_1} \right) 
	+ 2p \gamma \left( \frac{\partial \gamma}{\partial t} 
	+ \frac{\partial (\gamma v_1)}{\partial x_1} \right) \\
	&\quad + \gamma^2 \left( \frac{\partial p}{\partial t} + v_1 \frac{\partial p}{\partial x_1} \right) 
	+ \rho \gamma^2 \left( \frac{\partial e}{\partial t} + v_1 \frac{\partial e}{\partial x_1} \right) 
	- \rho h \gamma^2 \frac{\partial v_1}{\partial x_1}.
\end{align*}
Dividing through by $\gamma^2$ and using \eqref{LM23}, we find
\begin{equation} \label{idtp}
	p \left( \frac{\partial \gamma}{\partial t} 
	+ \frac{\partial (\gamma v_1)}{\partial x_1} \right) 
	+ \rho \left( \frac{\partial e}{\partial t} 
	+ v_1 \frac{\partial e}{\partial x_1} \right) = 0,
\end{equation}
which along with the continuity equation yields
\begin{equation*}
	\rho \left( \frac{\partial e}{\partial t} + v_1 \frac{\partial e}{\partial x_1} \right) 
	= \theta \left( \frac{\partial \rho}{\partial t} + v_1 \frac{\partial \rho}{\partial x_1} \right).
\end{equation*}
By the definition $S := -\ln{\rho} + \int^{\theta} \frac{e'(\xi)}{\xi} ~ \mathrm{d}\xi$, we have
\begin{equation*}
	\frac{\partial S}{\partial t} 
	+ v_1 \frac{\partial S}{\partial x_1} 
	= -\frac{1}{\rho} \left( \frac{\partial \rho}{\partial t} 
	+ v_1 \frac{\partial \rho}{\partial x_1} \right) 
	+ \frac{1}{\theta} \left( \frac{\partial e}{\partial t} 
	+ v_1 \frac{\partial e}{\partial x_1} \right) = 0.
\end{equation*}
For any smooth function $\mathcal{H}(S)$, applying the product rule yields
\begin{equation*}
	\frac{\partial (D \mathcal{H}(S))}{\partial t} 
	+ \frac{\partial (D \mathcal{H}(S) v_1)}{\partial x_1} 
	= \left( \frac{\partial D}{\partial t} 
	+ \frac{\partial (D v_1)}{\partial x_1} \right) \mathcal{H}(S) 
	+ D \mathcal{H}'(S) \left( \frac{\partial S}{\partial t} 
	+ v_1 \frac{\partial S}{\partial x_1} \right) = 0, 
\end{equation*}
which is an additional conservation law for smooth solutions. 
Thus, according to the methodology in \cite{TADMOR1986211, gouasmi2020minimum}, $(\SE,\bm{\SF})$ defined in \eqref{eq:409} forms an entropy pair.
\end{proof}

\subsection{Minimum Principle on Specific Entropy}

We are now in the position to prove the MEP for the relativistic Euler system \eqref{eq:RHD3D} with a general EOS \eqref{eq:175} satisfying \eqref{eq:122}.

\begin{theorem} \label{thm:473}
	Let $\BU(\Bx,t)$ be an entropy solution of the relativistic Euler system \eqref{eq:RHD3D} with a general EOS \eqref{eq:175} satisfying \eqref{eq:122}. Then, the local MEP \eqref{LMEP} holds for the specific entropy $S = S(\BU)$ defined in \eqref{eq:405}.
\end{theorem}

\begin{proof}
	Following \cite[Theorem 4.1]{tadmor1984skew} and \cite[Lemma 3.1]{TADMOR1986211}, 
	for any smooth function $\CH(S)$ satisfying \eqref{eq:410} and any $R>0$, the following entropy inequality holds:
	\begin{equation*}
		\int_{|\Bx|\leq R}D(\Bx,t)\CH(S(\Bx,t)){\rm d}\Bx\geq
		\int_{|\Bx|\leq R+tv_{\rm max}}D(\Bx,0)\CH(S(\Bx,0)){\rm d}\Bx,
	\end{equation*}
	where $v_{\rm max}$ denotes the maximal wave speed in the domain.  
	Now, we follow \cite{TADMOR1986211} and consider a particular function $\CH_0(S)$: 
	\begin{equation*}
		\CH_0(S):=\min\{S-S_0,0\}\quad {\rm with} \quad 
		S_0=\operatorname*{ess~inf}_{|\Bx|\leq R+tv_{\rm max}}S(\Bx,0).
	\end{equation*}
	Although $\CH_0(S)$ is not globally smooth, it can be approximated as the $\epsilon \to 0^+$ limit of a sequence of smooth functions $\{\CH_{\epsilon}(S)\}_{\epsilon > 0}$ 
	defined by
	\begin{equation*}
		\CH_{\epsilon}(S) := \int_{-\infty}^{+\infty} \CH_0(S - s) \phi_{\epsilon}(s) \, {\rm d}s,
	\end{equation*}
	where $\phi_{\epsilon}(s) = \phi\left(s / \epsilon\right) / \epsilon$ with $\phi(s) := \mathrm{e}^{-s^2} / \sqrt{\pi}$. 
	Notice that $\CH_{\epsilon}(S)$ satisfies \eqref{eq:410}, because 
	\begin{equation*}
		\CH_{\epsilon}'(S)=\int_{S-S_0}^{+\infty}\phi_{\epsilon}(s){\rm d}s>0,
		\;
		\CH_{\epsilon}''(S)=-\phi_{\epsilon}(S-S_0)<0,
		\;
		\CH_{\epsilon}'(S)-(1+e'(\theta))\CH_{\epsilon}''(S)>0.
	\end{equation*}
	Therefore, Lemma \ref{thm:400} implies that the function $\mathcal{E}_{\epsilon}(\BU) := -D \CH_{\epsilon}(S)$ is a valid entropy function for the relativistic Euler equations \eqref{eq:RHD3D} with a general EOS \eqref{eq:175}, leading to the entropy inequality:
	\begin{equation*}
		\int_{|\Bx| \leq R} D(\Bx,t) \CH_{\epsilon}(S(\Bx,t)) \, {\rm d}\Bx \geq
		\int_{|\Bx| \leq R + t v_{\rm max}} D(\Bx,0) \CH_{\epsilon}(S(\Bx,0)) \, {\rm d}\Bx \quad \forall R > 0.
	\end{equation*}
	Taking the limit as $\epsilon \to 0^+$ under the integral, we have
	\begin{equation*}
		\int_{|\Bx| \leq R} D(\Bx,t) \CH_0(S(\Bx,t)) \, {\rm d}\Bx \geq
		\int_{|\Bx| \leq R + t v_{\rm max}} D(\Bx,0) \CH_0(S(\Bx,0)) \, {\rm d}\Bx = 0.
	\end{equation*}
	Because $\CH_0(S) \leq 0$, this implies
	\begin{equation*}
		\int_{|\Bx| \leq R} D(\Bx,t) \CH_0(S(\Bx,t)) \, {\rm d}\Bx = 0 \quad \forall R > 0,
	\end{equation*}
	which leads to the local MEP:
	\begin{equation*} 
		S(\Bx,t) \geq
		S_0 = \operatorname*{ess~inf}_{|\Bx| \leq  t v_{\rm max}} S(\Bx,0) \quad \forall t \geq 0.
	\end{equation*}
\end{proof}

\begin{theorem} \label{thm:610}
	Let $\BU(\Bx,t)$ be an entropy solution of the relativistic Euler system \eqref{eq:RHD3D} with a general EOS of the form \eqref{eq:175} satisfying \eqref{eq:122}. Then, the global MEP \eqref{GMEP} holds for the specific entropy $S = S(\BU)$ defined in \eqref{eq:405}.
\end{theorem}

Theorem \ref{thm:610} is a direct consequence of Theorem \ref{thm:473}, and its proof is omitted.

\section{Crucial Inequalities for Entropy-Preserving Analysis}\label{sec:3}

This section establishes several crucial inequalities that will play a central role in the entropy-preserving analysis presented in \Cref{sec:1Dscheme}. It is important to note that the derivation and proof of some inequalities are highly nontrivial. In fact, they represent the most challenging aspect of this work and are key to deriving the provably entropy-preserving schemes. 

Given a function $e(\xi)$, we define an auxiliary function $C_e(\xi): \mathbb{R}_+ \to \mathbb{R}$ by
\begin{equation}\label{eq:500}
	C_e(\xi) := (e(\xi) + 1)^2 - \xi^2 - 1, \quad \xi \in (0, +\infty).
\end{equation}
Conversely, $e(\xi)$ can be expressed in terms of $C_e(\xi)$ as
\begin{equation}\label{eq:504}
	e(\xi) = \sqrt{\xi^2 + 1 + C_e(\xi)} - 1, \quad \xi \in (0, +\infty).
\end{equation}

\begin{lemma}
	For any functions $e(\xi)$ and $C_e(\xi)$ related by \eqref{eq:500} and \eqref{eq:504}, the function $e(\xi)$ satisfies the condition \eqref{eq:206} if and only if
	\[
	C_e'(\xi) > 0 \quad   \forall \, \xi \in (0, +\infty).
	\]
\end{lemma}

\begin{proof}
	For any $\xi > 0$, we compute the derivative of $e(\xi)$:
	\[
	e'(\xi) = \left( \sqrt{\xi^2 + 1 + C_e(\xi)} - 1 \right)' 
	= \frac{\frac{1}{2} C_e'(\xi) + \xi}{\sqrt{C_e(\xi) + \xi^2 + 1}} 
	= \frac{\frac{1}{2} C_e'(\xi) + \xi}{e(\xi) + 1}.
	\]
	Thus, we conclude that
	\[
	C_e'(\xi) > 0 \quad \forall \xi \in \mathbb{R}_+ 
	\quad \Longleftrightarrow \quad e'(\xi) > \frac{\xi}{1 + e(\xi)} \quad \forall \xi \in \mathbb{R}_+.
	\]
\end{proof}

\begin{lemma}
	If $e(\xi)$ satisfies \eqref{eq:206}, then it is strictly increasing on $(0,+\infty)$.
\end{lemma}

\begin{proof}
	For any $\theta \in (0, +\infty)$, since $e(\theta) > 0$, we have 
	$
	e'(\theta) > \frac{\theta}{e(\theta) + 1} > 0.
	$ 
	This shows that $e(\xi)$ is strictly increasing on $(0, +\infty)$.
\end{proof}

Due to the strict monotonicity of $\eta = e(\xi)$, its inverse function exists and is denoted by $\xi = g(\eta)$, which is also strictly increasing on $(e_0, e_\infty)$, where 
$$
e_0 := \lim_{\xi \to 0^+} e(\xi), \quad e_\infty := \lim_{\xi \to \infty} e(\xi).
$$
We can rewrite \eqref{eq:504} as
$$
\eta = \sqrt{g(\eta)^2 + 1 + C_e(g(\eta))} - 1,
$$
which leads to the following expression for $g(\eta)$:
\begin{equation}\label{eq:539}
	g(\eta) = \sqrt{(\eta + 1)^2 - 1 - C_e(g(\eta))} = \sqrt{\eta^2 + 2\eta + C_g(\eta)},
	\quad \eta \in (e_0, e_\infty),
\end{equation}
where $C_g(\eta) := -C_e(g(\eta))$. Since both $g(\cdot)$ and $C_e(\cdot)$ are strictly increasing, the auxiliary function $C_g(\eta)$ is strictly decreasing.
Furthermore, substituting $\eta = e(\xi)$ into \eqref{eq:539} yields
$$
\xi = g(e(\xi)) = \sqrt{(e(\xi))^2 + 2e(\xi) + C_g(e(\xi))},
$$
or equivalently,
\begin{equation}\label{eq:646}
	C_g(e(\xi)) = \xi^2 - (e(\xi))^2 - 2e(\xi),
	\quad \xi \in (0, +\infty).
\end{equation}

We now establish the following key inequality, whose discovery and proof are both very technical.

\begin{lemma} \label{lem:556}
	If the function $e(\xi)$ satisfies the condition \eqref{eq:206}, then for any $a, b > 0$,
	\begin{equation}\label{eq:558}
		\exp\left\{
		\int_a^{b} \frac{e^{\prime}(\xi)}{\xi} \, \mathrm{d} \xi 
		\right\} \ge G(a,b),
	\end{equation}
	where
	\begin{equation*}
		\begin{gathered}
			G(a,b) := 
			\left(
			\frac
			{b+\sqrt{b^2-a^2+(e(a)+1)^2}}
			{1+a+e(a)}
			\right)
			\exp
			\left\{
			\frac{1}{b} \left( e(b)+1-\sqrt{b^2-a^2+(e(a)+1)^2} \right)
			\right\}.
		\end{gathered}
	\end{equation*}
\end{lemma}

\begin{proof}
 We first consider the case where $a \le b$. Note that
	\begin{equation}\label{eq:583}
		\int_a^{b} \frac{e^{\prime}(\xi)}{\xi} \textrm{d} \xi 
		= 
		\int_a^{b} \frac{1}{\xi} \textrm{d} e(\xi)
		\xlongequal{\eta = e(\xi)} 
		\int_{e(a)}^{e(b)} \frac{1}{g(\eta)} \textrm{d} \eta
		\overset{\eqref{eq:539}}{=}
		\int_{e(a)}^{e(b)} \frac{1}{\sqrt{\eta^2+2\eta+C_g(\eta)}} \textrm{d} \eta.
	\end{equation}
Since $e(\xi)$ is increasing and $C_g(\eta)$ is decreasing, we have $e(a) \le e(b)$, and thus the following upper and lower bounds for $C_g(\eta)$:
	\begin{equation}\label{eq:593}
		C_g(e(a)) \ge C_g(\eta) \ge C_g(e(b)) 
		\qquad 
		\forall \, \eta \in [e(a),e(b)].
	\end{equation}
     The monotonicity of $g(\eta)$ gives $g(e(a)) \le g(\eta) \le g(e(b))$ for any $\eta \in [e(a), e(b)]$. Substituting \eqref{eq:539} yields
	\[
	\sqrt{(e(a))^2+2e(a)+C_g(e(a))}
	\le
	\sqrt{\eta^2+2\eta+C_g(\eta)}
	\le
	\sqrt{(e(b))^2+2e(b)+C_g(e(b))},
	\]
which leads to the following upper and lower bounds of $C_g(\eta)$ for any $\eta \in [e(a),e(b)]$:
	\begin{equation}\label{eq:600}
		C_g(e(a))+(e(a)+1)^2-(\eta+1)^2 
		 \le 
		C_g(\eta) 
		\le 
		C_g(e(b))+(e(b)+1)^2-(\eta+1)^2.
	\end{equation}
Combining the two pairs of lower and upper bounds from \eqref{eq:593} and \eqref{eq:600}, we obtain
	\[
	C_{g,\min}(\eta) \le C_g(\eta) \le C_{g,\max}(\eta)
	\quad 
	\forall ~ \eta \in [e(a),e(b)],
	\]
	where
	\[
	\begin{aligned}
		C_{g,\min}(\eta)
		:= & \max\{C_g(e(b)),(e(a)+1)^2+C_g(e(a))-(\eta+1)^2\} \\
		\overset{\eqref{eq:646}}{=} & \max\{b^2-e(b)^2 - 2e(b),a^2-\eta^2-2\eta\} \\
		= & \begin{dcases}
			a^2-\eta^2-2\eta 	& \text{if} ~ \eta \in [e(a), \sqrt{a^2-b^2+(e(b)+1)^2}-1], \\
			b^2-e(b)^2-2e(b)	& \text{if} ~ \eta \in [\sqrt{a^2-b^2+(e(b)+1)^2}-1, e(b)],
		\end{dcases}
	\end{aligned}
	\]
	\[	
	\begin{aligned}
		C_{g,\max}(\eta)
		:= & \min\{C_g(e(a)),(e(b)+1)^2+C_g(e(b))-(\eta+1)^2\}  \\
		\overset{\eqref{eq:539}}{=} & \min\{a^2-e(a)^2 - 2e(a),b^2-\eta^2-2\eta\} \\
		= & \begin{dcases}
			a^2-e(a)^2-2e(a)    & \text{if} ~ \eta \in [e(a), \sqrt{b^2-a^2+(e(a)+1)^2}-1], \\
			b^2-\eta^2-2\eta 	& \text{if} ~ \eta \in [\sqrt{b^2-a^2+(e(a)+1)^2}-1,e(b)].
		\end{dcases}
	\end{aligned}
	\]
	Thus, we obtain the following lower and upper bounds for the integral in \eqref{eq:583}:
	\[
	\begin{aligned}
		& 
		\int_a^{b} \frac{e^{\prime}(\xi)}{\xi} \textrm{d} \xi 
		=
		\int_{e(a)}^{e(b)} \frac{1}{\sqrt{\eta^2+2\eta+C_g(\eta)}} \textrm{d} \eta 
		\ge  
		\int_{e(a)}^{e(b)} \frac{1}{\sqrt{\eta^2+2\eta+C_{g,\max}(\eta)}} \textrm{d} \eta \\
		= & 
		\int_{e(a)}^{\sqrt{b^2-a^2+(e(a)+1)^2}-1} \frac{1}{\sqrt{(\eta+1)^2+a^2-(e(a)+1)^2}} \textrm{d} \eta
		+   
		\int_{\sqrt{b^2-a^2+(e(a)+1)^2}-1}^{e(b)} \frac{1}{b} \textrm{d} \eta \\
		= & 
		\left.\log\left( \sqrt{(\eta+1)^2+a^2-(e(a)+1)^2}+\eta+1 \right)\right|_{e(a)}^{\sqrt{b^2-a^2+(e(a)+1)^2}-1}
		\\
		& + \frac{1}{b} \left( e(b)+1-\sqrt{b^2-a^2+(e(a)+1)^2} \right),
	\end{aligned}
	\]
	\[
	\begin{aligned}
		& 
		\int_a^{b} \frac{e^{\prime}(\xi)}{\xi} \textrm{d} \xi 
		=
		\int_{e(a)}^{e(b)} \frac{1}{\sqrt{\eta^2+2\eta+C_g(\eta)}} \textrm{d} \eta 
		\le \int_{e(a)}^{e(b)} \frac{1}{\sqrt{\eta^2+2\eta+C_{g,\min}(\eta)}} \textrm{d} \eta \\
		= & \int_{e(a)}^{\sqrt{a^2-b^2+(e(b)+1)^2}-1} \frac{1}{a} \textrm{d} \eta
		+   \int_{\sqrt{a^2-b^2+(e(b)+1)^2}-1}^{e(b)} \frac{1}{\sqrt{(\eta+1)^2+b^2-(e(b)+1)^2}} \textrm{d} \eta \\
		= & \frac{1}{a} \left( \sqrt{a^2-b^2+(e(a)+1)^2}-e(a)-1 \right)
		\\
		&+ \left.\log\left( \sqrt{(\eta+1)^2+b^2-(e(b)+1)^2}+\eta+1 \right)\right|_{\sqrt{a^2-b^2+(e(b)+1)^2}-1}^{e(b)}.
	\end{aligned}
	\]
	Applying the exponential function to both sides of these inequalities gives
	\begin{equation}\label{eq:671}
		\exp\left\{\int_a^{b} \frac{e^{\prime}(\xi)}{\xi} \mathrm{d} \xi\right\} 
		\ge G(a,b)
		\quad \textrm{and} \quad
		\exp\left\{\int_a^{b} \frac{e^{\prime}(\xi)}{\xi} \mathrm{d} \xi \right\} 
		\le G(b,a)^{-1}.
	\end{equation}
    The first inequality in \eqref{eq:671} immediately proves the claim \eqref{eq:558} when $a \le b$. 
    For the remaining case, where $a \ge b$, we use the second inequality in \eqref{eq:671} to prove the claim \eqref{eq:558}:
	\[
	\exp\left\{ \int_a^{b} \frac{e^{\prime}(\xi)}{\xi} \mathrm{d} \xi \right\} 
	= 
	\exp\left\{ \int_b^{a} \frac{e^{\prime}(\xi)}{\xi} \mathrm{d} \xi \right\}^{-1} 
	\ge 
	\left(G(a,b)^{-1}\right)^{-1} = G(a,b). 
	\]
	The proof is completed. 
\end{proof}

\begin{lemma} \label{lem:359}
	For any $x \ge 0$, we have 
	$
	\mathrm{e}^x \ge 1 + x + \frac{x^2}{2}.
	$ 
\end{lemma}

\begin{proof}
	For $x \ge 0$, consider the Taylor expansion of $\mathrm{e}^x$ around 0: there exists some $\xi \in [0, x]$ such that
	\[
	\mathrm{e}^x = 1 + x + \frac{x^2}{2} + \frac{x^3}{6} \mathrm{e}^\xi.
	\]
	Therefore, 
	$ 
	\mathrm{e}^x \ge 1 + x + \frac{x^2}{2},
	$ 
	because $\mathrm{e}^\xi \ge 1$ for $\xi \ge 0$.
\end{proof}

\begin{lemma} \label{lem:373}
	For any $x \le 0$, we have 
	$
	\mathrm{e}^x \ge 1 + x + \frac{x^2}{2(1-x)}.$ 
\end{lemma}

\begin{proof}
	For $x \le 0$, consider the Taylor expansion of $\mathrm{e}^x$ around 0: there exists $\xi \in [x,0]$ such that
	\[
	\mathrm{e}^x = 1 + x + \frac{x^2}{2} + \frac{x^3}{6} + \frac{x^4}{24}  \mathrm{e}^\xi.
	\]
	Hence,
	\[
	\mathrm{e}^x \ge 1 + x + \frac{x^2}{2} + \frac{x^3}{6} =: h_1(x).
	\]
	Now consider the function
	\[
	h(x) := 1 + x + \frac{x^2}{2(1-x)} = 1 + x + \frac{x^2}{2} + \frac{x^3}{2(1-x)}.
	\]
	It follows that
	\[
	 h(x) \le h_1(x) \le \mathrm{e}^x \qquad \forall x \in [-2,0]. 
	\]
	On the other hand, for $x \le -2$, we have
	\[
	h(x) = 1 + x + \frac{x^2}{2(1-x)} = \frac{2 - x^2}{2(1-x)} < 0 < \mathrm{e}^x. 
	\]
	Thus, the proof is complete.
\end{proof}

Finally, we are in a position to present the main result of this section, which serves as the cornerstone for establishing the entropy-preserving analysis through the GQL technique.

\begin{lemma} \label{lem:707}
	Consider the function $e(\cdot)$ satisfying \eqref{eq:206}. 
	For any $\theta, \theta_* > 0$, $\varepsilon \in [-1,1]$, $d \in \mathbb{N}$, $\Bv, \Bv_* \in \BALL $ with $\BALL$ being the unit open ball in $\mathbb{R}^d$, and $i = 1, \dots, d$, the following inequality holds: 
	\begin{equation}\label{eq:703}
		\left(1+\theta+e(\theta) \right) \C
		-\left(1+\theta_*+e(\theta_*)\right) \Cs
		+\theta_*\exp\left\{\int_\theta^{\theta_*} \frac{e^{\prime}(\xi)}{\xi} \mathrm{d} \xi \right\}  
		-\theta 
		\ge 0,
	\end{equation}
	where $\Bv = (v_1,\dots,v_d)^\top$, $\Bv_* = (v_{1,*},\dots,v_{d,*})^\top$, and
	\[
	\C
	:= 
	\Bigg( 
	\frac{1-\Bv \cdot \Bv_*}{1-|\Bv|^2} 
	\Bigg) 
	\cdot 
	\Bigg( 
	\frac{1+\varepsilon v_i}{1+\varepsilon v_{i,*}} 
	\Bigg), 
	\quad 
	\Cs
	:= 
	\Bigg( 
	\frac{\sqrt{1-|\Bv_*|^2}}{\sqrt{1-|\Bv|^2}} 
	\Bigg) 
	\cdot 
	\Bigg( 
	\frac{1+\varepsilon v_i}{1+\varepsilon v_{i,*}} 
	\Bigg).
	\]
\end{lemma}

\begin{proof}
	First, we observe $\C \ge (\Cs^2+1)/2$, because 
	\begin{align*}
		& \C - \frac{\Cs^2+1}{2} \\
		= &  \frac
		{
			2(1-\Bv \cdot \Bv_*)(1+\varepsilon v_i)(1+\varepsilon v_{i,*})
			-(1-|\Bv  |^2)(1+\varepsilon v_{i,*})^2
			-(1-|\Bv_*|^2)(1+\varepsilon v_{i  })^2
		}
		{2(1-|\Bv|^2)(1+\varepsilon v_{i,*})^2} \\
		= & \frac
		{
			\big|
			\Bv  (1+\varepsilon v_{i,*})
			-\Bv_*(1+\varepsilon v_{i  })
			\big|^2
			-(\varepsilon v_{i  }-\varepsilon v_{i,*})^2
		}
		{2(1-|\Bv|^2)(1+\varepsilon v_{i,*})^2} \\
		= & \frac{
			\sum\limits_{j\neq i} \big[v_j(1+\varepsilon v_{i,*})-v_{j,*}(1+\varepsilon v_{i})\big]^2
			+\big[v_i(1+\varepsilon v_{i,*})-v_{i,*}(1+\varepsilon v_{i})\big]^2 
			-\varepsilon^2 (v_{i  }-v_{i,*})^2	
		}{2(1-|\Bv|^2)(1+\varepsilon v_{i,*})^2} \\
		= & \sum_{j\neq i} \frac
		{\big[v_j(1+\varepsilon v_{i,*})-v_{j,*}(1+\varepsilon v_{i})\big]^2}
		{2(1-|\Bv|^2)(1+\varepsilon v_{i,*})^2}
		+\frac
		{(1-\varepsilon^2) (v_{i  }-v_{i,*})^2}
		{2(1-|\Bv|^2)(1+\varepsilon v_{i,*})^2} 
		\ge 0.
	\end{align*}
	Thus, in order to prove \eqref{eq:703}, it suffices to prove for any $\Cs \ge 0$ that
	\begin{equation}\label{eq:777}
		\Big(1+\theta+e(\theta) \Big) \frac{\Cs^2+1}{2}
		-\Big(1+\theta_*+e(\theta_*)\Big) \Cs
		+\theta_*\exp\left\{\int_\theta^{\theta_*} \frac{e^{\prime}(\xi)}{\xi} \mathrm{d} \xi \right\} -\theta \ge 0.
	\end{equation}
	For convenience, we denote 
	\begin{align} \nonumber
		& s:= \sqrt{\theta_*^2-\theta^2+(e(\theta)+1)^2}=\sqrt{\theta_*^2+1+C_e(\theta)},
		\\ \label{eq:782a}
		& \phi
		:= \frac{s}{\theta_*}
		= \frac{\sqrt{\theta_*^2+1+C_e(\theta)}}{\theta_*} > 0,
		\\ \label{eq:782b}
		& 
		\phi_* 
		:= \frac{e(\theta_*)+1}{\theta_*} 
		= \frac{\sqrt{\theta_*^2+1+C_e(\theta_*)}}{\theta_*} > 0.
	\end{align}
	After applying Lemma \ref{lem:556} on the exponential term in \eqref{eq:777}, we only need to prove
	\begin{equation*}
		\begin{gathered}
			\underbrace{
				\Big(1+\theta+e(\theta) \Big) \frac{\Cs^2+1}{2}
				-\Big(1+\theta_*+e(\theta_*)\Big) \Cs 
				-\theta
			}_{H_1}
			+
			\underbrace{
				\frac
				{\theta_*(\theta_*+s)}
				{1+\theta+e(\theta)}
				\cdot\exp\left\{\frac{ e(\theta_*)+1-s }{\theta_*} \right\}
			}_{H_2}
			\ge 0.
		\end{gathered}
	\end{equation*}
	Note that $H_1$ satisfies 
	\[
	\begin{aligned}
		H_1 
		= & \frac
		{
			\left[
			\big(1+\theta+e(\theta) \big) \Cs
			-\big(1+\theta_*+e(\theta_*)\big)
			\right]^2
		}
		{2 \big(1+\theta+e(\theta) \big) }
		+ \frac
		{
			(e(\theta)+1)^2-\theta^2-\big(1+\theta_*+e(\theta_*)\big)^2
		}
		{2 \big(1+\theta+e(\theta) \big) }\\
		\ge & \frac
		{
			(e(\theta)+1)^2-\theta^2-\big(1+\theta_*+e(\theta_*)\big)^2
		}
		{2 \big(1+\theta+e(\theta) \big) } 
		=  \frac
		{
			s^2-\big(e(\theta_*)+1\big)^2-2\theta_*^2-2 \theta_* \big(1+e(\theta_*)\big)
		}
		{2 \big(1+\theta+e(\theta) \big) } \\
		= & \frac
		{\theta_*^2}
		{2 \big(1+\theta+e(\theta) \big) } 
		\left[
		\left(\frac{s}{\theta_*}\right)^2
		- \left( \frac{e(\theta_*)+1}{\theta_*} \right)^2
		-2 \left( \frac{e(\theta_*)+1}{\theta_*} \right)
		-2
		\right] \\
		= &  \frac
		{\theta_*^2}
		{2 \big(1+\theta+e(\theta) \big) } 
		\left( \phi^2 - \phi_*^2 - 2 \phi_*	-2	\right) 
		=-\frac
		{\theta_*^2 (\phi+1)}
		{1+\theta+e(\theta) } 
		\left[ 1 + (\phi_*-\phi) + \frac{(\phi_*-\phi)^2}{2(\phi+1)} \right].
	\end{aligned}
	\]
	Similarly, $H_2$ can be rewritten as
	$\displaystyle{
		H_2 = \frac
		{\theta_*^2 (\phi+1)}
		{1+\theta+e(\theta)  }
		\exp \{\phi_*-\phi\}.
	}$
	Therefore,
	\begin{equation*}
		H_1+H_2 \ge  
		\frac
		{\theta_*^2 (\phi+1)}
		{ 1+\theta+e(\theta) }
		\Bigg[\underbrace{
			\exp \{\phi_*-\phi\} - 1 - (\phi_*-\phi) - \frac{(\phi_*-\phi)^2}{2(\phi+1)}
		}_{H_4} 
		\Bigg].
	\end{equation*}
	Now, it suffices to prove that $H_4 \ge 0$. 
	In the case of $0 < \theta \le \theta_*$, the definitions in \eqref{eq:782a} and \eqref{eq:782b} imply that $\phi_* \ge \phi$, due to the monotonicity of $C_e(\cdot)$. Therefore, we have 
\begin{align*}
		H_4 &\ge \exp \{\phi_*-\phi\} - 1 - (\phi_*-\phi) - \frac{(\phi_*-\phi)^2}{2(1+\phi-\phi)}
		\\
	&= 
	\left[ \mathrm{e}^\xi-1-\xi-\frac{\xi^2}{2} \right]_{\xi = \phi_*-\phi \ge 0} 
	\stackrel{\rm Lemma \ \ref{lem:359}}{\ge} 
	0.
\end{align*}
	In the remaining case of $0 < \theta_* \le \theta$, we have $\phi_* \le \phi$, and thus
\begin{align*}
	H_4 &\ge \exp \{\phi_*-\phi\} - 1 - (\phi_*-\phi) - \frac{(\phi_*-\phi)^2}{2(1+\phi-\phi_*)}
	\\
&	= \left[ \mathrm{e}^\xi-1-\xi-\frac{\xi^2}{2(1-\xi)} \right]_{\xi = \phi_*-\phi \le 0} 
	\stackrel{\rm Lemma\ \ref{lem:373}}{\ge} 0.
\end{align*}
Therefore, the proof is complete. 
\end{proof}

\section{Analysis and Design of Entropy-Preserving Schemes} \label{sec:scheme}

This section is dedicated to developing a provably entropy-preserving framework for finite volume and DG schemes for the relativistic Euler system \eqref{eq:RHD3D} with a general EOS of the form \eqref{eq:175}.  The proposed schemes are designed to preserve the (global or local) MEP:
\begin{equation}\label{eq:MEP}
	S(\BU) \ge \sigma,
\end{equation}
where $\sigma$ represents a priori lower bound for the specific entropy. If $\sigma = S^{\rm G}_{\min}$, this defines a (global) entropy-preserving scheme. If $\sigma = S^{\rm L}_{\min}({\bm x}, t)$, where $S^{\rm L}_{\min}({\bm x}, t)$ is a properly estimated local entropy bound, then the scheme is locally entropy-preserving. The estimation of the local entropy bound will be discussed in Section \ref{sec:6}.

Since there is no explicit expression for the implicit function $S(\BU)$, we must first ensure that $\BU$ satisfies the fundamental constraints:
\begin{equation}\label{eq:3constraints}
	\rho({\bf U}) > 0, \quad p({\bf U}) > 0, \quad |{\bm v}({\bf U})| < 1,
\end{equation}
i.e., ${\bf U} \in {\mathcal G}$, where ${\mathcal G}$ is defined in \eqref{eq:DefG}. The pressure $p({\bf U})$ is an implicit function defined via the positive root of the following nonlinear equation \cite{WuTang2017ApJS}:
$$
\frac{1}{D} \sqrt{(E + p)^2 - |{\bm m}|^2} = h(\theta) = h \left( \frac{p}{D} \left( 1 - \frac{|{\bm m}|^2}{(E + p)^2} \right)^{-\frac{1}{2}} \right),
$$
and the velocity ${\bm v}({\bf U})$ and density $\rho({\bf U})$ can be expressed as
$$
{\bm v}({\bf U}) = \frac{{\bm m}}{E + p({\bf U})}, \quad \rho({\bf U}) = D \sqrt{1 - |{\bm v}({\bf U})|^2} = D \sqrt{1 - \frac{|{\bm m}|^2}{(E + p({\bf U}))^2}}.
$$
Once the basic physical constraints \eqref{eq:3constraints} are satisfied, the function $S({\bf U})$ is defined as
$$
S(\mathbf{U}) = -\ln \Big( \rho({\bf U}) \Big) + \int_1^{\theta({\bf U})} \frac{e'(\xi)}{\xi} \, \mathrm{d}\xi \quad \text{with} \quad \theta({\bf U}) := \frac{p({\bf U})}{\rho({\bf U})}.
$$
Thus, the preservation of the MEP \eqref{eq:MEP} must be considered alongside the preservation of the fundamental constraints \eqref{eq:3constraints}, which is essential for the robustness of the numerical schemes \cite{WuTang2015, WuTang2017ApJS, chen2022physical}. 
To this end, we define the set
\begin{equation*}
	\Omega_\sigma := \left\{ \BU = (D, \Bm^\top, E)^\top \in \mathbb{R}^{d+2} \mid \rho(\BU) > 0, \, p(\BU) > 0, \, \Bv(\BU) \in \BALL, \, S(\BU) \ge \sigma \right\}.
\end{equation*}
Our objective is to preserve the numerical solutions within $\Omega_\sigma$ for a priori entropy bound $\sigma$.

\subsection{Geometric Quasi-Linearization (GQL): Overcoming the Challenges Arising from Nonlinearity}

Although the three functions involved in \eqref{eq:3constraints} cannot be explicitly expressed, the set $\mathcal{G}$ defined by these constraints \eqref{eq:3constraints} does have an explicit equivalent representation:
\begin{equation} \label{1stequivG}
	\mathcal{G} = \left\{ \BU = (D, \Bm^\top, E)^\top \in \mathbb{R}^{d+2} : D > 0, \, q(\BU) := E - \sqrt{D^2 + |\Bm|^2} > 0 \right\},
\end{equation}
as proven in \cite{WuTang2015} for the ideal EOS and in \cite{WuTang2017ApJS} for general EOSs. This leads to the following equivalent representation for $\Omega_\sigma$:
\begin{equation} \label{1stequiv}
	\Omega_{\sigma}^{(1)} = \left\{ \BU = (D, \Bm^\top, E)^\top \in \mathbb{R}^{d+2} : D > 0, \, q(\BU) > 0, \, S(\BU) \geq \sigma \right\}.
\end{equation}

However, the implicit function $S(\BU)$ is highly nonlinear, making the entropy-preserving analysis theoretically challenging. To overcome this difficulty, we employ the GQL framework \cite{wu2023geometric}, which allows us to transform all the nonlinear constraints in \eqref{1stequiv} into equivalent linear constraints by introducing additional free parameters. Before applying the GQL approach, we first verify that the set $\Omega_\sigma$ is convex.

\begin{lemma}
	For any given $\sigma \in \mathbb{R}$, the set $\Omega_\sigma$ is convex.
\end{lemma}

\begin{proof}
	The set $\Omega_{\sigma}^{(1)}$ can be equivalently expressed as
	\begin{equation} \label{1stequivNew}
		\Omega_{\sigma}^{(1)} = \left\{ \BU = (D, \Bm^\top, E)^\top \in \mathbb{R}^{d+2} : D > 0, \, q(\BU) > 0, \, -D(S(\BU) - \sigma) \leq 0 \right\},
	\end{equation}
	where $D$ is a linear function of $\BU$, $q(\BU) = E - \sqrt{D^2 + |\Bm|^2}$ is a concave function of $\BU$ (as shown in \cite{WuTang2015}), and $-D(S(\BU) - \sigma)$ is a convex function of $\BU$, according to \Cref{lem:convexentropy} with ${\mathcal H}(S) = S - \sigma$. 
	
	By applying Jensen's inequality, it follows that $\Omega_{\sigma}^{(1)}$, and equivalently $\Omega_\sigma$, is a convex set.
\end{proof}

Thanks to the GQL technique, we 
 discover the following equivalent linear representation \eqref{eq:1218} of $\Omega_{\sigma}$. 

\begin{theorem}[GQL Representation] \label{thm:1216}
	The set $\Omega_{\sigma}$, or equivalently $\Omega_{\sigma}^{(1)}$, can be represented as 
	\begin{equation}\label{eq:1218}
		\Omega_{\sigma}^{(2)} 
		= 
		\big\{
		\BU
		=
		(D,\Bm^\top,E)^\top
		:
		D>0,
		\BU\cdot \tilde \Bn_* > 0,
		\BU\cdot \Bn_*+p_* \geq 0,
		\forall \, \Bv_*\in \BALL,
		\forall \, \theta_* > 0
		\big\},
	\end{equation}
	where $Z(\theta):=\int_1^\theta { e'(\xi)}/{\xi} \, \mathrm{d}\xi$ is the integral part of the specific entropy defined in \eqref{eq:405}, $\BALL$ is the unit open ball in $\mathbb{R}^d$, and
		\begin{align}
		&\tilde\Bn_*:=
		\big(
		-\sqrt{1-|\Bv_*|^2},
		-\Bv_*^\top,
		1
		\big)^\top, \qquad p_* := \theta_* \exp\{ Z(\theta_*) - \sigma \},\label{tildenstar}
		\\
		&\Bn_*:=
		\big(
		-h_*\sqrt{1-|\Bv_*|^2},
		-\Bv_*^\top,
		1
		\big)^\top, \qquad h_* := 1+\theta_*+e(\theta_*).\nonumber
	\end{align}
Here, $\Bv_*$ and $\theta_*$ are auxiliary parameters introduced in the GQL framework \cite{wu2023geometric} to linearize the constraints.
\end{theorem}

\begin{proof}
	The nonlinear constraint $q(\BU) = E - \sqrt{D^2 + |\Bm|^2} > 0$ in \eqref{1stequiv} is equivalent to the linear constraints $\BU \cdot \tilde{\Bn}_* > 0$ for all $\Bv_* \in \BALL$, as established in \cite[Theorem 4.9]{wu2023geometric} for the 1D case ($d = 1$), with a similar argument holding for $d \ge 2$.

	Next, we address the second nonlinear constraint $S(\BU) \geq \sigma$ in \eqref{1stequiv}. Consider the following hypersurface on the boundary of $\Omega_{\sigma}^{(1)}$:
	\begin{equation*}
		\widetilde{\mathcal{S}}_\sigma = 
		\left\{
		\textbf{U}=(D,\Bm^\top,E)^\top\in \mathbb{R}^{d+2}:\ D>0,\ S(\textbf{U})=\sigma
		\right\}.
	\end{equation*}
Since $p$ and $\rho$ can be expressed as
\begin{equation*}
	\rho = \textrm{e}^{-S+Z(\theta)},\qquad
	p = \rho\theta = \theta \textrm{e}^{-S+Z(\theta)}, 
\end{equation*}
the hypersurface $\widetilde{\mathcal{S}}_\sigma$ can be parameterized as
	\begin{equation*}
		\widetilde{\mathcal{S}}_\sigma = \left\{
		{\bm u}_*
		=
		\left(
		\frac{\rho_*}{\sqrt{1-|\Bv_*|^2}},
		\frac{\rho_*h(\theta_*)\Bv_*^\top}{1-|\Bv_*|^2},
		\frac{\rho_*h(\theta_*)}{1-|\Bv_*|^2}-\theta_*\rho_*
		\right)^\top: \Bv_*\in\mathbb{B}_1(\mathbf{0}), \theta_*>0 
		\right\},
	\end{equation*}
where $\theta_*$ and $\Bv_*$ are auxiliary parameters, and $\rho_* := \mathrm{e}^{-\sigma + Z(\theta_*)}$. 
The normal vector ${\bm n}_*$ to $\widetilde{\mathcal{S}}_\sigma$ can be derived using the cross product: 
	\begin{equation*}
		\frac{\partial {\bm u}_*}{\partial \theta_*}\times\left(\bigwedge\limits_{i=1}^d\frac{\partial {\bm u}_*}{\partial v_{i,*}}\right)=\frac{1}{\delta_*}{\bm n}_*, 
	\end{equation*}
where 
	\begin{align*}
	\frac{\partial {\bm u}_*}{\partial \theta_*}&=
	\begin{pmatrix}
		\frac{\rho_*e'(\theta_*)}{\theta_*\sqrt{1-|\Bv_*|^2}},&
		\frac{\frac{\rho_*e'(\theta_*)}{\theta_*}h(\theta_*)+\rho_*(1+e'(\theta_*))}{1-|\Bv_*|^2}\Bv_*^\top,&
		\frac{\frac{\rho_*e'(\theta_*)}{\theta_*}h(\theta_*)+\rho_*(1+e'(\theta_*))}{1-|\Bv_*|^2}-\rho_*(1+e'(\theta_*))
	\end{pmatrix}^\top,\\
	\frac{\partial{\bm u}_*}{\partial v_{i,*}}&=
	\begin{pmatrix}
		\frac{\rho_*v_{i,*}}{\left(1-|\Bv_*|^2\right)^\frac{3}{2}},&
		\frac{\rho_*h(\theta_*)\big((1-|\Bv_*|^2)\mathbf{e}_i^\top+2v_{i,*}\Bv_{*}^\top\big)}{\left(1-|\Bv_*|^2\right)^2},&
		\frac{2\rho_*h(\theta_*)v_{i,*}}{\left(1-|\Bv_*|^2\right)^2}
	\end{pmatrix}^\top,\qquad i = 1,\ldots,d,
\end{align*}
and 
	\begin{align*}
		&\delta_*:=\frac{(-1)^{d+1}\rho_*^{d+1}\big(h(\theta_*)\big)^{d-1}\big(h(\theta_*)e'(\theta_*)-\theta_*(1+e'(\theta_*))|\Bv_*|^2\big)}{\theta_*\left(1-|\Bv_*|^2\right)^{d+\frac{3}{2}}},
		\\
		&{\bm n}_*:=\left(-\sqrt{1-|\Bv_*|^2}h(\theta_*),-\Bv_*^\top,1\right)^\top. 
	\end{align*}
	According to the GQL framework \cite[Theorem 4.2]{wu2023geometric}, the equivalent linear representation for the nonlinear constraint $S(\BU) \geq \sigma$ is
	\begin{equation} \label{GQL:lconstraint}
		\phi_{\sigma}(\mathbf{U};\theta_*,\Bv_*)\geq 0,\qquad \forall \, \theta_*>0, \ |\Bv_*|<1, 
	\end{equation}
	where
	\begin{align*}
		\phi_{\sigma}({\bm U};\theta_*,\Bv_*)=(\mathbf{U}-{\bm u}_*)\cdot{\bm n}_*=
		\BU \cdot {\bm n}_*+p_*.
	\end{align*}
Thus, the two nonlinear constraints in \eqref{1stequiv} can be equivalently represented by the linear constraints in the GQL representation \eqref{eq:1218}. This completes the proof.
\end{proof}

\begin{remark}
	Thanks to its linearity, the GQL representation $\Omega_{\sigma}^{(2)}$ in \Cref{thm:1216} offers significant advantages over the original nonlinear form of the set $\Omega_{\sigma}$ in entropy-preserving analysis, as demonstrated in the subsequent subsections. 
	
	It is worth noting that if the function $Z(\theta)$ satisfies
				\begin{equation} \label{oneGQLcond}
			\lim\limits_{\theta\to0^+}Z(\theta)=-\infty,\qquad
			\lim\limits_{\theta\to+\infty}Z(\theta)=\infty,
		\end{equation}
	then the two linear constraints $\BU \cdot \tilde{\Bn}_* > 0$ and $\BU \cdot \Bn_* + p_* \geq 0$ in \eqref{eq:1218} reduce to a single constraint $\BU \cdot \Bn_* + p_* \geq 0$. 
Indeed, under the condition \eqref{oneGQLcond}, there always exist parameters $\theta_* > 0$ and $\Bv_* \in \mathbb{B}_1(\mathbf{0})$ such that:
		\begin{equation*}
			\rho_*:=\textrm{e}^{-\sigma+Z({\theta}_*)}=\frac{D^2}{\sqrt{D^2+|\Bm|^2}},\qquad
			{\Bv}_*=\frac{\Bm}{\sqrt{D^2+|\Bm|^2}}. 
		\end{equation*}
		With these choices, the linear constraint $\BU \cdot \Bn_* + p_* \geq 0$ implies
		\begin{align*}
			0
			\leq
			\BU \cdot {\bm n}_* + p_*
			=
			E-\sqrt{D^2+|\Bm|^2}-\frac{D^2e(\theta_*)}{\sqrt{D^2+|\Bm|^2}}<E-\sqrt{D^2+|\Bm|^2}=q(\BU).  
		\end{align*}
		This demonstrates that the two constraints $\BU \cdot \tilde{\Bn}_* > 0$ and $\BU \cdot \Bn_* + p_* \geq 0$ are equivalent to the single constraint $\BU \cdot \Bn_* + p_* \geq 0$, provided that \eqref{oneGQLcond} holds.
		
	The condition \eqref{oneGQLcond} is satisfied for a broad class of Synge-type EOSs, including the ideal EOS \eqref{ID-EOS}, RC-EOS \eqref{hEOS1}, IP-EOS \eqref{hEOS2}, and TM-EOS \eqref{hEOS3}. For these EOSs, the GQL representation of $\Omega_\sigma$ simplifies to
		\begin{equation*}
			\Omega_{\sigma}^{(2),*} = \left\{\textbf{U}=(D,\Bm^\top,E)^\top\in \mathbb{R}^{d+2}:\ D>0,\ \BU\cdot \Bn_*+p_* \geq 0,
			\forall \, \Bv_*\in \BALL,
			\forall \, \theta_* > 0
			\right\}.
		\end{equation*}
	\end{remark}
	
Based on the GQL framework, we derive the following two critical inequalities, which play a pivotal role in analyzing and estimating the influence of fluxes on the entropy-preserving property of a scheme.

\begin{lemma}
	\label{lem:GQLnstar}
	For any {$\mathbf{U} \in \mathcal{G}$}, the following inequality holds:
	\begin{equation*}
		\big(\beta\mathbf{U} \pm \mathbf{F}_i(\mathbf{U})\big) \cdot \tilde \Bn_* > 0,  
		\quad \text{for any} \quad \beta \geq c=1, \quad i=1,2,\dots,d,
	\end{equation*}
	where $\tilde \Bn_*$ is defined in \eqref{tildenstar}.
\end{lemma}

\begin{proof}
	By \cite[Lemma 3.4(iii)]{WuTang2017ApJS}, we have $\beta\mathbf{U} \pm \mathbf{F}_i(\mathbf{U}) \in \mathcal{G}$. The result then follows directly from \Cref{thm:1216}, utilizing the GQL representation of $\mathcal{G}$.
\end{proof}

\begin{lemma} \label{thm:475}
	Given $\BU \in \Omega_\sigma$, for any $\Bv_* \in \mathbb{B}_1(\mathbf{0})$ and $\theta_* > 0$, the following inequalities hold for any $i = 1, \dots, d$:
	\begin{equation}\label{eq:1416}
		\BF_i(\BU)\cdot \Bn_*
		\ge
		-\alpha
		(
		\BU\cdot \Bn_*+p_*
		)
		-v_{i,*}p_*, \quad
		-\BF_i(\BU)\cdot \Bn_*
		\ge
		-\alpha
		(
		\BU\cdot \Bn_*
		+p_*
		)
		+v_{i,*}p_*.	
	\end{equation}
\end{lemma}

\begin{proof}
	It suffices to prove that for $\varepsilon = \pm 1$, the following inequality holds:
	\[
	\Pi
	:=
	\big( \BU + \varepsilon \BF_i(\BU) \big)\cdot \Bn_*
	+
	p_*
	+
	\varepsilon \, v_{i,*} p_*
	\ge 0.		
	\]
	Using the definitions of $\BU$ and $\BF_i(\BU)$ in \eqref{eq:86} and \eqref{eq:93}, we express $\BU + \varepsilon \BF_i(\BU)$ in terms of primitive variables:
	\begin{align*}
		\BU + \varepsilon \, \BF_i (\BU)
		& =
		\big(~ 
		\rho \gamma, ~
		\rho h \gamma^2 \Bv^\top, ~
		\rho h \gamma^2 - p 
		~\big)^\top
		+
		\varepsilon
		\big(~ 
		\rho \gamma v_i, ~
		\rho h \gamma^2 v_i  \Bv^\top + p \Be_i^\top, ~
		\rho h \gamma^2 v_i 
		~\big)^\top \\
		& =
		\big(~ 
		\rho \gamma (1 + \varepsilon v_i), ~
		\rho h \gamma^2 (1 + \varepsilon v_i) \Bv^\top + \varepsilon p \Be_i^\top, ~
		\rho h \gamma^2 (1 + \varepsilon v_i) - p
		~\big)^\top. 
	\end{align*}
	Substituting this expression into $\Pi$, we obtain
	\begin{align*}
		\Pi
		& =
		\big[
		\rho h \gamma^2 (1 + \varepsilon v_i) - p
		\big]
		-
		\big[
		\rho h \gamma^2 (1 + \varepsilon v_i) \Bv + \varepsilon p \Be_i
		\big] \cdot \Bv_*
		-
		\big[
		\rho \gamma (1 + \varepsilon v_i)
		\big]
		h_* \sqrt{1-v_*^2}
		+
		p_*
		+
		\varepsilon \, v_{i,*} \, p_* \\
		& =
		\rho h \gamma^2 (1 + \varepsilon v_i)(1-\Bv\cdot\Bv_*)
		-
		\rho h_* \gamma (1 + \varepsilon v_i)  \sqrt{1-v_*^2}
		+
		p_*(1+\varepsilon \, v_{i,*})
		-
		p (1+\varepsilon v_{i,*}).
	\end{align*}
	Substituting $\gamma = (1-|\Bv|^2)^{-\frac12}$, $\rho = \mathrm{e}^{Z(\theta)-S}$, $p = \theta \mathrm{e}^{Z(\theta)-S}$, and $p_* = \theta_* \mathrm{e}^{ Z(\theta_*) - \sigma }$, we further get 
	\[
	\Pi
	=
	\rho(1+\varepsilon v_{i,*})
	\left[
	\left(\frac{1-\Bv \cdot \Bv_*}{1-|\Bv|^2}\right)
	\left(\frac{1+\varepsilon v_i}{1+\varepsilon v_{i,*}}\right) h
	-
	\left(\frac{\sqrt{1-|\Bv_*|^2}}{\sqrt{1-|\Bv|^2}}\right)
	\left(\frac{1+\varepsilon v_i}{1+\varepsilon v_{i,*}}\right) h_*
	+\theta_* \mathrm{e}^{Z(\theta_*)-\sigma-Z(\theta)+S}-\theta
	\right].
	\]
	Because $\BU \in \Omega_\sigma$, we have $S = S(\BU) \ge \sigma$. It follows that 
	\[
	\Pi
	\ge
	\rho(1+\varepsilon v_{i,*})
	\underbrace{\left[
		\left(\frac{1-\Bv \cdot \Bv_*}{1-|\Bv|^2}\right)
		\left(\frac{1+\varepsilon v_i}{1+\varepsilon v_{i,*}}\right)h
		-
		\left(\frac{\sqrt{1-|\Bv_*|^2}}{\sqrt{1-|\Bv|^2}}\right)  	
		\left(\frac{1+\varepsilon v_i}{1+\varepsilon v_{i,*}}\right)h_*
		+\theta_* \mathrm{e}^{Z(\theta_*)-Z(\theta)}-\theta
		\right]}_{\Pi_*}.
	\]
Applying \Cref{lem:707} yields \(\Pi_* \geq 0\), which along with \(\rho(1+\varepsilon v_{i,*}) > 0\) completes the proof.
\end{proof}


\subsection{1D Entropy-Preserving Schemes}\label{sec:1Dscheme}

This section develops entropy-preserving schemes for the 1D relativistic Euler system with a general EOS \eqref{eq:175}. For clarity, we denote the 1D spatial coordinate by $x$. 
Consider a computational mesh $\{x_{\frac12} < x_{\frac32} < \dots < x_{N+\frac12}\}$ over the domain. Let $I_j = [x_{j-\frac12},x_{j+\frac12}]$ denote the $j$-th cell, with $\dx_j = x_{j+\frac12}-x_{j-\frac12}$ as the cell size and $x_j = (x_{j-\frac12}+x_{j+\frac12})/2$ as the cell center. Define $\oBU_j^n$ as the approximation of the cell-averaged solution $\BU(x,t)$ over $I_j$ at $t = t^n$, and set $\lambda_j = \Delta t / \Delta x_j$.

\subsubsection{First-Order 1D Scheme}

Consider the first-order scheme:
\begin{equation}\label{eq:1161}
	\oBU_j^{n+1} 
	= 
	\oBU_j^n 
	- 
	\lambda_j
	\Big[~
	\widehat \BF_1 ( \oBU_j^n , \oBU_{j+1}^n) 
	- 
	\widehat \BF_1 ( \oBU_{j-1}^n, \oBU_j^n )  
	~\Big],
\end{equation}
where $\widehat \BF_1 (\cdot,\cdot)$ is the LF-type numerical flux defined as
\begin{equation}\label{eq:LFflux1D} 
	\widehat \BF_1 ( \BU^- , \BU^+ ) 
	= 
	\frac12 
	\big(
	\BF_1 (  \BU^- ) + \BF_1 ( \BU^+ ) + \alpha \BU^{-}  - \alpha  \BU^{+}  
	\big),
\end{equation}
with $\alpha$ as an estimate of the maximum wave propagation speed. Here, $\alpha = c = 1$ is used, representing the speed of light in a vacuum, which serves as an upper bound for all wave speeds in special relativity.

\begin{theorem} \label{lem:1Dlocal}
	For the 1D relativistic Euler system with a general EOS \eqref{eq:175}, the first-order scheme \eqref{eq:1161} is (locally) entropy-preserving under the CFL condition:
	\begin{equation*}
		\alpha \Delta t \le \Delta x_j
		\qquad \forall j.
	\end{equation*}
\end{theorem}

\begin{proof}
	Assume $\oBU^n_{j-1}, \oBU^n_{j}, \oBU^n_{j+1} \in \Omega_\sigma$ for a given $\sigma$. We aim to prove $\oBU^{n+1}_{j} \in \Omega^{(2)}_\sigma = \Omega_\sigma$. Let the mass density and velocity of $\oBU^{n}_{j}$ be $\overline D^n_{j}$ and $\overline{v}^n_{j}$, respectively. 
	The first component of $\oBU_j^{n+1} $ is 
	\begin{equation*}
		\oD_j^{n+1} 
		=
		\left( 1 - \alpha \lambda_j \right) \oD_{j  }^n
		+\frac{\lambda_j}{2} 
		\left( \alpha + \ov^n_{j-1} \right) \oD_{j-1}^n
		+\frac{\lambda_j}{2}
		\left( \alpha - \ov^n_{j+1} \right) \oD_{j+1}^n > 0.
	\end{equation*}
	For any $\Bv_* \in \BALL$ and $\theta_* > 0$, we have
	\begin{align*}
		& \oBU_j^{n+1} \cdot \Bn_* + p_* \\
		= & 
		( 1-\alpha \lambda_j ) \oBU^n_j \cdot \Bn_*
		+\frac{\lambda_j}{2}
		\Big[
		\alpha \oBU^n_{j-1} \cdot \Bn_*
		+\alpha \oBU^n_{j+1} \cdot \Bn_*
		+\BF_1(\oBU^n_{j-1}) \cdot \Bn_*
		-\BF_1(\oBU^n_{j+1}) \cdot \Bn_*
		\Big]
		+p_*.
	\end{align*}
	Using \Cref{thm:475}, we estimate
	\begin{align*}
		\oBU_j^{n+1} \cdot \Bn_* + p_* 
		&\ge 
		( 1- \alpha \lambda_j ) (\oBU^n_j \cn + p_*)
		+ 
		\frac{\lambda_j}{2} (\alpha-1) (\oBU^n_{j-1} \cn + p_*)
		\\
		& \qquad 
		+
		\frac{\lambda_j}{2} (\alpha-1) (\oBU^n_{j+1} \cn + p_*)
		\ge 0.
	\end{align*}
	According to \Cref{lem:GQLnstar}, we have $\alpha \oBU_{j-1  }^n \cdot \tilde \Bn_*  + \BF_1(\oBU^n_{j-1}) \cdot \tilde \Bn_* > 0$ and $\alpha \oBU_{j+1  }^n \cdot \tilde \Bn_* -  \BF_1(\oBU^n_{j+1}) \cdot \tilde \Bn_* > 0$.
	This implies 
	\begin{equation*} 
		\oBU_j^{n+1}\cdot \tilde \Bn_*= 
		( 1-\alpha \lambda_j ) \oBU^n_j \cdot \tilde \Bn_*
		+\frac{\lambda_j}{2}
		\left[
		\alpha \oBU^n_{j-1} \cdot \tilde \Bn_*
		+\alpha \oBU^n_{j+1} \cdot \tilde \Bn_*
		+\BF_1(\oBU^n_{j-1}) \cdot \tilde \Bn_*
		-\BF_1(\oBU^n_{j+1}) \cdot \tilde \Bn_*
		\right] > 0. 
	\end{equation*}
	In conclusion, $\oBU_j^{n+1} \in \Omega_{\sigma}^{(2)} = \Omega_\sigma$ follows from \Cref{thm:1216}. The proof is complete.
	\end{proof}

\subsubsection{High-Order 1D Schemes}

This subsection develops high-order, provably entropy-preserving schemes for the 1D relativistic Euler system with a general EOS \eqref{eq:175}. For clarity, we consider the forward Euler time discretization, while discussions on high-order time discretization methods are deferred to \Cref{rmk:1050}.

The high-order finite volume schemes or the discrete equations for cell averages in DG schemes can be expressed in a unified form as
\begin{equation}\label{eq:930}
	\oBU_j^{n+1} 
	= 
	\oBU_j^n 
	- 
	\frac{\dt}{\dx_j} 
	\Big[~
	\widehat \BF_1 ( \BU_{j+\frac12}^{-} , \BU_{j+\frac12}^{+} ) 
	- 
	\widehat \BF_1 ( \BU_{j-\frac12}^{-} , \BU_{j-\frac12}^{+} )  
	~\Big],
\end{equation}
where $\widehat \BF_1 (\cdot,\cdot)$ is taken as the LF-type numerical flux defined in \eqref{eq:LFflux1D}, and $\BU_{j\pm\frac12}^{\mp}$ represent the limiting values of a high-order approximation polynomial $\BU^n_j(x)$ at cell endpoints: 
$
\BU_{j+\frac12}^{-}
=
\BU^n_j(x_{j+\frac12})
$
and
$
\BU_{j-\frac12}^{+}
=
\BU^n_j(x_{j-\frac12})
$.

Define $L:=\lceil \frac{k+3}2 \rceil$. 
Let the $L$-point Gauss--Lobatto quadrature nodes within cell $I_j$ be denoted as $\mathbb{X}_j = \{x_{\ell,j}^{\tt GL}\}_{\ell = 1}^{L}$, with corresponding weights $\{\omega_\ell^{\tt GL}\}_{\ell=1}^L$, satisfying $\sum_{\ell = 1}^L \omega_\ell^{\tt GL} = 1$. Employing this quadrature, the cell average $\oBU^n_j$ of $\BU^n_j(x)$ can be written as
\[
\oBU^n_j
=
\frac{1}{\dx_j} \int_{I_j} \BU^n_j(x) ~ \textrm{d} x
=
\sum_{\ell = 1}^L \omega^{\tt GL}_\ell \BU^n_j(x_{\ell,j}^{\tt GL})
\qquad 
\forall ~ j = 1,\dots,N.
\]
The following theorem provides a sufficient condition to ensure that the updated cell averages remain entropy-preserving.

\begin{theorem}
	If the high-order approximation polynomials $\BU^n_j(x)$ satisfy
	\begin{equation}\label{eq:960}
		\BU^n_j(x^{\tt GL}_{\ell,j}) \in \Omega_\sigma
		\quad
		\forall ~ \ell = 1,\dots,L,
		\;
		\forall ~ j = 1, \dots,N,
	\end{equation}
	then the updated cell averages $\oBU_j^{n+1}$ computed by the high-order scheme \eqref{eq:930} belong to $\Omega_\sigma$, provided the CFL condition:
	\begin{equation}\label{eq:968}
		\frac{\alpha \dt}{\dx_j} 
		\le 
		\omega^{\tt GL}_1 
		= 
		\omega^{\tt GL}_L
		= \frac{1}{L(L-1)} 
		\qquad \forall ~ j.
	\end{equation}
\end{theorem}

\begin{proof}
	Define $\tBU^{\ell}_{j} := \BU_j^n(x^{\tt GL}_{\ell,j})$ and 
	$\tp^{\ell}_{j} := \tBU^{\ell}_{j} \cn + p_* \ge 0$ for all $\ell$ and $j$. For any $\Bv_* \in \BALL$ and $\theta_* > 0$, we have 
	\begin{align*}
		\oBU^{n+1}_j \cn + p_* 
		=  & 
		\sum_{\ell=1}^{L} \omega_\ell^{\tt GL} 
		\left(\tBU^{\ell}_j \cn + p_*\right)
		+
		\frac{\alpha\lambda_j}{2}
		\left(
		\tBU^{L}_{j-1}\cn
		-
		\tBU^{1}_{j  }\cn
		-
		\tBU^{L}_{j  }\cn
		+
		\tBU^{1}_{j+1}\cn
		\right)
		\\
		& + 
		\frac{\lambda_j}{2}
		\left(
		\BF_1(\tBU^{L}_{j-1})\cn
		+
		\BF_1(\tBU^{1}_{j  })\cn
		-
		\BF_1(\tBU^{L}_{j  })\cn
		-
		\BF_1(\tBU^{1}_{j+1})\cn
		\right) \\
		\stackrel{\eqref{eq:1416}}{\ge}   &
		\sum_{\ell=1}^{L} \omega_\ell^{\tt GL} \tp^{\ell}_j
		+
		\frac{\alpha\lambda_j}{2}
		\left(
		\tp^{L}_{j-1}
		-
		\tp^{1}_{j  }
		-
		\tp^{L}_{j  }
		+
		\tp^{1}_{j+1}
		\right)
		+
		\frac{\alpha\lambda_j}{2}
		\left(
		-\tp^{L}_{j-1}
		-\tp^{1}_{j  }
		-\tp^{L}_{j  }
		-\tp^{1}_{j+1}
		\right) \\
		= &
		%
		\left(
		\omega_1^{\tt GL}-\alpha\lambda_j
		\right)\tp^1_{j}
		+
		\left(
		\omega_L^{\tt GL}-\alpha\lambda_j
		\right)\tp^L_{j}
		+
		\sum_{\ell=2}^{L-1}	\omega_{\ell}^{\tt GL} \tp^{\ell}_{j} \stackrel{\eqref{eq:968}}{\ge} 0.
	\end{align*}
	Similarly, using \Cref{lem:GQLnstar} and \eqref{eq:968}, we derive   
	\begin{align*}
		\oBU^{n+1}_j \cdot \tilde\Bn_*&=\sum_{\ell=1}^{L} \omega_\ell^{\tt GL} 
		\tBU^{\ell}_j \cdot \tilde\Bn_*+\frac{\lambda_j}{2}\left(\alpha\tBU^{L}_{j-1}+{\BF_1(\tBU^{L}_{j-1})}{}\right)\cdot \tilde\Bn_*-\frac{\alpha\lambda_j}{2}\left( \alpha \tBU^{1}_{j  }-{\BF_1(\tBU^{1}_{j  })}{}\right)\cdot \tilde\Bn_*\\
		&\qquad  -\frac{\lambda_j}{2}\left( \alpha  \tBU^{L}_{j  }+{\BF_1(\tBU^{L}_{j  })}{}
		\right)\cdot \tilde\Bn_*+\frac{\lambda_j}{2}\left( \alpha  \tBU^{1}_{j+1}-{\BF_1(\tBU^{1}_{j+1})}{}\right)\cdot \tilde\Bn_*>0.
	\end{align*}
	Denote the mass density and velocity of $\tBU^\ell_j$ by $\tD^\ell_j$ and $\tv^\ell_j$. Then 
	\begin{align*}
		& \oD^{n+1}_j 
		= 
		\sum_{\ell=1}^{L} \omega^{\tt GL}_{\ell} \tD^\ell_j
		+
		\frac{\alpha\lambda_j}{2}
		\left(
		\tD^L_{j-1}
		-\tD^1_{j  }
		-\tD^L_{j  }
		+\tD^1_{j+1}
		\right)
		+
		\frac{\lambda_j}{2} 
		\left(
		\tv^L_{j-1} \tD^L_{j-1}
		+\tv^1_{j  } \tD^1_{j  }
		-\tv^L_{j  } \tD^L_{j  }
		-\tv^1_{j+1} \tD^1_{j+1}
		\right) \\
		& =  
		\frac{\lambda_j}{2} (\alpha+\tv^L_{j-1}) \tD^L_{j-1}
		+
		\left(
		\omega_1^{\tt GL}-\frac{\lambda_j}{2}(\alpha-\tv^1_j)
		\right) \tD^1_j
		+
		\sum_{\ell=2}^{L-1} \omega^{\tt GL}_{\ell} \tD^\ell_j 
		\\
		& \qquad +
		\left(
		\omega_L^{\tt GL}-\frac{\lambda_j}{2}(\alpha+\tv^L_j)
		\right) \tD^L_j
		+ 
		\frac{\lambda_j}{2} (\alpha+\tv^1_{j+1}) \tD^1_{j+1}
		\\
		& \stackrel{\eqref{eq:968}}{\ge}
		\frac{\lambda_j}{2} (\alpha+\tv^L_{j-1}) \tD^L_{j-1}
		+
		\frac{\lambda_j}{2}
		\left(
		\alpha+\tv^1_j
		\right) \tD^1_j
		+
		\sum_{\ell=2}^{L-1} \omega^{\tt GL}_{\ell} \tD^\ell_j 
		+
		\frac{\lambda_j}{2}
		\left(
		\alpha-\tv^L_j
		\right) \tD^L_j
		+ 
		\frac{\lambda_j}{2} (\alpha+\tv^1_{j+1}) \tD^1_{j+1}
		> 0.
	\end{align*}
	Therefore, according to \Cref{thm:1216}, we obtain $\oBU^{n+1}_j \in \Omega_\sigma^{(2)} = \Omega_\sigma$ and finish the proof.
\end{proof}

\subsubsection{Limiters} \label{sec:limiter}

For conventional high-order finite volume or DG schemes, the approximation polynomial $\BU^n_j(x)$ may not automatically satisfy the condition \eqref{eq:960}.

To address this, we introduce a limiting procedure ${\mathcal I} (\BU^n_j(x);\sigma)$, which modifies $\BU^n_j(x)$ so that the values of the modified polynomial $\BU^{\tt III}_j(x)$ at the Gauss--Lobatto nodes are guaranteed to lie within $\Omega_\sigma$. The procedure consists of the following steps:

	{\bf Step 1:}	Ensure positivity of $D$ by setting
		\begin{equation*}
			\BU^{\tt I}_j(x)
			:=
			\left(
			\theta_1 (D_j(x)-\overline{D}^n_j) + \overline{D}_j^n,
			\Bm^n_j(x),
			E^n_j(x)
			\right)^\top \quad 
		{\rm with} \quad 
			\theta_1
			=
			\min
			\left\{
			1,
			\frac
			{\overline{D}_j^n - \varepsilon_1}
			{
				\overline{D}_j^n
				-
				\min\limits_{x \in \mathbb{X}_j} D^n_j(x)
			}
			\right\},
		\end{equation*}
		where $\mathbb{X}_j =\{x^{\tt GL}_{\ell,j}\}_{1\le \ell \le L}$ denotes the Gauss--Lobatto nodes in $I_j$.
		
		{\bf Step 2:} Ensure positivity of $q({\bf U}) := E - \sqrt{D^2 + |\Bm|^2}$ by setting
		\begin{equation*}
			\BU^{\tt II}_j(x)
			:=
			\theta_2 \left( \BU^{\tt I}_j(x)-\oBU_j^n \right)+\oBU_j^n
		\quad 
		{\rm with} \quad 
			\theta_2 
			:= 
			\min
			\left\{
			1,
			\left|
			\frac
			{q(\oBU^n_j)-\varepsilon_2}
			{q(\oBU^n_j)-\min\limits_{x\in\mathbb{X}_j} q(\BU^{\tt I}_j(x))}
			\right|
			\right\}.
		\end{equation*}
		
		{\bf Step 3:} Ensure positivity of $g_\sigma (\BU) := D(S(\BU)-\sigma)$ by setting
		\begin{equation*}
			\BU^{\tt III}_j(x)
			:=
			\theta_3 \left( \BU^{\tt II}_j(x)-\oBU_j^n \right)+\oBU_j^n
		\quad 
		{\rm with} \quad 
			\theta_3 
			:= 
			\min
			\left\{
			1,
			\left|
			\frac
			{g_\sigma(\oBU^n_j)}
			{g_\sigma(\oBU^n_j)-\min\limits_{x\in\mathbb{X}_j} g_\sigma(\BU^{\tt II}_j(x))}
			\right|
			\right\}.
		\end{equation*}

According to \eqref{1stequivNew}, the values of the limited polynomial $\BU^{\tt III}_j(x)$ at the Gauss--Lobatto nodes are guaranteed to lie within $\Omega_\sigma$. Here, $\varepsilon_1$ and $\varepsilon_2$ are small parameters approximating the lower bounds of density $D(\Bx,t^n)$ and $q(\BU(\Bx,t^n))$ in $I_j$. To ensure robustness against round-off errors, one may set $\varepsilon_1 = \min\{10^{-13},\overline{D}^n_j\}$ and $\varepsilon_2 = \min\{10^{-13}, q(\oBU^n_j)\}$ in numerical experiments. This limiting procedure is inspired by the bound-preserving and physical-constraint-preserving limiters developed in \cite{zhang2010b, zhang2012minimum, WuTang2015, QinShu2016, WuTang2017ApJS, Wu2017, JIANG2018}.

\begin{remark} \label{rmk:1050}
	The above entropy-preserving analysis focuses on the first-order forward Euler time discretization \eqref{eq:930}.  
	For high-order time discretization methods, one can employ a strong-stability-preserving (SSP) Runge--Kutta or multistep (MS) method \cite{GottliebShuTadmor2001}, which can be expressed as a convex combination of formal forward Euler steps.
	Thanks to the convexity of $\Omega_\sigma$, the entropy-preserving property of the resulting high-order schemes is preserved. For instance, if the third-order SSP MS method is used, the entropy-preserving high-order scheme is given by
	\begin{equation}\label{eq:2469}
		\widehat{\bf U}^{n+1}_j
		=
		\frac{16}{27}\left({\bf U}^n+3\Delta t{\bf L}({\bf U}^n)\right)+\frac{11}{27}\left({\bf U}^{n-3}+\frac{12}{11}\Delta t{\bf L}({\bf U}^{n-3})\right), \qquad 	{\bf U}^{n+1}_j
		=
		{\mathcal I} (\widehat{\bf U}^{n+1}_j;\sigma^{n+1}_j).
	\end{equation}
	where  $\sigma_j^{n+1}$ represents an estimate of the local or global entropy lower bound in $I_j$ at time level $t^{n+1}$, and $\mathcal{I}$ denotes the limiting procedure presented above. The local estimates of $\sigma_j^{n+1}$ will be further explored in \Cref{sec:6}.
\end{remark}


\subsection{Multidimensional Entropy-Preserving Schemes}\label{sec:2Dschemes}

This subsection extends the 1D entropy-preserving numerical methods proposed in \Cref{sec:1Dscheme} to multiple dimensions for the relativistic Euler equations. For clarity, we consider a 2D mesh $\mathcal{T}$ consisting of polygonal cells, denoted by $K$. Assume that a cell $K$ has $N_K$ edges, and let $K_j$, $j = 1,\dots,N_K$, denote the edge-neighboring cells of $K$ with the shared edge denoted by $\CE_K^{(j)}$.
The unit normal vector of $\CE_K^{(j)}$, pointing from $K$ to $K_j$, is denoted by $\Bxi_K^{(j)} = \big(\xi_{1,K}^{(j)},\dots,\xi_{d,K}^{(j)} \big)^\top$.
We use $|K|$ and $|\CE_K^{(j)}|$ to represent the area of $K$ and the length of $\CE_K^{(j)}$, respectively.
Let $\oBU_K^n$ denote the approximation to the cell average over $K$ at $t = t^n$.

\subsubsection{First-Order Scheme}

Consider the first-order scheme for the multidimensional relativistic Euler system:
\begin{equation}\label{eq:1091}
	\oBU_{K}^{n+1}
	= 
	\oBU_{K}^{n} 
	- 
	\frac{\Delta t}{|K|} 
	\sum_{j=1}^{N_K}
	\CEj
	\widehat \BF 
	\big(  
	\oBUk,
	\oBUj;
	\Bxij
	\big),
\end{equation}
with the LF numerical flux $\widehat {\bf F}$ defined by
\begin{equation}\label{eq:1102}
	\widehat\BF 
	\left( 
	\BU^{-}, \BU^{+}; \Bxi 
	\right) 
	= 
	\frac12 
	\Big( 
	{\bf F} ({\bf U}^-) \Bxi  + 
	{\bf F} ({\bf U}^+) \Bxi 
	+\alpha \BU^{-}
	-\alpha \BU^{+} 
	\Big),
\end{equation}
where $\BF = (\BF_1,\dots,\BF_d) \in \mathbb{R}^{(d+2) \times d}$. 

We first introduce several lemmas, which will be useful in proving that the first-order scheme \eqref{eq:1091} is entropy-preserving. 

\begin{lemma} \label{lem:685}
	For any $\BU \in \Omega_\sigma$, if $\BQ \in \mathbb{R}^{d\times d}$ is a rotation matrix and $\dBQ = \diag \{1,\BQ,1\}\in\mathbb{R}^{(d+2)\times(d+2)}$, then $\dBQ \BU \in \Omega_\sigma$ and 
	\begin{equation*}
		\BF(\BU) (\BQ^{-1} \, \Be_i)
		= 
		\dBQ^{-1} \BF_i(\dBQ \BU),
		\quad i = 1,\dots,d.
	\end{equation*}
\end{lemma}

\begin{proof}
	Since $D(\dBQ \BU) = D(\BU)$, $|\Bm(\dBQ \BU)| = |\BQ \Bm (\BU)| = |\Bm(\BU)|$, and $E(\dBQ \BU) = E(\BU)$, it follows that $q(\dBQ \BU) = q(\BU) > 0$, where the function $q$ is defined in \Cref{1stequivG}. Similarly, we obtain that $S(\dBQ \BU) = S(\BU) \ge \sigma$, so $\dBQ \BU \in \Omega_\sigma$.
	
	Denoting $\Bv' = \BQ \Bv = (v'_1,\dots,v'_d)^\top$ and unit vector $\Bxi = \BQ^{-1} \Be_i$, we write
	\[
	\dBQ \BU
	=
	\begin{pmatrix}
		\rho \gamma \\ \rho h \gamma^2 \Bv' \\ \rho h \gamma^2 - p 
	\end{pmatrix},
	\quad
	\BF_i(\dBQ \BU)
	=
	\begin{pmatrix}
		\rho \gamma v'_i \\ \rho h \gamma^2 v'_i \Bv' + p \Be_i \\ \rho h \gamma^2 v'_i
	\end{pmatrix}.
	\]
	It follows that
	\[
	\dBQ^{-1} \BF_i(\dBQ \BU)
	=
	\begin{pmatrix}
		\rho \gamma v'_i \\ \rho \gamma^2 v'_i \Bv + p \Bxi \\ \rho h \gamma^2 v'_i
	\end{pmatrix}.
	\]
	Since $\BQ$ is orthogonal, we have 
	\[
	\Bxi \cdot \Bv = (\BQ^{-1}\Be_i)^\top \Bv = \Be_i^\top \BQ \Bv = \Be_i \cdot \Bv' = v'_i.
	\]
	Thus, we conclude:
	\begin{align*}
		\BF(\BU) (\BQ^{-1} \Be_i)
		=
		\sum_{i=1}^d \xi_i \BF_i(\BU) 
		= 
		\begin{pmatrix}
			\rho \gamma v'_i \\ \rho \gamma^2 v'_i \Bv + p \Bxi \\ \rho h \gamma^2 v'_i
		\end{pmatrix}
		=
		\dBQ^{-1} \BF_i(\dBQ \BU).
	\end{align*}
	This completes the proof. 
\end{proof}

\begin{lemma} \label{thm:685}
	If $\BU \in \Omega_\sigma$, then for any $\Bv_* \in \mathbb B_1(\bf 0)$, $\theta_* > 0$, and unit vector $\Bxi \in \mathbb{R}^d$, it holds that
	\begin{equation}\label{eq:1845}
		(\BF(\BU)\Bxi)\cdot \Bn_*
		\ge
		-\alpha(\BU\cdot\Bn_*+p_*)
		-(\Bxi\cdot \Bv_*)p_*.
	\end{equation}
\end{lemma}

\begin{proof}
	There exists a rotation matrix $\BQ \in \mathbb{R}^{d\times d}$ such that $\Bxi = \BQ^{-1} \Be_1$. Let $\dBQ = \diag \{1,\BQ,1\}\in\mathbb{R}^{(d+2)\times(d+2)}$, $\widehat{\BU} = \dBQ\BU \in \Omega_\sigma$, $\hat \Bv_* = \BQ \Bv_* \in \BALL$, and $\hat\Bn_* = (-h_*\sqrt{1-|\hat\Bv_*|^2},-\hat\Bv_*^\top,1)^\top$. Then Lemma \ref{lem:685} implies 
	$
	(\BF(\BU)\Bxi)\cdot \Bn_*
	= \BF_1(\widehat\BU)\cdot \hat\Bn_*.
	$ 
	Using Lemma \ref{thm:475}, we obtain
	\[
	(\BF(\BU)\Bxi)\cdot \Bn_*
	\ge -\alpha (\BU\cdot\Bn_* + p_*) - (\Bxi\cdot \Bv_*)p_*,
	\]
	which completes the proof.
\end{proof}

\begin{remark}
	\Cref{thm:475} is a special case of Lemma \ref{thm:685} with $\Bxi = \Be_i$ and $\Bxi = -\Be_i$.
\end{remark}

\begin{lemma}
	\label{lem:2DGQLnstar}
	If $\mathbf{U} \in \mathcal{G}$, then for any unit vector $\Bxi \in \mathbb{R}^d$, the following inequality holds:
	\begin{equation*}
		\big(\beta\BU \pm \BF(\BU)\Bxi\big) \cdot \tilde \Bn_* > 0,  
		\quad \text{for any} \quad \beta \geq c = 1,
	\end{equation*}
	where $\tilde \Bn_*$ is defined in \eqref{tildenstar}.
\end{lemma}

\begin{proof}
	Utilizing Lemma \ref{lem:GQLnstar} and the rotational invariance property (Lemma \ref{lem:685}), the proof follows similarly to that of Theorem \ref{thm:685} and is therefore omitted.
\end{proof}


\begin{theorem} \label{thm:2170}
	For the multidimensional relativistic Euler system, the first-order scheme \eqref{eq:1091} is (locally) entropy-preserving under the CFL condition
	\begin{equation}\label{eq:1134}
		\frac{\alpha \dt} { 2 |K|} 
		\sum_{j=1}^{N_K} \left|\CE_{K}^{j}\right|  \le 1.
	\end{equation}
\end{theorem}

\begin{proof}
	Assume $\oBU^n_K \in \Omega_\sigma$ and $\oBU^n_{K_j} \in \Omega_\sigma$ for all $j = 1,\dots,N_K$. We will show that $\oBU^{n+1}_K \in \Omega_\sigma$.
	
	For convenience, denote the mass density and velocity of $\oBU^{n}_{K}$ by $\oDk$ and $\oBv^n_{K}$. 
	Note that $ \sum_{j=1}^{N_K} \CEj \Bxij = \mathbf{0} $ and
	$ \oBvj \cdot \Bxij < \alpha =c $ for any $j$.
	Then, we have
	\begin{align*}
		\oD_{K}^{n+1}
		& = 
		\oD_{K}^{n} 
		- 
		\frac{\Delta t}{2|K|} \sum_{j=1}^{N_K} \CEj
		\left(
		\oBvk \cdot \Bxij \, \oDk + 
		\oBvj \cdot \Bxij \, \oDj
		+\alpha \oDk
		-\alpha \oDj
		\right) \\
		& =
		\Bigg(
		1
		-
		\frac{\alpha \Delta t}{2|K|} \sum_{j=1}^{N_K} \CEj
		\Bigg)
		\oDk
		+ 
		\frac{\Delta t}{2|K|} \sum_{j=1}^{N_K} \CEj
		\Big(\alpha - \oBvj \cdot \Bxij \Big) \oDj 
		> 0.
	\end{align*}
	Additionally,
	\begin{align*}
		\oBU^{n+1}_K\cdot \Bn_* + p_* 
		&	= 
		\oBU^{n}_K \cdot \Bn_* + p_*
		-\frac{\Delta t}{2|K|} \sum_{j=1}^{N_K} \CEj 
		\bigg(
		\Bn_*^\top {\bf F} (\oBUk) \Bxij   
		+\Bn_*^\top {\bf F} (\oBUj) \Bxij 
		\\
		& \qquad 
		+\alpha \big(\oBUk\cn+p_*\big)
		-\alpha \big(\oBUj\cn+p_*\big)
		\bigg)
		\\
		&= 
		\left( 
		1- \frac{\alpha\Delta t}{2|K|} \sum_{j=1}^{N_K} \CEj 
		\right)
		\big( \oBU^{n}_K \cdot \Bn_* + p_* \big)
		-
		\frac{\Delta t}{2|K|} \Bn_*^\top {\bf F} (\oBUk)
		\left(
		\sum_{j=1}^{N_K} \CEj \Bxij 
		\right) \\
		& \qquad +
		\frac{\Delta t}{2|K|} \sum_{j=1}^{N_K} \CEj
		\left[
		\Bn_*^\top {\bf F} (\oBUj) (-\Bxij) 
		+
		\alpha \big(\oBUj\cn+p_*\big)
		\right] 
		\\
		&
		\stackrel{\eqref{eq:1845}}{\ge} 
		\frac{p_* \Delta t}{2|K|} \Bv_* \cdot
		\left(
		\sum_{j=1}^{N_K} \CEj \Bxij
		\right) 
		\stackrel{\eqref{eq:1134}}{\ge} 0.
	\end{align*}
	Similarly, using \Cref{lem:2DGQLnstar} yields $\oBU^{n+1}_K \cdot \tilde\Bn_* > 0$. 
	Therefore, we obtain that $\oBU^{n+1}_{K} \in \Omega_\sigma$ by \Cref{thm:1216}. The proof is completed.
\end{proof}

\subsubsection{High-Order Schemes}

We are now in a position to discuss high-order entropy-preserving schemes for the multidimensional relativistic Euler equations. 
At time $t = t^n$, assume that a polynomial $\BU^n_K(\Bx)$ of degree at most $k$ is available in each cell $K$, either reconstructed in a finite volume method or evolved in a DG method.
The cell average of $\BU^n_K(\Bx)$ over $K$ is exactly $\oBU^n_{K}$. 
If the forward Euler method is used for time discretization, the cell averages are updated using the following scheme:
\begin{equation}\label{eq:1487}
	\oBU_{K}^{n+1}  
	= 
	\oBU_{K}^{n} 
	- 
	\frac{\Delta t}{|K|} 
	\sum_{j=1}^{N_K}  |\CE_K^{(j)}| \widehat \BF_K^{(j)},
\end{equation}
where
\[
\widehat \BF_K^{(j)} 
= 
\sum_{q=1}^Q
\omega^{\tt G}_q \,
\widehat\BF 
\left( 
\BU_K^n ( {\bm x}_K^{(j,q)} ), 
\BU_{K_j}^n ( {\bm x}_K^{(j,q)} ); 
{\bm \xi}^{(j)}_{K} 
\right).
\] 
Here, $\widehat{\BF}$ denotes the LF numerical flux \eqref{eq:1102}, and $\{ \Bx_K^{(j,q)}, \omega^{\tt G}_q \}_{q = 1}^{Q}$ are the nodes and weights of a $Q$-point Gauss-type quadrature on the edge ${\CE}_K^{(j)}$, exact for polynomials of degree up to $k$. 

\begin{definition}[Feasible CAD]
	A multidimensional CAD    
	\begin{equation}\label{eq:1514}
		\oBU^n_K
		=
		\frac{1}{|K|} \int_K \BU^n_K(\Bx) ~ \mathrm{d} \Bx
		=
		\sum_{j=1}^{N_K}
		\omega^{(j)}_K
		\sum_{q=1}^{Q}
		\omega^{\tt G}_{q}
		\BU^n_K(\Bx^{(j,q)}_K)
		+
		\sum_{s=1}^{S}
		\hat\omega^{(s)}_K
		\BU^n_K(\hat\Bx^{(s)}_K),
	\end{equation}
	is said to be {\em feasible} for polynomial space $\mathbb{P}^k$, if it simultaneously satisfies 
	the following three conditions:
	\begin{enumerate}[label=(\roman*)]
		\item  The identity \eqref{eq:1514} holds exactly for all polynomials in $\mathbb{P}^k$.
		\item  The {\em boundary weights} ${\omega}^{(j)}_{K}$ and the {\em internal weights} ${\hat\omega}_K^{(s)}$ are all nonnegative, and their summation equals one. 
		\item  The {\em internal node} set $\mathbb I_K := \{ \hat\Bx^{(s)}_K \}_{s=1}^{S} \subset K$. 
	\end{enumerate}
\end{definition}

For convenience, define the set of boundary nodes $\mathbb B_K := \{ \Bx^{(j,q)}_K \}_{1\le j \le N_K, 1 \le q \le Q}$ and the set of CAD nodes $\mathbb X_K := \mathbb I_K \cup \mathbb B_K$.

\begin{theorem}\label{thm:high2D} 
	If $\BU^n_K(\Bx) \in \Omega_\sigma$ for all $\Bx \in \mathbb{X}_K$ and $K \in \mathcal{T}$, 
	then the solution $\oBU^{n+1}_K$ computed by the high-order scheme \eqref{eq:1487} belongs to $\Omega_\sigma$ for all $K \in \mathcal{T}$ under the CFL condition
	\begin{equation}\label{eq:2095}
		\frac{\alpha \dt}{|K|}
		\le
		\min_{1\le j\le N_K}\frac{\omega_K^{(j)}}{|\CE^{(j)}_K|}
		\quad
		\forall K \in \mathcal{T}.
	\end{equation}
\end{theorem}

\begin{proof}
	Employing the CAD \eqref{eq:1514} and \Cref{thm:685} gives
	\begin{align*}
		\oBU^{n+1}_K \cn+p_* 
		= &
		\sum_{j = 1}^{N_K} \sum_{q = 1}^{Q}
		\omega^{\tt G}_q
		\left(
		\omega_K^{(j)} - \frac{\alpha\dt}{|K|}\CEj
		\right)
		(\BU^{n}_K(\Bx^{(j,q)}_K)\cn+p_*)
		+
		\sum_{s=1}^{S}
		\hat\omega^{(s)}_K
		(\BU^n_K(\hat\Bx^{(s)}_K)\cn+p_*) \\
		& +
		\frac{\dt}{2|K|}
		\sum_{j = 1}^{N_K} \CEj
		\sum_{q = 1}^{Q} \omega_q^{\tt G}
		\left[
		\alpha \BU^n_{K}(\Bx^{(j,q)}_K)\cn
		+\alpha p_*
		-\big(\BF(\BU^n_K(\Bx^{(j,q)}_K)) \Bxij \big)\cn
		\right] \\
		& +
		\frac{\dt}{2|K|}
		\sum_{j = 1}^{N_K} \CEj
		\sum_{q = 1}^{Q} \omega_q^{\tt G}
		\left[
		\alpha \BU^n_{K_j}(\Bx^{(j,q)}_K) \cn
		+\alpha p_*
		-\big( \BF(\BU^n_{K_j}(\Bx^{(j,q)}_K)) \Bxij \big) \cn 
		\right] \\
		\ge & 
		\frac{\dt}{|K|}
		\sum_{j = 1}^{N_K} \CEj
		\sum_{q = 1}^{Q} \omega_q^{\tt G}
		\left(
		p_* \Bv_* \cdot \Bxij 
		\right)
		=
		\frac{\dt}{|K|}
		\sum_{q = 1}^{Q} \omega_q^{\tt G}
		\, p_* \Bv_* \cdot
		\left(
		\sum_{j = 1}^{N_K} \CEj
		\Bxij 
		\right) = 0.
	\end{align*}
	Similarly, using \Cref{lem:2DGQLnstar}, we obtain $\oBU^{n+1}_K \cdot \tilde\Bn_* > 0$. 
	Denote the mass density of $\BU^n_{K}(\Bx^{(j,q)}_K)$ and $\BU^n_K(\hat\Bx^{(s)}_K)$ by $D^{(j,q)}_K$ and $\widehat{D}^{(s)}_K$, and their velocities by $\Bv^{(j,q)}_K$ and $\widehat{\Bv}^{(s)}_K$, respectively. Then, we rewrite the scheme for $\oD^{n+1}_K$ as
	\begin{align*}
	\oD^{n+1}_K
	= & 
	\sum_{j = 1}^{N_K} \sum_{q = 1}^{Q}
	\omega^{\tt G}_q
	\left(
	\omega_K^{(j)} - \frac{\alpha\dt}{|K|}\CEj
	\right)
	D^{(j,q)}_K
	+
	\sum_{s=1}^{S}
	\widehat\omega^{(s)}_K
	\widehat D^{(s)}_K \\
	& +
	\frac{\dt}{2|K|}
	\sum_{j = 1}^{N_K} \CEj
	\sum_{q = 1}^{Q} \omega_q^{\tt G}
	\left[
	\left(
	\alpha - \Bv_K^{(j,q)} \cdot \Bxij
	\right)
	D^{(j,q)}_K
	+
	\left(
	\alpha - \Bv_{K_j}^{(j,q)} \cdot \Bxij
	\right)
	D^{(j,q)}_{K_j}
	\right] >0.
\end{align*}
	Thus, $\oBU^{n+1}_{K} \in \Omega_\sigma$ by \Cref{thm:1216}, completing the proof.
\end{proof}

\begin{remark}
	The theoretical CFL condition \eqref{eq:2095} depends on the boundary weights $\omega^{(j)}_K$ of the adopted CAD. 
	A feasible CAD is not unique, and an {\em optimal} CAD (OCAD) should maximize $\bar \omega := \min_{j} \omega_K^{(j)}/|\CE^{(j)}_K|$, leading to a less restrictive CFL condition.
	Recent studies \cite{cui2024optimal,cui20203,ding2025} have established OCAD frameworks for Cartesian and unstructured triangular meshes.
\end{remark}

To ensure $\BU^n_K(\Bx) \in \Omega_\sigma$, a simple limiter, similar to the 1D counterpart in \Cref{sec:limiter}, can be applied with $\mathbb{X}_j$ replaced by $\mathbb{X}_K$.


\section{Priori Estimations of Local Entropy Bounds}	\label{sec:6}
In the previous Section \ref{sec:scheme}, we have established a rigorous theoretical framework for designing provably entropy-preserving high-order schemes that enforce the constraint $S(\BU) \geq \sigma$. The lower bound of the specific entropy, $\sigma$, can be determined either globally ($\sigma = S^{\rm G}_{\min}$) or locally ($\sigma = S^{\rm L}_{\min}$), leading to globally or locally entropy-preserving schemes, respectively. 
Estimating local lower bounds for entropy is a challenging task. 
In this section, we discuss a priori estimations for the specific entropy lower bounds and propose two new approaches tailored for 1D and 2D Cartesian meshes, respectively. 
The design of such an estimation is crucial for both the robustness and accuracy of entropy-preserving high-order schemes. 
On one hand, if the estimated lower bound is too small, spurious oscillations in the numerical results would not be effectively controlled by the entropy-preserving limiter, leading to nonphysical solutions. 
On the other hand, if the estimated lower bound is too large, the entropy-preserving limiter may become overly restrictive, thereby compromising the accuracy of the high-order schemes. 
For clarity, we first discuss the estimation in the case of forward Euler time discretization, while the treatment for high-order time discretization methods is deferred to Remark \ref{rmk:3270}.

\subsection{Estimation on 1D Meshes}
Entropy bound estimation has been previously studied in the context of non-relativistic Euler equations. Lv and Ihme \cite{LV2015715} introduced the following formula for estimating the local minimum entropy in cell $I_j$ at time level $t^{n+1}$:
\begin{equation}\label{eq:3133}
	(\hat{S}^{\rm L}_{\min})^{n+1}_{j,\texttt{LI1}} = \min \left\{
	S(\BU_{j-1}^n(x_{j-\frac12})),
	S(\BU_{j+1}^n(x_{j+\frac12})),
	\min_{q=1,\dots,Q} S(\BU_{j}^n(x_{q,j}^{\tt G}))
	\right\},
\end{equation}
where $\{ x_{q,j}^{\tt G} \}_{q = 1}^{Q}$ denotes the nodes of a $Q$-point Gauss-type quadrature in $I_j$ that is exact for all polynomials in $\mathbb{P}^k$. 
Lv and Ihme found that the above formula \eqref{eq:3133} tends to overestimate the entropy lower bound, which may degrade the high-order accuracy. To mitigate this issue, they further relaxed \eqref{eq:3133} by introducing the following correction:
\begin{equation}\label{eq:3141}
	(\hat{S}^{\rm L}_{\min})^{n+1}_{j,\texttt{LI2}} = 
	\min
	\left\{
	(\hat{S}^{\rm L}_{\min})^{n+1}_{j,\texttt{LI1}},
	(S_{\triangledown}^{\tt G})_j^n
	-
	\theta_{j}^{\tt LI}
	\Big(
	(S_{\vartriangle}^{\tt G})_j^n
	-
	(S_{\triangledown}^{\tt G})_j^n
	\Big)
	\right\},
	\quad
	\theta_{j}^{\tt LI} =
	\frac
	{
		\min\limits_{q = 1,\dots,Q, q \neq q_\triangledown} 
		|x^{\tt G}_{q,j}-x^{\tt G}_{q_\triangledown,j}|
	}
	{
		|x^{\tt G}_{q_\vartriangle,j}-x^{\tt G}_{q_\triangledown,j}|
	}.
\end{equation}
Here, $(S_{\vartriangle}^{\tt G})_j^n$ (resp. $(S_{\triangledown}^{\tt G})_j^n$) denotes the maximum (resp. minimum) among $\{S(\BU(x^{\tt G}_{q,j}))\}_{q=1}^Q$, with corresponding indices $q_\vartriangle$ and $q_\triangledown$. 
However, Lv and Ihme observed that the formula \eqref{eq:3141} may result in an underestimation, especially in the vicinity of discontinuities. To address this issue, they modified \eqref{eq:3141} by incorporating information from neighboring cells:
\begin{equation} \label{LV Final star}
	(\hat{S}^{\rm L}_{\min})^{n+1}_{j,\texttt{LI3}} = 
	\max
	\left\{
	(\hat{S}^{\rm L}_{\min})^{n+1}_{j,\texttt{LI2}},
	(\hat{S}^{\rm L}_{\min})^{n}_{j-1,\texttt{LI1}},
	(\hat{S}^{\rm L}_{\min})^{n}_{j,\texttt{LI1}},
	(\hat{S}^{\rm L}_{\min})^{n}_{j+1,\texttt{LI1}}
	\right\}.
\end{equation}
Recently, Ching, Johnson, and Kercher \cite{ching2024positivity, ching2024positivity2, ching2025positivity} proposed an alternative relaxation of \eqref{eq:3133} by including Gauss-type quadrature nodes from neighboring cells:
\begin{equation} \label{CJK}
	(\hat{S}^{\rm L}_{\min})^{n+1}_{j,\texttt{CJK}}
	=
	\min\left\{
	(S_{\triangledown}^{\tt G})_{j-1}^n,
	(S_{\triangledown}^{\tt G})_j^n,
	(S_{\triangledown}^{\tt G})_{j+1}^n
	\right\}.
\end{equation}

{\bf New Estimation Approach I}: In this paper, we propose the following new approach for estimating the local entropy lower bound in cell $I_j$ at time level $t^{n+1}$: 
\begin{equation}\label{eq:3089}
	(\hat{S}^{\rm L}_{\min})^{n+1}_{j} = 
	\hat{S}^{\rm G}_{\min} \vee \min \left\{
	\min_{I_{j-1}} \mathcal{P}^{\tt 1D}_{I_{j-1}} S(\BU_{j-1}^n(x)),
	\min_{I_j} \mathcal{P}^{\tt 1D}_{I_j} S(\BU_j^n(x)),
	\min_{I_{j+1}} \mathcal{P}^{\tt 1D}_{I_{j+1}} S(\BU_{j+1}^n(x))
	\right\},
\end{equation}
where $\vee$ denotes taking the maximum of two numbers, $\mathcal{P}^{\tt 1D}_{I_j}$ denotes a projection from any continuous function $f$ defined on $I_j$ to a quadratic polynomial $p \in \mathbb{P}^2$, uniquely determined by the following conditions
\[
p(x_{j-\frac12}) = f(x_{j-\frac12}), \quad
p(x_{j}) = f(x_{j}), \quad
p(x_{j+\frac12}) = f(x_{j+\frac12}),
\]
and $\hat{S}^{\rm G}_{\min}$ denotes an estimate of the global entropy lower bound $S^{\rm G}_{\min}$. In this paper, we choose 
\begin{equation*}
	\hat{S}^{\rm G}_{\min} = \min_j \min_{I_j} \mathcal{P}^{\tt 1D}_{I_j} \, S(\BU_j^0(x)).
\end{equation*}
As demonstrated in \Cref{section:5}, our new formula \eqref{eq:3089} is more suitable for estimating the local entropy bound in the context of the relativistic Euler system, compared with the existing formulas \eqref{eq:3141} and \eqref{CJK} in the literature.

\begin{remark} \label{rmk:3097}
	To illustrate the projection $\mathcal{P}^{\tt 1D}_{I_j}$, consider a continuous function $f(x)$ defined on $I = [-1,1]$. The resulting polynomial after the projection, $p(x) = \mathcal{P}^{\tt 1D}_I f(x)$, can be expressed as
	\[
	p(x) 
	=
	c_2 \, x^2 + c_1 \, x + c_0 
	=
	\left[
	\frac{1}{2}f(-1)-f(0)+\frac{1}{2}f(1)
	\right]x^2
	+
	\left[
	\frac{1}{2}f(1)-\frac{1}{2}f(-1)
	\right]
	x+f(0).
	\]
	The minimum of $p(x)$ over $I$ is explicitly given by 
	\[
	\min_{I} p(x) =
	\begin{dcases}
		c_0-\frac{c_1^2}{4 c_2}, & \textrm{if }~ c_2 > 0, \\
		\min\{f(-1),f(1)\}, & \textrm{if }~ c_2 \le 0.
	\end{dcases}
	\]
\end{remark}

{\bf New Estimation Approach II}: We also propose another new estimation approach based on the monotonicity-preserving technique \cite{suresh1997accurate}. In this approach, 
the interface point values contributing to numerical fluxes are limited to satisfy the local MEP, while the internal point values are only required to meet the global MEP. The global entropy bound can be directly estimated as  
$
S_{\min}^{\rm G,MP}:=\min\limits_{x}S(\BU_{0}(x)),
$ 
while the local entropy lower bounds are given by 
\begin{align*}
	(S_{\min}^{\rm L,MP})_j^{n+1} &:= S_{\min}^{\rm G,MP} \vee \min\{S_{j+1/2}^{\min,\pm},\ S_{j-1/2}^{\min,\pm}, S_{\rm avg}^{\min}\},\\
	S_{\rm avg}^{\min} &:= \min\{S(\overline{\BU}_{j-1}^n),\ S(\overline{\BU}_j^n),\ S(\overline{\BU}_{j+1}^n)\},
\end{align*}
with the quantities $S_{j+1/2}^{\min,\pm}$ computed through the monotonicity-preserving technique \cite{suresh1997accurate}:
\begin{align*}
	S_{j+1/2}^{\min,-} &:= \max\Big[\min\Big(S(\overline{\bf U}_{j}^n),\ S(\overline{\bf U}_{j+1}^n),\ S_{j+1/2}^{\rm MD}\Big),\ 
	\min\Big(S(\overline{\bf U}_{j}^n),\ S_{j+1/2}^{\rm UL},\ S_{j+1/2}^{\rm LC}\Big)\Big],\\
	d_{j} &:= S(\overline{\bf U}_{j+1}^n)-2S(\overline{\bf U}_{j}^n)+S(\overline{\bf U}_{j-1}^n),\\
	S_{j+1/2}^{\rm UL} &:= S(\overline{\bf U}_{j}^n)+\alpha\Big(S(\overline{\bf U}_{j}^n)-S(\overline{\bf U}_{j-1}^n)\Big),\\
	d_{j+1/2}^{\rm M4X} &:= {\rm minmod}\ \left(4d_{j}-d_{j+1},\ 4d_{j+1}-d_{j},\ d_{j},\ d_{j+1}\right),\\
	S_{j+1/2}^{\rm MD} &:= \frac{1}{2}\left(S(\overline{\bf U}_{j}^n)+ S(\overline{\bf U}_{j+1}^n)-d_{j+1/2}^{\rm M4X}\right),\\
	S_{j+1/2}^{\rm LC} &:= S(\overline{\bf U}_{j}^n)+\frac{1}{2}\left(S(\overline{\bf U}_{j}^n)-S(\overline{\bf U}_{j-1}^n)\right)+\frac{\beta}{3}d_{j-1/2}^{\rm M4X},
\end{align*}
where $\alpha=2$ and $\beta=4$. The estimation of $S_{j-1/2}^{\min,+}$ follows a symmetric transformation. 
We observe that this monotonicity-preserving estimation is usually too strict to maintain the high-order accuracy of the original scheme in smooth regions. To mitigate this, a shock-aware parameter $\nu_j$ is introduced to relax the overestimated local entropy bound by a convex combination with the global one:
\begin{equation} \label{1DMPrelax}
	(\widehat{S}_{\min}^{\rm L,MP})_j^{n+1}:=(1-\nu_j)S_{\min}^{\rm G,MP}+\nu_j (S_{\min}^{\rm L,MP})_j^{n+1}
\end{equation}
with
\begin{equation*}
	\nu_j := \min\left\{\frac{\max\{\Delta p_j, \Delta p_{j-1}\}}{\min\Big\{p(\overline{\BU}_{j-1}^n), p(\overline{\BU}_j^n), p(\overline{\BU}_{j+1}^n)\Big\}}, 1\right\},\quad \Delta p_j=\Big|p(\overline{\BU}_{j+1}^n)-p(\overline{\BU}_j^n)\Big|.
\end{equation*}

\subsection{Estimation on 2D Cartesian Meshes}
We consider a rectangular mesh that consists of cells $I_{j,k} = [x_{j-\frac12},x_{j+\frac12}] \times [y_{k-\frac12},y_{k+\frac12}]$. The solution polynomial in $I_{j,k}$ at time level $t^n$ is denoted by $\BU^n_{j,k}(x,y)$.

{\bf New Estimation Approach I}: Given the solution polynomial at initial time $\{\BU^0_{j,k}(x,y)\}_{j,k}$, the global entropy lower bound $S^{\rm G}_{\min}$ is estimated as 
\[
\hat{S}^{\rm G}_{\min} = \min_{j,k} \min_{I_{j,k}} \mathcal{P}^{\tt 2D}_{I_{j,k}} \, S(\BU_{j,k}^0(x,y)),
\]
where $\mathcal{P}^{\tt 2D}_{I_{j,k}}$ denotes a projection from any continuous function $f$ on $I_{j,k}$ to the quadratic polynomial $p \in \mathbb{P}^2$ that solves the following least-square problem:
\[
\underset{p \in \mathbb{P}^2}{{\min}} 
\sum_{(\xi,\eta) \in \mathcal{L}_{j,k}} 
\Big(p(\xi,\eta)-f(\xi,\eta)
\Big)^2
\quad
\textrm{with}
\quad
\mathcal{L}_{j,k} = \left\{x_{j-\frac12},x_j,x_{j+\frac12}\right\} \times \left\{y_{k-\frac12},y_k,y_{k+\frac12}\right\}.
\]
Similarly, given the solution polynomials at time level $t^n$, $\{\BU_j^{n}(x,y)\}_{j,k}$, the local entropy lower bound in cell $I_{j,k}$ at time level $t^{n+1}$ can be estimated by
\begin{equation} \label{2D_CUI}
	(\hat{S}^{\rm L}_{\min})^{n+1}_{j,k} 
	= 
	\hat{S}^{\rm G}_{\min} \vee
	\min \left\{
	\begin{aligned}
		&\min_{I_{j,k}} \mathcal{P}^{\tt 2D}_{I_{j,k}} S(\BU_{j,k}^n(x,y)),
		\\
		&\min_{I_{j\pm 1,k}} \mathcal{P}^{\tt 2D}_{I_{j\pm 1,k}} S(\BU_{j\pm 1,k}^n(x,y)),
		\\
		&\min_{I_{j,k\pm 1}} \mathcal{P}^{\tt 2D}_{I_{j,k\pm 1}} S(\BU_{j,k\pm 1}^n(x,y))
	\end{aligned}
	\right\}.
\end{equation}

\begin{remark}
	To illustrate the projection $\mathcal{P}^{\tt 2D}_I$, let us consider a continuous function $f(x,y)$ defined on $I = [-1,1] \times [-1,1]$. The resulting polynomial, $p(x,y) = \mathcal{P}^{\tt 2D}_I f(x,y)$, can be expressed as
	\begin{equation*}
		p(x,y) = c_{20} \, x^2+c_{11}\, xy+c_{02}\, y^2+c_{10}\, x+c_{01}\,y+c_{00}
	\end{equation*}
	with
	\begin{align*}
		c_{20} &= \frac{1}{6}
		\Big[
		f(-1,1)+f(-1,0)+f(-1,1)-f(0,-1)-f(0,0)-f(0,1)+f(1,-1)+f(1,0)+f(1,1)
		\Big],\\
		c_{11} &= \frac{1}{4}
		\Big[
		f(-1,-1)-f(-1,1)-f(1,-1)+f(1,1)
		\Big],\\
		c_{02} &= \frac{1}{6}
		\Big[
		f(-1,-1)-f(-1,0)+f(-1,1)+f(0,-1)-f(0,0)+f(0,1)+f(1,-1)-f(1,0)+f(1,1)
		\Big],\\
		c_{10} &= \frac{1}{6}
		\Big[
		-f(-1,-1)-f(-1,0)-f(-1,1)+f(1,-1)+f(1,0)+f(1,1)
		\Big],\\
		c_{01} &= \frac{1}{6}
		\Big[
		-f(-1,-1)+f(-1,1)-f(0,-1)+f(0,1)-f(1,-1)+f(1,1)
		\Big],\\
		c_{00} &= \frac{1}{9}
		\Big[
		-f(-1,-1)+2f(-1,0)-f(-1,1)+2f(0,-1)+5f(0,0)
		\\
		& \qquad \quad +2f(0,1)-f(1,-1)+2f(1,0)-f(1,1)
		\Big].
	\end{align*}
	After obtaining $p(x,y)$, we consider two cases:
	\begin{itemize}
		\item[\textbf{Case 1}:] If $c_{11}^2 \neq 4 c_{20} c_{02}$, then the minimum of $p(x,y)$ over $I$ can be expressed as
		\[
		\min_{I} f(x,y) = 
		\frac
		{c_{02} c_{10}^2 + c_{20} c_{01}^2 - c_{11} c_{10} c_{01}}
		{c_{11}^2 - 4 c_{20} c_{02}} 
		+ 
		c_{00}.
		\]
		\item[\textbf{Case 2}:] If $c_{11}^2 = 4 c_{20} c_{02}$, then the minimum of $p(x,y)$ over $I$ is always achieved on the edges of $I$, and one may find the minimum by applying the 1D minimum-finding procedure (see Remark \ref{rmk:3097}) along all the edges.
	\end{itemize} 
\end{remark}

{\bf New Estimation Approach II}: Another new estimation approach based on the monotonicity-preserving technique is given below. 
First, the global entropy bound is estimated from the initial data ${\bf U}_0(x,y)$:
\begin{equation*}
	S_{\min}^{\rm G,MP}:=\min\limits_{x,y}S(\BU_{0}(x,y)).
\end{equation*}
The local entropy lower bound is estimated as 
\begin{align*}
	&(S_{\min}^{\rm L,MP})_{j,k}^{n+1}:=S_{\min}^{\rm G,MP} \vee \min\{S_{j+1/2,k}^{\min,\pm},\ S_{j-1/2,k}^{\min,\pm}, S_{j,k+1/2}^{\min,\pm},\ S_{j,k-1/2}^{\min,\pm},\ S_{\rm avg}^{\min}\},\\
	&S_{\rm avg}^{\min}:=\min\{S(\overline{\BU}_{j-1,k}^n),\ S(\overline{\BU}_{j,k-1}^n),\ S(\overline{\BU}_{j,k}^n),\ S(\overline{\BU}_{j+1,k}^n),\ S(\overline{\BU}_{j,k+1}^n)\},
\end{align*}
where the quantities $S_{j+1/2,k}^{\min,\pm}$ and $S_{j,k+1/2}^{\min,\pm}$ are computed through the monotonicity-preserving technique \cite{suresh1997accurate}:
\begin{align*}
	&S_{j+1/2,k}^{\min,-}:=\max\Big[\min\Big(S(\overline{\bf U}_{j,k}^n),\ S(\overline{\bf U}_{j+1,k}^n),\ S_{j+1/2,k}^{\rm MD}\Big),\ 
	\min\Big(S(\overline{\bf U}_{j,k}^n),\ S_{j+1/2,k}^{\rm UL},\ S_{j+1/2,k}^{\rm LC}\Big)\Big],\\
	&d_{j,k}^{x}:=S(\overline{\bf U}_{j+1,k}^n)-2S(\overline{\bf U}_{j,k}^n)+S(\overline{\bf U}_{j-1,k}^n),\quad
	S_{j+1/2,k}^{\rm UL}:=S(\overline{\bf U}_{j,k}^n)+\alpha\Big(S(\overline{\bf U}_{j,k}^n)-S(\overline{\bf U}_{j-1,k}^n)\Big),\\
	&d_{j+1/2,k}^{{\rm M4X},x}:={\rm minmod}\ \left(4d_{j,k}^{x}-d_{j+1,k}^{x},\ 4d_{j+1,k}^{x}-d_{j,k}^{x},\ d_{j,k}^{x},\ d_{j+1,k}^{x}\right),\\
	&S_{j+1/2,k}^{\rm MD}:=\frac{1}{2}\left(S(\overline{\bf U}_{j,k}^n)+ S(\overline{\bf U}_{j+1,k}^n)-d_{j+1/2,k}^{{\rm M4X},x}\right),\\
	&S_{j+1/2,k}^{\rm LC}:=S(\overline{\bf U}_{j,k}^n)+\frac{1}{2}\left(S(\overline{\bf U}_{j,k}^n)-S(\overline{\bf U}_{j-1,k}^n)\right)+\frac{\beta}{3}d_{j-1/2,k}^{{\rm M4X},x},
	\\
	&S_{j,k+1/2}^{\min,-}:=\max\Big[\min\Big(S(\overline{\bf U}_{j,k}^n),\ S(\overline{\bf U}_{j,k+1}^n),\ S_{j,k+1/2}^{\rm MD}\Big),\ 
	\min\Big(S(\overline{\bf U}_{j,k}^n),\ S_{j,k+1/2}^{\rm UL},\ S_{j,k+1/2}^{\rm LC}\Big)\Big],\\
	&d_{j,k}^{y}:=S(\overline{\bf U}_{j,k+1}^n)-2S(\overline{\bf U}_{j,k}^n)+S(\overline{\bf U}_{j,k-1}^n),\quad
	S_{j,k+1/2}^{\rm UL}:=S(\overline{\bf U}_{j,k}^n)+\alpha\Big(S(\overline{\bf U}_{j,k}^n)-S(\overline{\bf U}_{j,k-1}^n)\Big),\\
	&d_{j,k+1/2}^{{\rm M4X},y}:={\rm minmod}\ \left(4d_{j,k}^{y}-d_{j,k+1}^{y},\ 4d_{j,k+1}^{y}-d_{j,k}^{y},\ d_{j,k}^{y},\ d_{j,k+1}^{y}\right),\\
	&S_{j,k+1/2}^{\rm MD}:=\frac{1}{2}\left(S(\overline{\bf U}_{j,k}^n)+ S(\overline{\bf U}_{j+1,k}^n)-d_{j+1/2,k}^{{\rm M4X},x}\right),\\
	&S_{j,k+1/2}^{\rm LC}:=S(\overline{\bf U}_{j,k}^n)+\frac{1}{2}\left(S(\overline{\bf U}_{j,k}^n)-S(\overline{\bf U}_{j,k-1}^n)\right)+\frac{\beta}{3}d_{j,k+1/2}^{{\rm M4X},y}.
\end{align*}
The quantities $S_{j-1/2,k}^{\min,+}$ and $S_{j,k-1/2}^{\min,+}$ are based on a symmetric estimation. 
Similar to the 1D case, we introduce a shock-aware relaxation parameter $\nu_{j,k}$ to deal with the overestimation problem:
\begin{equation} \label{2DMPrelax}
	(\widehat{S}_{\min}^{\rm L,MP})_{j,k}^{n+1}:=(1-\nu_{j,k})S_{\min}^{\rm G,MP}+\nu_{j,k} (S_{\min}^{\rm L,MP})_{j,k}^{n+1}
\end{equation}
with
\begin{align*}
	&\nu_{j,k} := \min\left\{\frac{\max\{\Delta_j p_{j,k},\ \Delta_j p_{j-1,k},\ \Delta_k p_{j,k},\ \Delta_k p_{j,k-1}\}}{\min\Big\{p(\overline{\BU}_{j-1,k}^n),\ p(\overline{\BU}_{j,k-1}^n),\ p(\overline{\BU}_{j,k}^n),\ p(\overline{\BU}_{j+1,k}^n),\ p(\overline{\BU}_{j,k+1}^n)\Big\}}, 1\right\},\\
	&\Delta_j p_{j,k}=\Big|p(\overline{\BU}_{j+1,k}^n)-p(\overline{\BU}_{j,k}^n)\Big|,\quad
	\Delta_k p_{j,k}=\Big|p(\overline{\BU}_{j,k+1}^n)-p(\overline{\BU}_{j,k}^n)\Big|.
\end{align*}

\begin{remark} \label{rmk:3270}
	As discussed in \Cref{rmk:1050}, one can use an SSP multistep (or Runge--Kutta) temporal  discretization method to achieve high-order accuracy in time. In this case, the estimation of the local minimum entropy at $t^{n+1}$ should also incorporate local information from all the involved time levels (stages). 
	For example, if the third-order SSP multistep method \eqref{eq:2469} is used, the estimations for the local minimum entropy on 1D and 2D Cartesian meshes are given by
	\[
	(\hat{S}^{\rm L}_{\min})^{n+1}_{j} = 
	\hat{S}^{\rm G}_{\min} \vee \min \left\{
	\begin{aligned}
		&\min_{I_j} \mathcal{P}^{\tt 1D}_{I_j} S(\BU_j^n(x)),
		&&\min_{I_j} \mathcal{P}^{\tt 1D}_{I_j} S(\BU_j^{n-3}(x)),
		\\
		&\min_{I_{j\pm 1}} \mathcal{P}^{\tt 1D}_{I_{j\pm 1}} S(\BU_{j\pm 1}^n(x)),
		&&\min_{I_{j\pm 1}} \mathcal{P}^{\tt 1D}_{I_{j\pm 1}} S(\BU_{j\pm 1}^{n-3}(x))
	\end{aligned}
	\right\}
	\]
	and
	\begin{equation*}
		(\hat{S}^{\rm L}_{\min})^{n+1}_{j,k} 
		= 
		\hat{S}^{\rm G}_{\min} \vee
		\min \left\{
		\begin{aligned}
			&\min_{I_{j,k}} \mathcal{P}^{\tt 2D}_{I_{j,k}} S(\BU_{j,k}^n(x,y)),
			&&\min_{I_{j,k}} \mathcal{P}^{\tt 2D}_{I_{j,k}} S(\BU_{j,k}^{n-3}(x,y)),
			\\
			&\min_{I_{j\pm 1,k}} \mathcal{P}^{\tt 2D}_{I_{j\pm 1,k}} S(\BU_{j\pm 1,k}^n(x,y)),
			&&\min_{I_{j\pm 1,k}} \mathcal{P}^{\tt 2D}_{I_{j\pm 1,k}} S(\BU_{j\pm 1,k}^{n-3}(x,y)),
			\\
			&\min_{I_{j,k\pm 1}} \mathcal{P}^{\tt 2D}_{I_{j,k\pm 1}} S(\BU_{j,k\pm 1}^n(x,y)),
			&&\min_{I_{j,k\pm 1}} \mathcal{P}^{\tt 2D}_{I_{j,k\pm 1}} S(\BU_{j,k\pm 1}^{n-3}(x,y))
		\end{aligned}
		\right\}.
	\end{equation*}
	For the monotonicity-preserving estimation approach, we use the following formulas for the third-order SSP multistep method, combined with \eqref{1DMPrelax} and \eqref{2DMPrelax}:
	\begin{equation*}
		(\widehat{S}_{\min}^{\rm L,MP})_j^{n+1}:=(1-\nu_j)S_{\min}^{\rm G,MP}+\nu_j \min\left\{(S_{\min}^{\rm L,MP})_j^{n+1},\ (S_{\min}^{\rm L,MP})_j^{n-2}\right\}
	\end{equation*}
	and
	\begin{equation*} 
		(\widehat{S}_{\min}^{\rm L,MP})_{j,k}^{n+1}:=(1-\nu_{j,k})S_{\min}^{\rm G,MP}+\nu_{j,k} \min\left\{(S_{\min}^{\rm L,MP})_{j,k}^{n+1},\ (S_{\min}^{\rm L,MP})_{j,k}^{n-2}\right\}.
	\end{equation*}
\end{remark}


\section{Numerical Results}
\label{section:5}

In this section, we present a series of numerical experiments to assess the accuracy and robustness of our high-order entropy-preserving discontinuous Galerkin (DG) schemes for the relativistic Euler system \eqref{eq:RHD3D} with various equations of state (EOSs) on both 1D and 2D rectangular meshes. 
We evaluate both globally and locally entropy-preserving DG schemes and compare our proposed entropy bound estimation techniques with existing approaches from the literature, as discussed in Section \ref{sec:6}. The comparison focuses on their accuracy and effectiveness in suppressing spurious oscillations. For further validation, we also examine high-order bound-preserving (BP) DG schemes, which are obtained by omitting Step 3 (i.e., the entropy-preserving limiter) in \Cref{sec:limiter}. 
For time integration, we employ the third-order strong-stability-preserving (SSP) multistep method \eqref{eq:2469}. Unless otherwise stated, the CFL numbers are set as $1/10$, $1/20$, and $1/30$ for the $\mathbb{P}^1$-, $\mathbb{P}^2$-, and $\mathbb{P}^3$-based DG methods, respectively.
It should be noted that, in all numerical experiments reported in this section, no additional limiters (e.g., TVD, TVB, or WENO limiters) are employed to mitigate spurious oscillations.

\begin{example}[1D Accuracy Test]\label{Ex5.1.1} 
	This example verifies the accuracy of our high-order entropy-preserving DG schemes for the 1D relativistic Euler system. We consider the {IP-EOS} \eqref{hEOS2} with the smooth initial condition
	\begin{equation*}
		\textbf{V}(x,0) = \left(1+0.99999\sin(2\pi x),0.9,1\right)^\top, \qquad x \in [0,1],
	\end{equation*}
	subject to periodic boundary conditions. The exact solution is given by
	\begin{equation*}
		\textbf{V}(x,t) = \left(1+0.99999\sin\left(2\pi(x-0.9t)\right),0.9,1\right)^\top.
	\end{equation*}
	
	We compute the numerical solution up to $t = 0.2$ using $\mathbb{P}^1$-, $\mathbb{P}^2$-, and $\mathbb{P}^3$-based locally entropy-preserving DG methods with different approaches for estimating the entropy lower bounds. For the $\mathbb{P}^3$-based DG method, the time step is adjusted to match the spatial accuracy using
		$
		\Delta t = \frac{1}{30} \Delta x^{\frac{4}{3}}.
		$ 
	 The $l^1$, $l^2$, and $l^\infty$ errors of the rest-mass density $\rho$ on various mesh resolutions are presented in \Cref{table:1DAT_IPEOS}. 
	The results confirm that our locally entropy-preserving schemes achieve the expected convergence rates, demonstrating that the proposed limiters, equipped with the new entropy bound estimation approaches \eqref{eq:3089} (abbreviated as ``{\tt New I}'') and \eqref{1DMPrelax} (abbreviated as ``{\tt New II}''), preserve high-order accuracy.
	
	In contrast, adopting the existing formulas \eqref{LV Final star} (abbreviated as ``{\tt LI}'') and \eqref{CJK} (abbreviated as ``{\tt CJK}'') for entropy bound estimation leads to noticeable degradation of convergence rates in the $\mathbb{P}^2$ and $\mathbb{P}^3$ cases. This indicates that \eqref{LV Final star} and \eqref{CJK} tend to overestimate the entropy bound, thereby compromising the accuracy of the high-order schemes.

\begin{figure}[!htb]
	\centering
	\begin{subfigure}[t]{0.32\textwidth}
		\centering
		\includegraphics[trim = 17 0 0 0,clip,width=\textwidth]{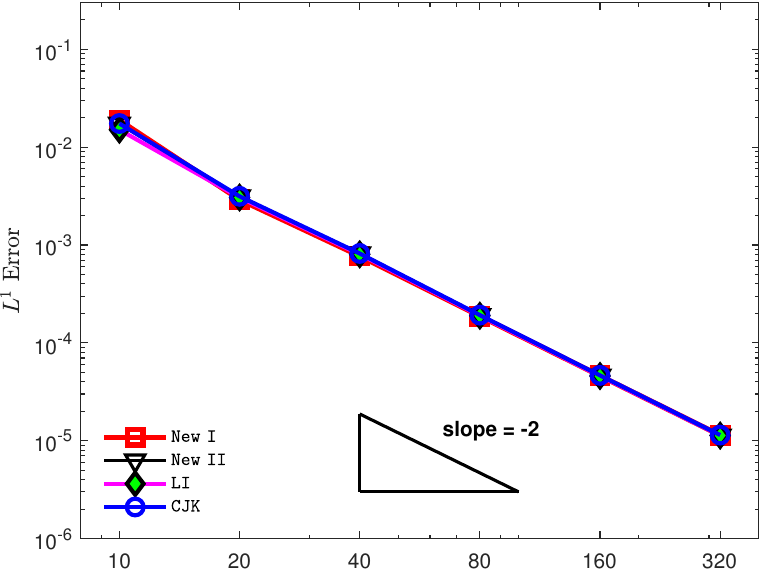}
		\caption{$\mathbb{P}^1$, $l^1$ errors.}
	\end{subfigure}
	\begin{subfigure}[t]{0.32\textwidth}
		\centering
		\includegraphics[trim = 17 0 0 0,clip,width=1\textwidth]{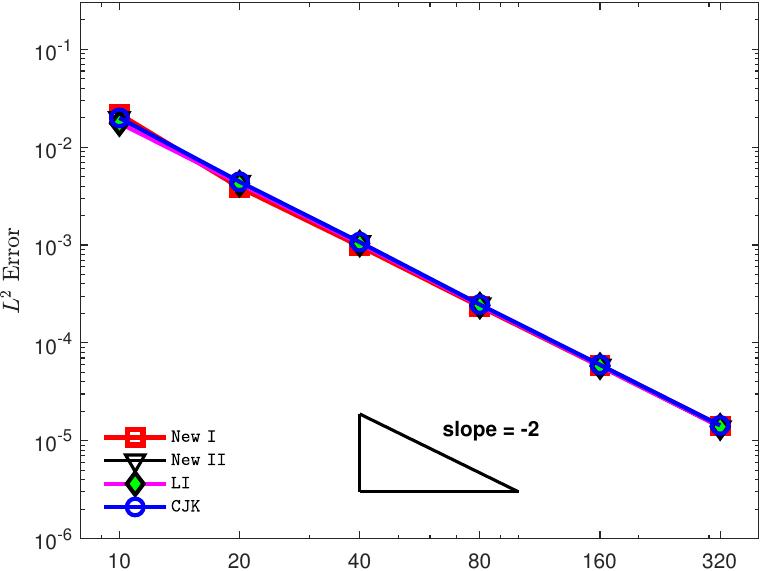}
		\caption{$\mathbb{P}^1$, $l^2$ errors.}
	\end{subfigure}
	\begin{subfigure}[t]{.32\textwidth}
		\centering
		\includegraphics[trim = 17 0 0 0,clip,width=1\textwidth]{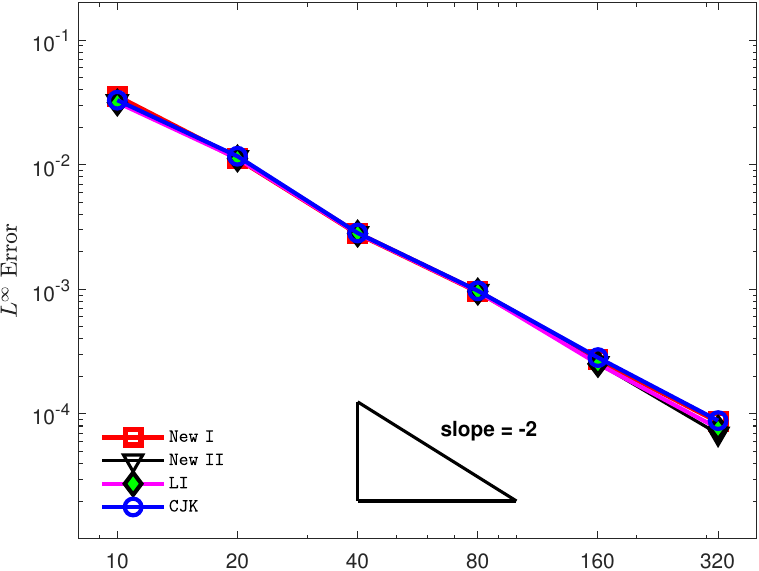}
		\caption{$\mathbb{P}^1$, $l^{\infty}$ errors.}
	\end{subfigure}
    \begin{subfigure}[t]{0.32\textwidth}
		\centering
		\includegraphics[trim = 17 0 0 0,clip,width=\textwidth]{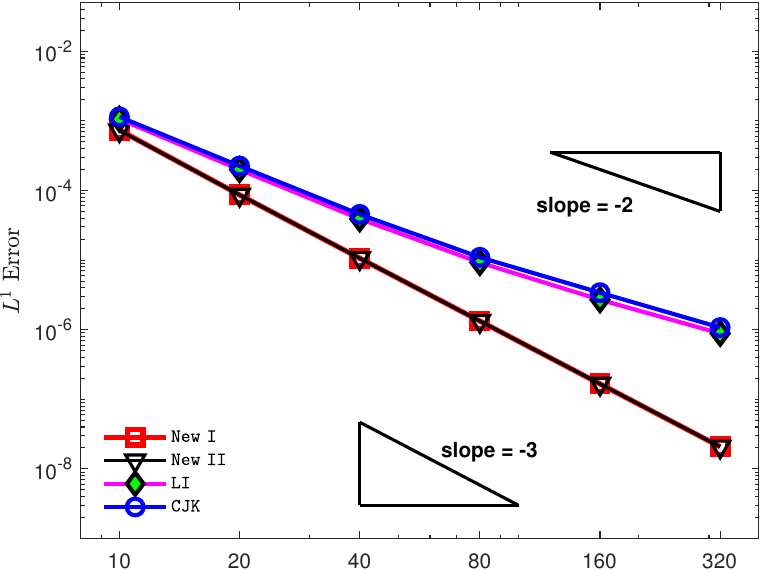}
		\caption{$\mathbb{P}^2$, $l^1$ errors.}
	\end{subfigure}
	\begin{subfigure}[t]{0.32\textwidth}
		\centering
		\includegraphics[trim = 17 0 0 0,clip,width=1\textwidth]{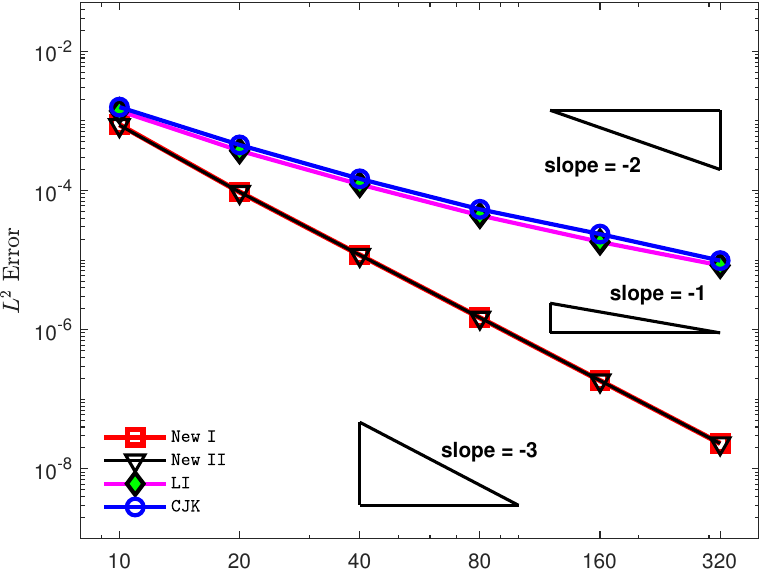}
		\caption{$\mathbb{P}^2$, $l^2$ errors.}
	\end{subfigure}
	\begin{subfigure}[t]{.32\textwidth}
		\centering
		\includegraphics[trim = 17 0 0 0,clip,width=1\textwidth]{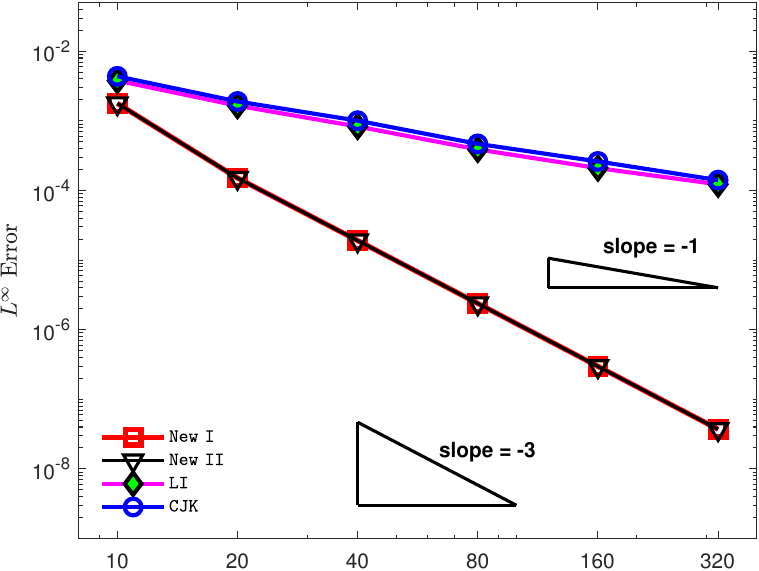}
		\caption{$\mathbb{P}^2$, $l^{\infty}$ errors.}
	\end{subfigure}
    \begin{subfigure}[t]{0.32\textwidth}
		\centering
		\includegraphics[trim = 17 0 0 0,clip,width=\textwidth]{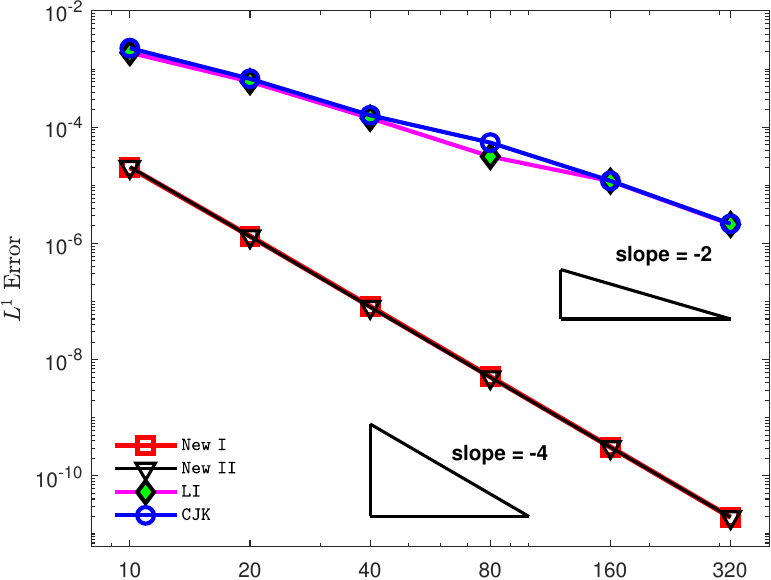}
		\caption{$\mathbb{P}^3$, $l^1$ errors.}
	\end{subfigure}
	\begin{subfigure}[t]{0.32\textwidth}
		\centering
		\includegraphics[trim = 17 0 0 0,clip,width=1\textwidth]{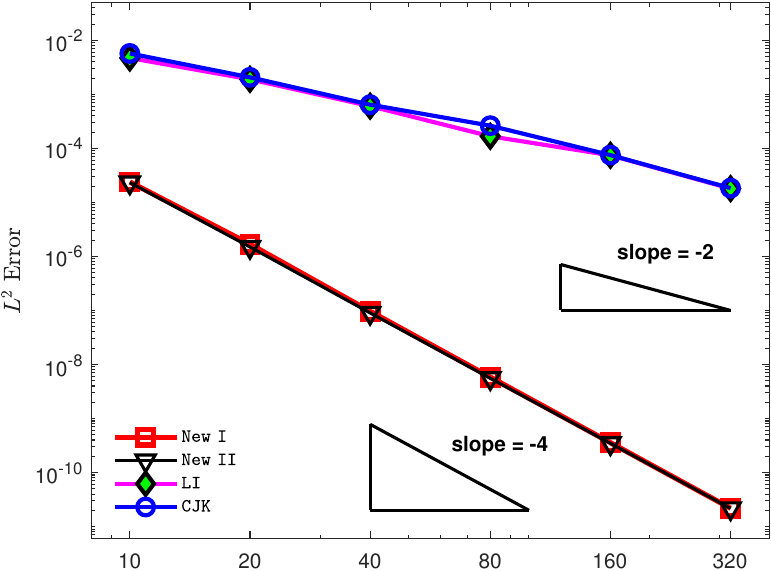}
		\caption{$\mathbb{P}^3$, $l^2$ errors.}
	\end{subfigure}
	\begin{subfigure}[t]{.32\textwidth}
		\centering
		\includegraphics[trim = 17 0 0 0,clip,width=1\textwidth]{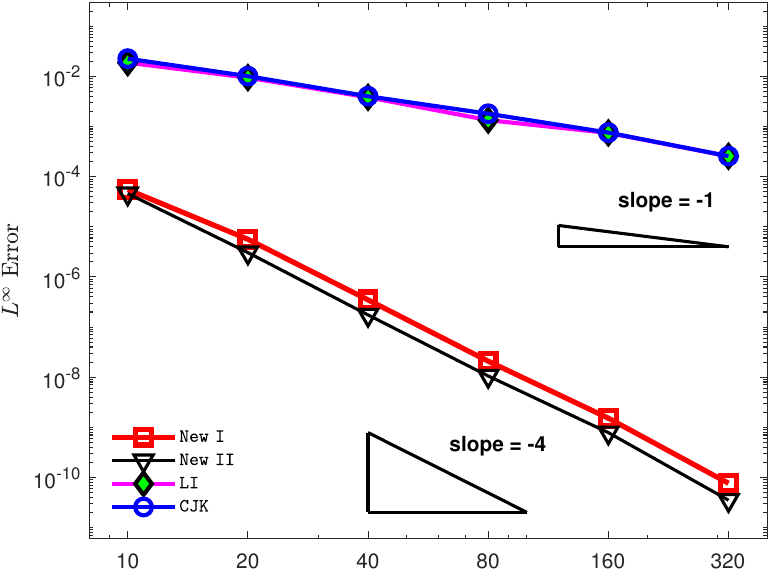}
		\caption{$\mathbb{P}^3$, $l^{\infty}$ errors.}
	\end{subfigure}
	\caption{Example \ref{Ex5.1.1}: $l^1$, $l^2$, and $l^\infty$ errors in $\rho$ for locally entropy-preserving DG methods on various mesh resolutions, comparing different entropy bound estimation approaches \eqref{LV Final star}, \eqref{CJK}, \eqref{eq:3089}, and \eqref{1DMPrelax}.}\label{table:1DAT_IPEOS}
\end{figure}

Table \ref{CPU:1D} presents the CPU time required for the locally entropy-preserving schemes using four different local entropy bound estimation approaches on a mesh with 320 uniform cells. Since the {\tt LI} approach \eqref{LV Final star}, the {\tt CJK} approach \eqref{CJK}, and the {\tt New I} approach \eqref{eq:3089} involve quadrature points within each cell, they naturally require higher computational costs compared to the {\tt New II} approach \eqref{1DMPrelax}, which relies solely on cell average information. The results in Table \ref{CPU:1D} confirm this expectation, demonstrating that the estimation approach \eqref{1DMPrelax} is the most computationally efficient.

\begin{table}[htb]
	\centering
	\caption{CPU time in seconds for simulating Example \ref{Ex5.1.1} up to $t=0.2$ with $320$ cells.}
	\label{CPU:1D}
	\renewcommand\arraystretch{1.2}
	\scalebox{0.8}{
		\begin{tabular}{c||c|c|c||c|c|c||c|c|c||c|c|c}
			\hline
			 estimation approach & \multicolumn{3}{c||}{{\tt LI} \eqref{LV Final star}} & \multicolumn{3}{c||}{{\tt CJR} \eqref{CJK}} & \multicolumn{3}{c||}{{\tt New I} \eqref{eq:3089}} & \multicolumn{3}{c}{{\tt New II} \eqref{1DMPrelax}}\\
			\hline
			$\mathbb{P}^k$ element & $\mathbb{P}^1$ & $\mathbb{P}^2$ & $\mathbb{P}^3$ & $\mathbb{P}^1$ & $\mathbb{P}^2$ & $\mathbb{P}^3$ & $\mathbb{P}^1$ & $\mathbb{P}^2$ & $\mathbb{P}^3$ & $\mathbb{P}^1$ & $\mathbb{P}^2$ & $\mathbb{P}^3$\\
			\hline
			CPU time & 4 & 10 & 135.5 & 4 & 10 & 135.5 & 4 & 10 & 139.75 & 4 & 9 & 108.75 \\
			\hline
		\end{tabular}
	}
\end{table}

\end{example}

\begin{example}[1D Riemann Problem I]\label{Ex5.1.2}
	To assess the capability of the proposed high-order entropy-preserving DG schemes in resolving complex wave structures and suppressing spurious oscillations, we consider the following discontinuous initial condition:
	\begin{equation*}
		\textbf{V}(x,0) =
		\begin{cases}
			(0.8,0.5,8)^\top, & 0\leq x<0.5, \\
			(1,0,1)^\top, & 0.5\leq x \leq 1,
		\end{cases}
	\end{equation*}
	subject to outflow boundary conditions and with the TM-EOS \eqref{hEOS3}. The computational domain $[0,1]$ is discretized into 400 uniform cells. 
	The exact solution to this Riemann problem consists of a left-moving rarefaction wave, a contact discontinuity, and a right-moving shock wave. 
	
	Figure \ref{Fig:1D_RP_TMEOS} presents the numerical solutions of the rest-mass density $\rho$ at $t=0.4$, obtained using the $\mathbb{P}^3$-based DG methods under three different settings:  
	(a) without the entropy-preserving limiter,  
	(b) with the global entropy-preserving limiter, and  
	(c) with the local entropy-preserving limiter incorporating our new estimation ({\tt New I}) approach  \eqref{eq:3089}.  
As shown in Figure \ref{Fig:1D_RP_TMEOS_d}, the proposed local entropy-preserving DG scheme accurately captures the wave structures while effectively preventing significant spurious oscillations. 
In contrast, Figure \ref{Fig:1D_RP_TMEOS_a} demonstrates that the scheme without any entropy-preserving limiters introduces noticeable nonphysical oscillations ahead of the shock and an overshoot near the contact discontinuity at $x = 0.74$. 
This behavior aligns with findings in \cite{khobalatte1994maximum, zhang2012minimum, JIANG2018}, which suggest that enforcing the MEP is crucial for mitigating spurious oscillations.
Furthermore, while the global entropy-preserving limiter helps reduce oscillations, minor artifacts persist near the contact discontinuity. 
In contrast, the local entropy-preserving limiter exhibits superior performance, effectively suppressing nonphysical oscillations while preserving the sharpness of the wave structures.

\begin{figure}[!htb]
	\centering
	\begin{subfigure}[t]{0.32\textwidth}
		\centering
		\includegraphics[width=\textwidth]{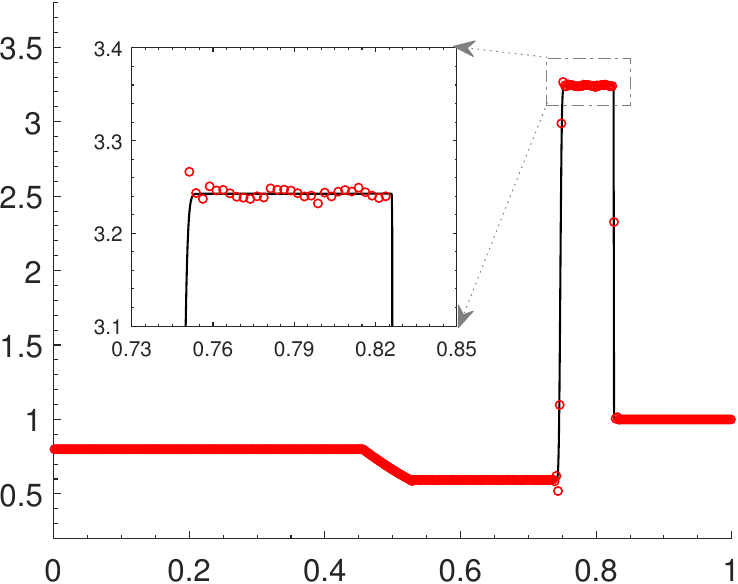}
		\caption{Without entropy-preserving limiter.}
		\label{Fig:1D_RP_TMEOS_a}
	\end{subfigure}
	\begin{subfigure}[t]{0.32\textwidth}
		\centering
		\includegraphics[width=1\textwidth]{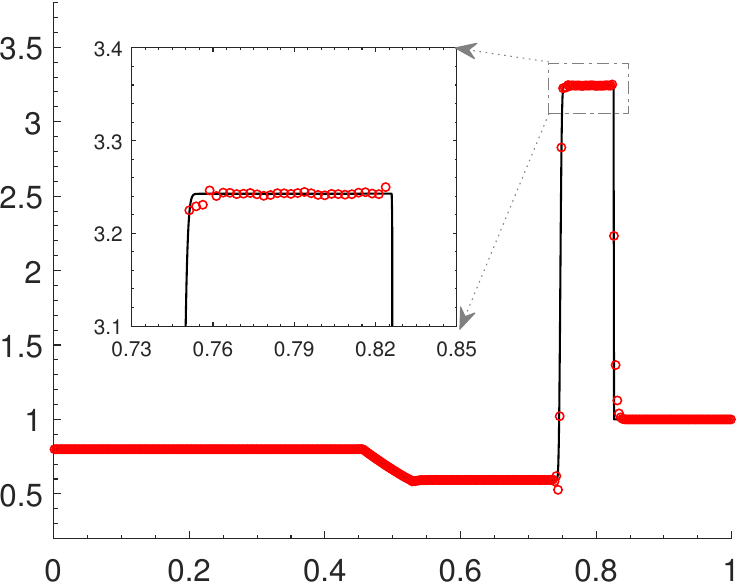}
		\caption{With global entropy-preserving limiter.}
		\label{Fig:1D_RP_TMEOS_b}
	\end{subfigure}
	\begin{subfigure}[t]{0.32\textwidth}
		\centering
		\includegraphics[width=1\textwidth]{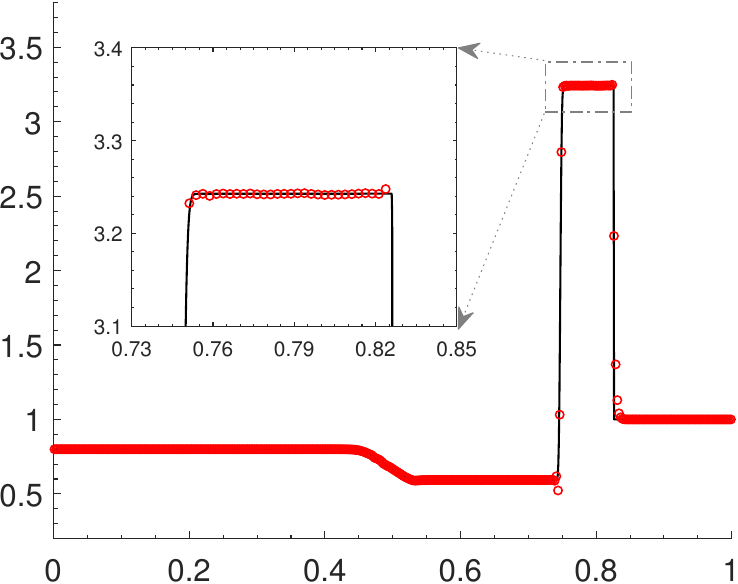}
		\caption{Local entropy-preserving with {\tt New I} \eqref{eq:3089}.}
		\label{Fig:1D_RP_TMEOS_d}
	\end{subfigure}
	\caption{Example \ref{Ex5.1.2}: the rest-mass density $\rho$ obtained by the $\mathbb{P}^3$-based DG methods with and without different entropy limiters (red circles) and reference solution (black solid line), $t = 0.4$, $\dx = 1/400$.}\label{Fig:1D_RP_TMEOS}
\end{figure}

\end{example}

\begin{example}{(1D Riemann Problem II)}\label{Ex5.1.3}
	This problem is initialized with the following discontinuous data:
\begin{equation*}
	\textbf{V}(x,0)=
	\begin{cases}
		(1.25,0,0.8)^\top,&  0\leq x<0.5,\\
		(0.1,0,0.1)^\top,& 0.5\leq x \leq 1,
	\end{cases}
\end{equation*}
	and employs the {RC-EOS \eqref{hEOS1}}. The exact solution consists of a rarefaction wave, a contact discontinuity, and a shock wave. The computational domain $[0,1]$ is discretized using $400$ uniform cells with outflow boundary conditions.
	
	Figure \ref{Fig:1D_RP2_RCEOS} presents the numerical results of the rest-mass density $\rho$ at $t = 0.4$, obtained using the $\mathbb{P}^3$-based DG methods with three different configurations:
	(a) without the entropy-preserving limiters,
	(b) with the local entropy-preserving limiter and the {\tt New I} entropy bound estimation approach \eqref{eq:3089}, and
	(c) with the local entropy-preserving limiter and the {\tt New II} entropy bound estimation approach \eqref{eq:3089}, \eqref{1DMPrelax}. 
	The result without using the entropy-preserving limiter exhibits slight overshoots and undershoots near the discontinuity at approximately $x = 0.66$, whereas the locally entropy-preserving schemes effectively suppress these spurious oscillations, demonstrating their robustness in maintaining physically consistent solutions.

\begin{figure}[!htb]
	\centering
	\begin{subfigure}[t]{0.32\textwidth}
		\centering
		\includegraphics[width=\textwidth]{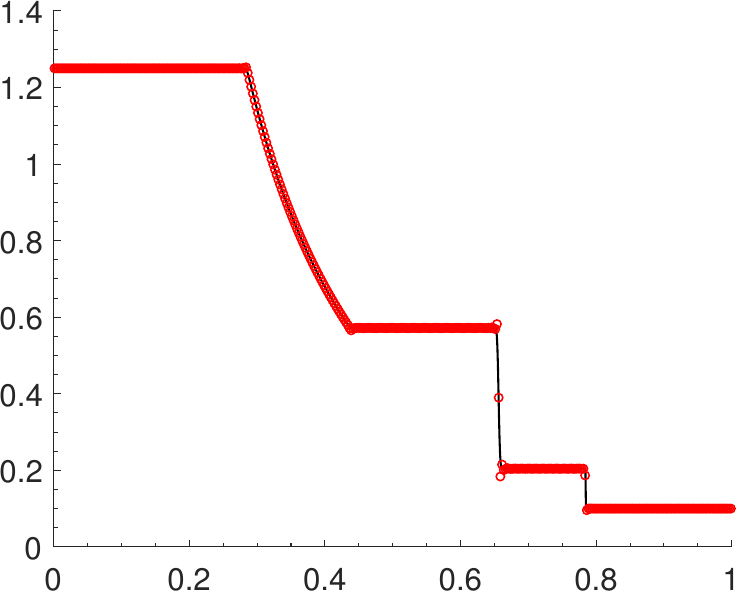}
		\caption{Without entropy-preserving limiter.}
	\end{subfigure}
	\begin{subfigure}[t]{0.32\textwidth}
		\centering
		\includegraphics[width=1\textwidth]{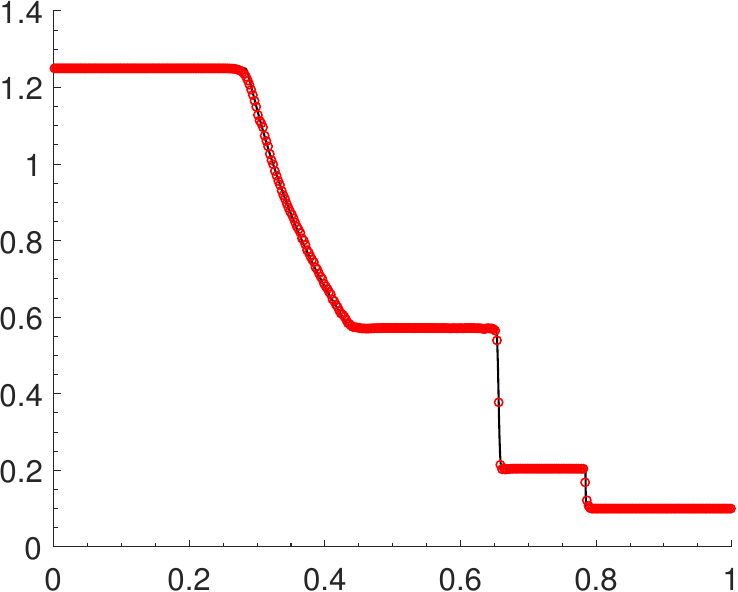}
		\caption{Local entropy-preserving with {\tt New I} \eqref{eq:3089}.}
	\end{subfigure}
    \begin{subfigure}[t]{0.32\textwidth}
		\centering
		\includegraphics[width=1\textwidth]{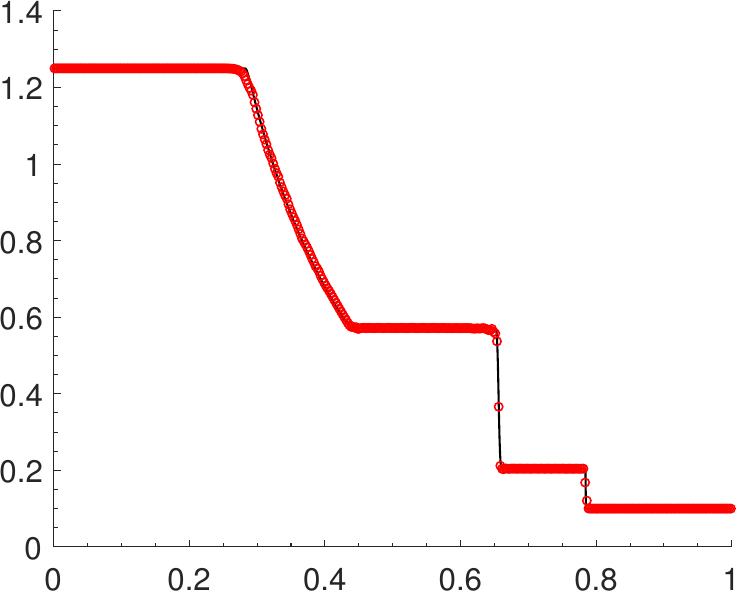}
		\caption{Local entropy-preserving with {\tt New II} \eqref{1DMPrelax}.}
	\end{subfigure}
	\caption{Example \ref{Ex5.1.3}: the rest-mass density $\rho$ obtained by the $\mathbb{P}^3$-based DG methods without or with entropy-preserving limiters (red circles) and reference solution (black solid line), $t = 0.4$, $\dx = 1/400$.}\label{Fig:1D_RP2_RCEOS}
\end{figure}

\end{example}

\begin{example}{(2D Accuracy Test)}\label{Ex5.2.1}
	This example examines the accuracy of the proposed high-order entropy-preserving DG methods for a smooth 2D problem with periodic boundary conditions over the computational domain $[0,1]^2$. 
	The exact solution is given by
	\begin{equation*}
		\textbf{V}(x,y,t) = \left(1+0.99999\sin\left(2\pi(x+y-0.99\sqrt{2}t)\right),\frac{0.99}{\sqrt{2}},\frac{0.99}{\sqrt{2}},0.01\right)^\top,
	\end{equation*}
	which describes a high-speed relativistic flow with low density and low pressure, propagating at an angle of $45^\circ$ relative to the $x$-axis. This test case is conducted using the {RC-EOS \eqref{hEOS1}}.
	
	The computational domain $[0,1]^2$ is discretized into a uniform grid of $N \times N$ square cells, where $N \in \{10,20,40,80,160,320\}$. For the $\mathbb{P}^3$-based DG method, the time step is adjusted to match the spatial accuracy using
		$
		\Delta t = \frac{1}{30} {\left(\frac{\Delta x}{2}\right)}^{\frac{4}{3}}.
		$  
	Figure \ref{table:2DAT_RCTMEOS} reports the $l^1$, $l^2$, and $l^{\infty}$ errors in the rest-mass density $\rho$ at $t = 0.2$, along with the corresponding convergence orders. 
	The results confirm that the proposed locally entropy-preserving DG methods, incorporating either the {\tt New I} or {\tt New II} approaches for entropy bound estimation, achieve the expected order of accuracy. 
	For comparison, we also include results obtained using the locally entropy-preserving schemes with the {\tt LI} approach \eqref{LV Final star} and {\tt CJK} approach \eqref{CJK} for entropy bound estimations. 
	The observed accuracy degradation in these cases suggests that these estimations are overly restrictive, leading to an overestimation of the entropy bounds and preventing the schemes from achieving the desired convergence rates.

\begin{figure}[!htb]
	\centering
	\begin{subfigure}[t]{0.32\textwidth}
		\centering
		\includegraphics[trim = 17 0 0 0,clip,width=\textwidth]{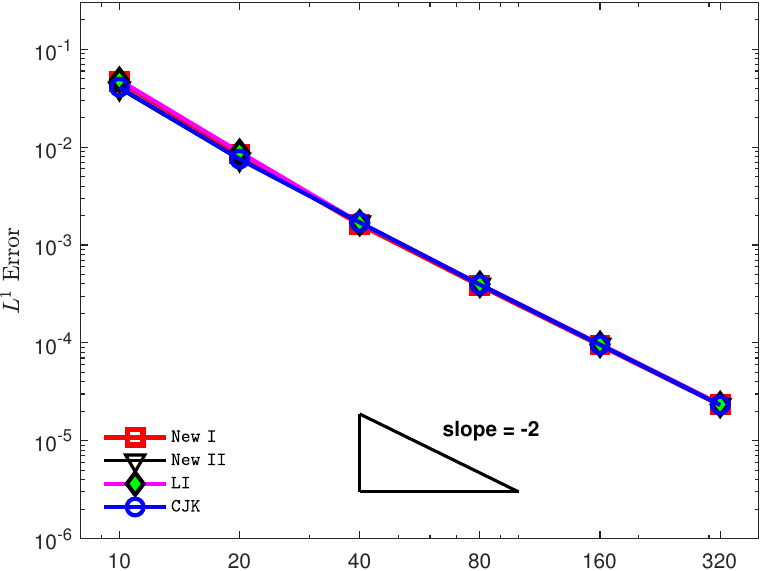}
		\caption{$\mathbb{P}^1$, $l^1$ errors.}
	\end{subfigure}
	\begin{subfigure}[t]{0.32\textwidth}
		\centering
		\includegraphics[trim = 17 0 0 0,clip,width=1\textwidth]{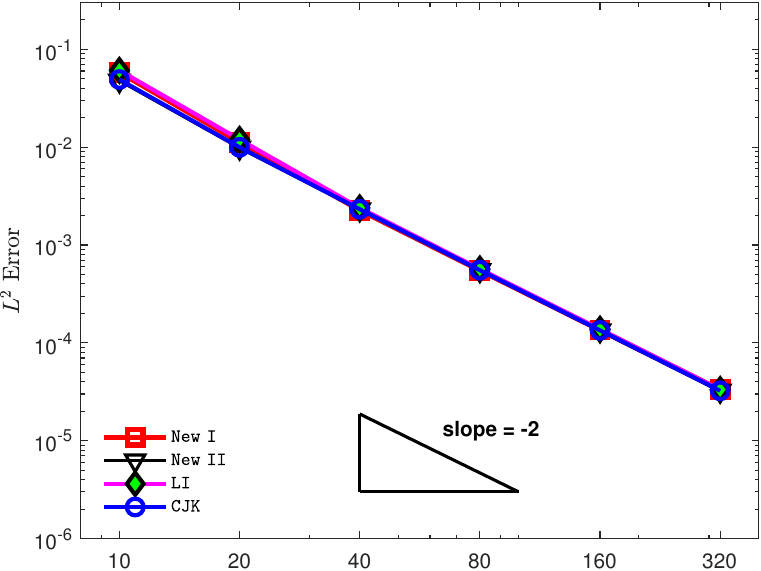}
		\caption{$\mathbb{P}^1$, $l^2$ errors.}
	\end{subfigure}
	\begin{subfigure}[t]{.32\textwidth}
		\centering
		\includegraphics[trim = 17 0 0 0,clip,width=1\textwidth]{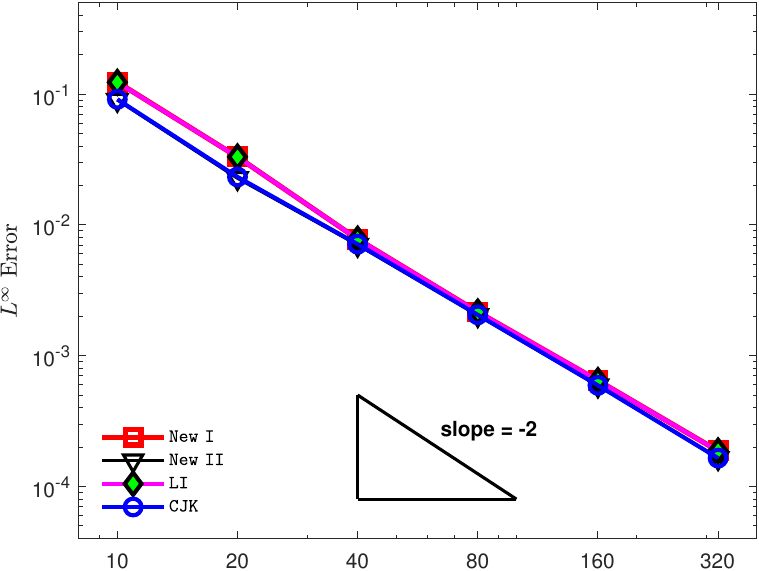}
		\caption{$\mathbb{P}^1$, $l^{\infty}$ errors.}
	\end{subfigure}
    \begin{subfigure}[t]{0.32\textwidth}
		\centering
		\includegraphics[trim = 17 0 0 0,clip,width=\textwidth]{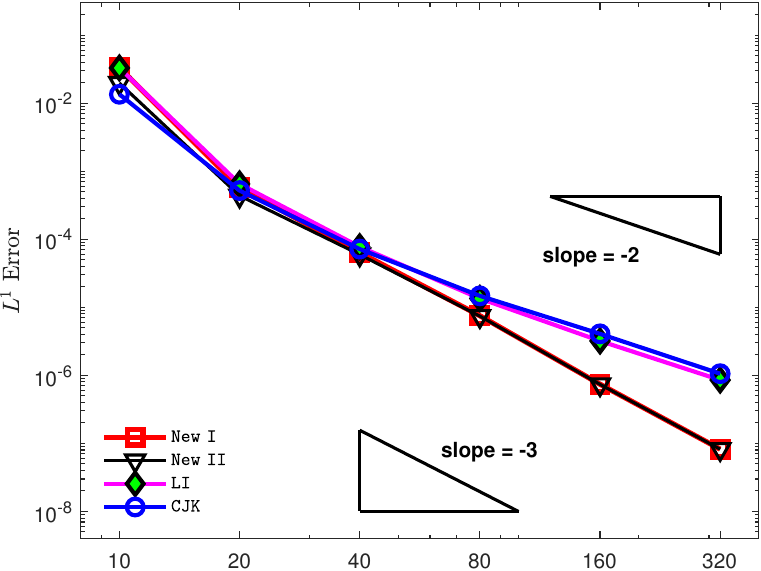}
		\caption{$\mathbb{P}^2$, $l^1$ errors.}
	\end{subfigure}
	\begin{subfigure}[t]{0.32\textwidth}
		\centering
		\includegraphics[trim = 17 0 0 0,clip,width=1\textwidth]{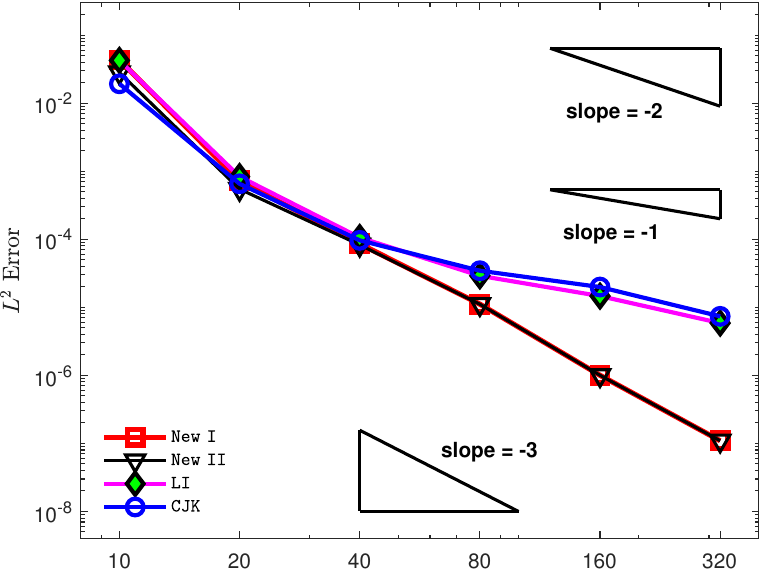}
		\caption{$\mathbb{P}^2$, $l^2$ errors.}
	\end{subfigure}
	\begin{subfigure}[t]{.32\textwidth}
		\centering
		\includegraphics[trim = 17 0 0 0,clip,width=1\textwidth]{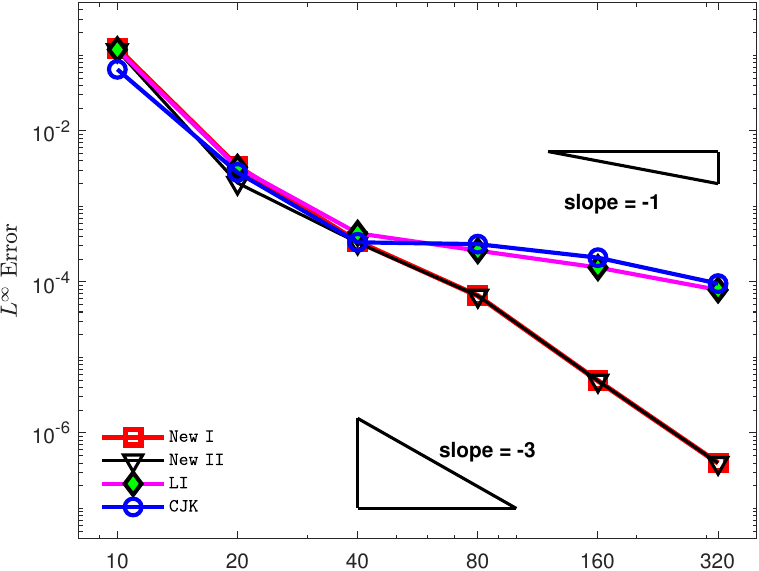}
		\caption{$\mathbb{P}^2$, $l^{\infty}$ errors.}
	\end{subfigure}
    \begin{subfigure}[t]{0.32\textwidth}
		\centering
		\includegraphics[trim = 17 0 0 0,clip,width=\textwidth]{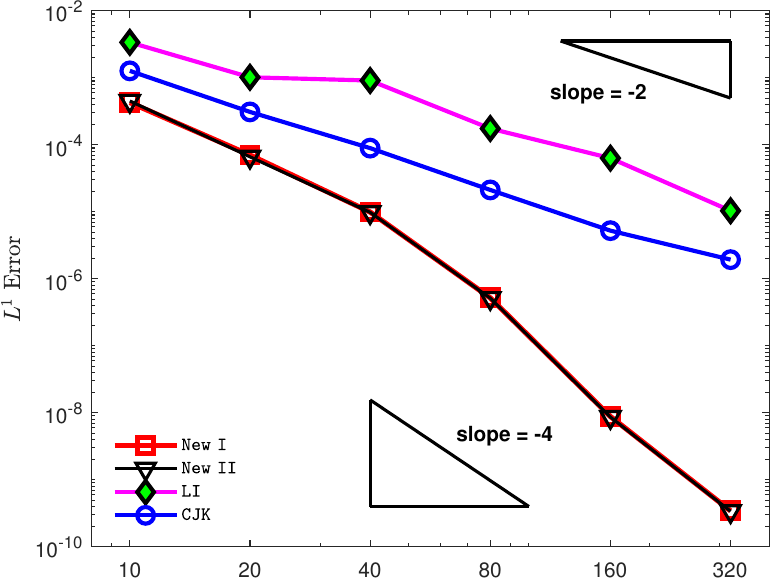}
		\caption{$\mathbb{P}^3$, $l^1$ errors.}
	\end{subfigure}
	\begin{subfigure}[t]{0.32\textwidth}
		\centering
		\includegraphics[trim = 17 0 0 0,clip,width=1\textwidth]{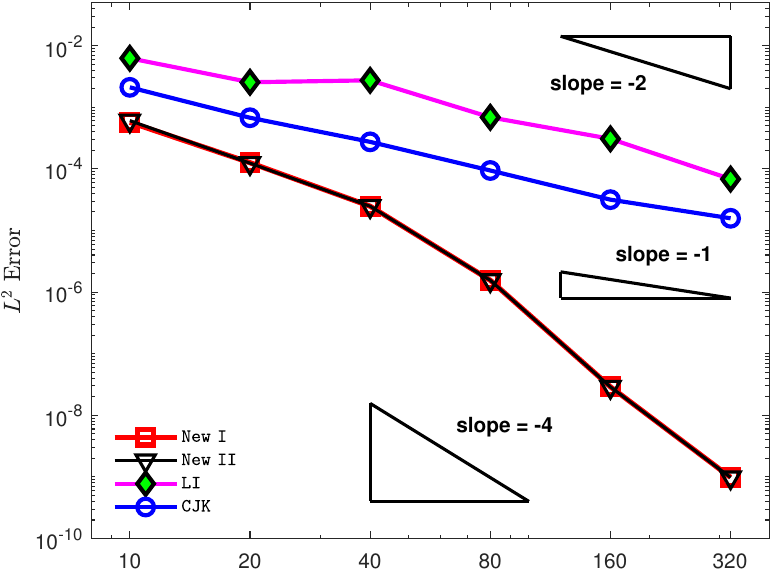}
		\caption{$\mathbb{P}^3$, $l^2$ errors.}
	\end{subfigure}
	\begin{subfigure}[t]{.32\textwidth}
		\centering
		\includegraphics[trim = 17 0 0 0,clip,width=1\textwidth]{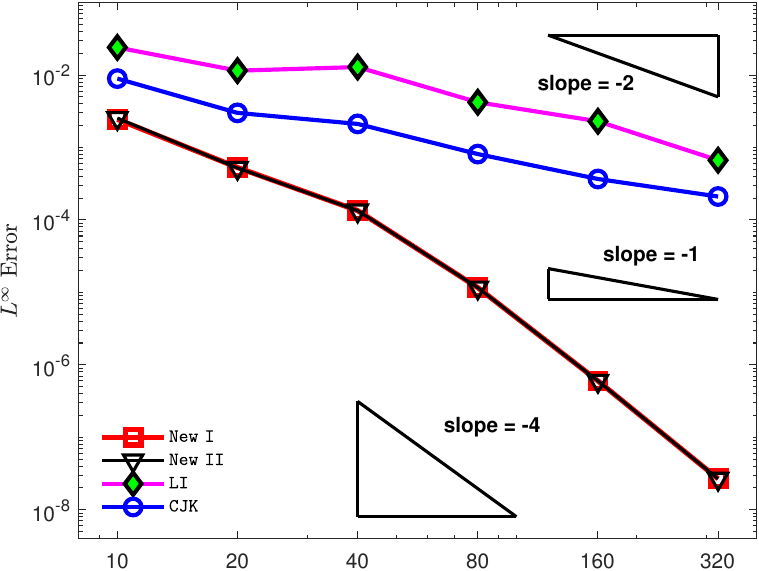}
		\caption{$\mathbb{P}^3$, $l^{\infty}$ errors.}
	\end{subfigure}
	\caption{Example \ref{Ex5.2.1}: $l^1$, $l^2$, and $l^\infty$ errors in $\rho$ for locally entropy-preserving DG methods on various mesh resolutions, comparing different entropy bound estimation approaches \eqref{LV Final star}, \eqref{CJK}, \eqref{2D_CUI}, and \eqref{2DMPrelax}.}\label{table:2DAT_RCTMEOS}
\end{figure}

\end{example}

We now test two 2D Riemann problems for the relativistic Euler system, originally proposed in \cite{WuTang2015} and \cite{he2012adaptive1}, respectively.

\begin{example}{(2D Riemann Problem I).}\label{Ex5.2.2}
	The initial condition for the first 2D Riemann problem is given by 
	\begin{equation*}
		\textbf{V}(x,y,0) = 
		\begin{cases}
			(0.1, 0, 0, 20)^\top, & x>0,\ y>0,\\
			(0.00414329639576, 0.9946418833556542, 0, 0.05)^\top, & x<0,\ y>0,\\
			(0.01, 0, 0, 0.05)^\top, & x<0,\ y<0,\\
			(0.00414329639576, 0, 0.9946418833556542, 0.05)^\top, & x> 0, \ y<0.
		\end{cases}
	\end{equation*} 
	The computational domain is set to $[-1,1]^2$ with outflow boundary conditions imposed on all boundaries. In this example, we adopt the ideal EOS \eqref{ID-EOS}. 
	The maximum fluid velocity is close to the speed of light, presenting significant challenges for numerical simulation. 
	
	We employ a uniform mesh with $250\times250$ cells to assess the performance of the $\mathbb{P}^3$-based DG schemes, with or without entropy-preserving techniques. The numerical solutions at $t = 0.8$ are displayed in Figure \ref{Fig:2D_RP1_RCEOS}. The results show that the  scheme without an entropy-preserving limiter procedures severe nonphysical oscillations. In contrast, the (locally) entropy-preserving schemes effectively suppress these oscillations, significantly improving the robustness, fidelity, and reliability of the numerical solutions.
	
	Table \ref{CPU:2D} presents a comparison of CPU time for locally entropy-preserving schemes with various entropy bound estimation approaches. Consistent with the 1D case, our {\tt New II} approach \eqref{2DMPrelax} proves to be the most efficient, as it relies solely on cell averages rather than evaluating entropy at all quadrature points.
\end{example}

\begin{figure}[!htb]
	\centering
	\begin{subfigure}[t]{.32\textwidth}
		\centering
		\includegraphics[width=1\textwidth]{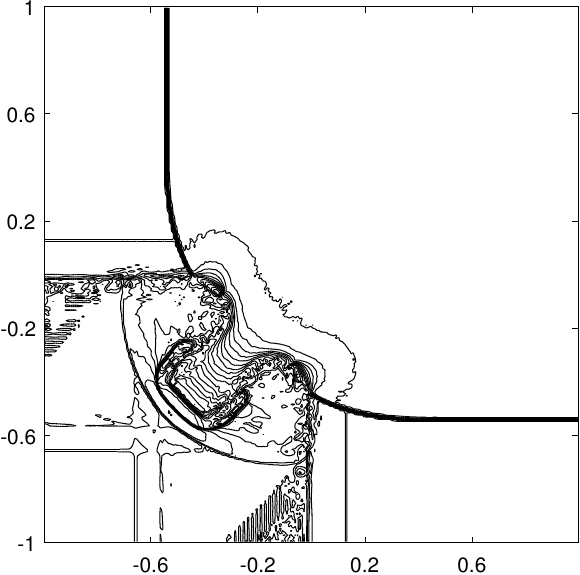}
		\caption{Without entropy-preserving limiter.}
	\end{subfigure}
	\begin{subfigure}[t]{.32\textwidth}
		\centering
		\includegraphics[width=1\textwidth]{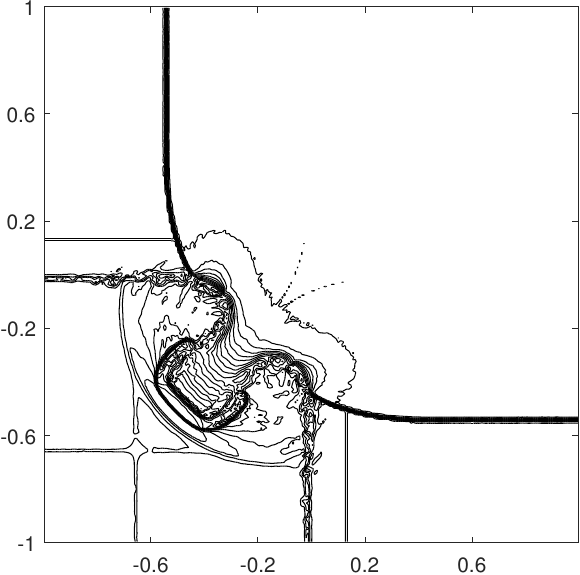}
		\caption{With global entropy-preserving limiter.}
	\end{subfigure}
	\begin{subfigure}[t]{.32\textwidth}
		\centering
		\includegraphics[width=1\textwidth]{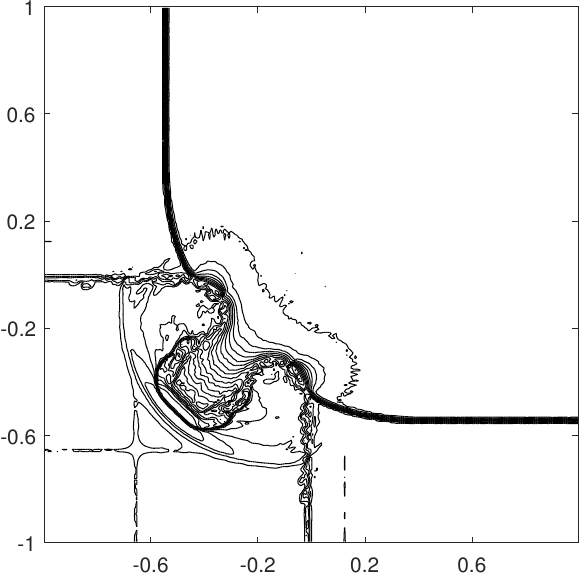}
		\caption{Local entropy-preserving with {\tt LI} \eqref{LV Final star}.}
	\end{subfigure}
	\begin{subfigure}[t]{.32\textwidth}
		\centering
		\includegraphics[width=1\textwidth]{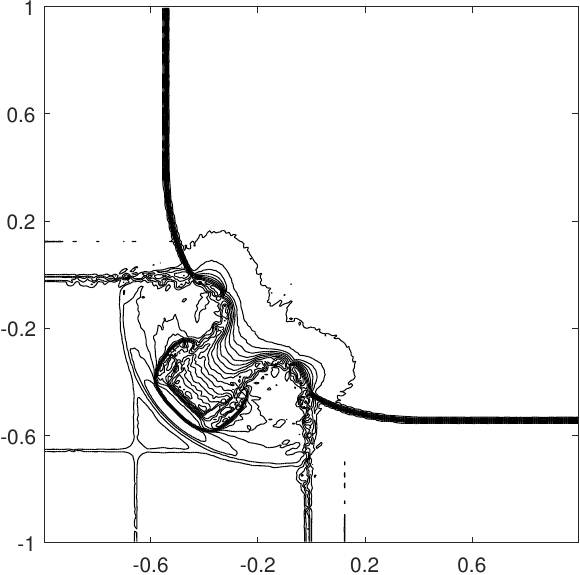}
		\caption{Local entropy-preserving with {\tt CJK} \eqref{CJK}.}
	\end{subfigure}
	\begin{subfigure}[t]{.32\textwidth}
		\centering
		\includegraphics[width=1\textwidth]{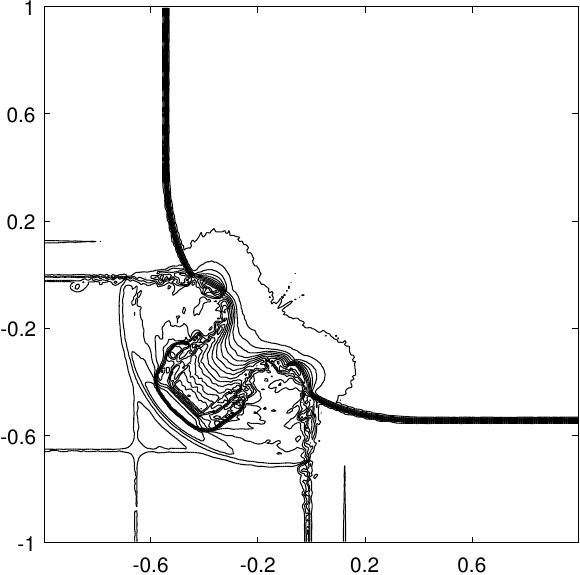}
		\caption{Local entropy-preserving with {\tt New I} \eqref{2D_CUI}.}
	\end{subfigure}
	\begin{subfigure}[t]{.32\textwidth}
		\centering
		\includegraphics[width=1\textwidth]{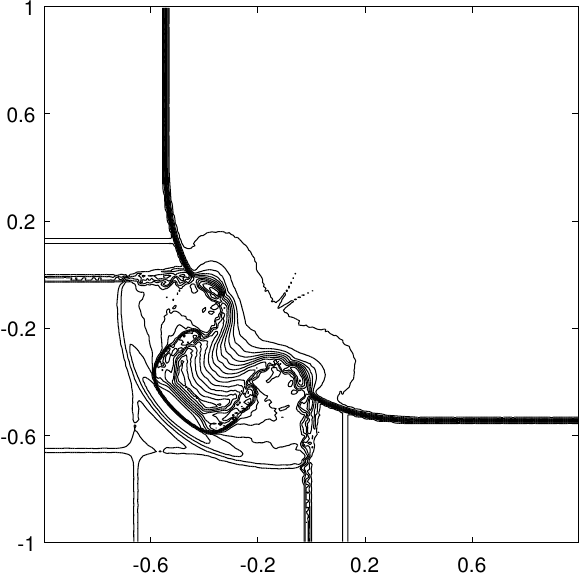}
		\caption{Local entropy-preserving with {\tt New II} \eqref{2DMPrelax}.}
	\end{subfigure}
	\caption{Example \ref{Ex5.2.2}: Contour plots of $\log_{10}\rho$ at $t=0.8$ obtained using $\mathbb{P}^3$-based DG methods without or with entropy-preserving limiters on a uniform mesh with {$\dx = \dy = 1/125$}. 18 equally spaced contour lines are displayed.
	}\label{Fig:2D_RP1_RCEOS}
\end{figure}

\begin{table}[htb]
	\centering
	\caption{CPU time in minutes for simulating Example \ref{Ex5.2.2} up to $t=0.8$ with $250\times 250$ cells.}
	\label{CPU:2D}
	\renewcommand\arraystretch{1.2}
	\scalebox{0.8}{
		\begin{tabular}{ccccc}
			\toprule
			estimation approach & {{\tt LI}} \eqref{LV Final star} & { {\tt CJR}} \eqref{CJK} & {\tt New I} \eqref{2D_CUI} & {\tt New II} \eqref{2DMPrelax}\\
			\midrule
			CPU time & 334.45 & 310.35 & 356.00 & 213.48\\
			\bottomrule
		\end{tabular}
	}
\end{table}

\begin{example}{(2D Riemann Problem II).}\label{Ex5.2.4}  
	The initial condition for the second 2D Riemann problem is given by  
	\begin{equation*}
		\textbf{V}(x,y,0) =
		\begin{cases}
			(0.035145216124503, 0, 0, 0.162931056509027)^\top, & x>0, \ y>0,\\
			(0.1, 0.7, 0, 1)^\top, & x<0, \ y>0,\\
			(0.5, 0, 0, 1)^\top, & x<0, \ y<0,\\
			(0.1, 0, 0.7, 1)^\top, & \text{otherwise}.
		\end{cases}
	\end{equation*}
	This problem involves two contact discontinuities and two shock waves, which interact and merge over time to form a characteristic ``mushroom cloud'' structure at the center of the domain $[-1,1]^2$. In this example, we consider the  {IP-EOS} \eqref{hEOS2}.
	
	The computational domain $[-1,1]^2$ is discretized into a uniform $200\times200$ mesh, with outflow boundary conditions imposed on all sides. We apply the $\mathbb{P}^3$-based DG schemes both with and without entropy-preserving limiters. The numerical results for the rest-mass density $\rho$ at $t=0.8$ are shown in Figure \ref{Fig:2D_RP3_IPEOS}. It is evident that the scheme without an entropy limiter suffers from significant spurious oscillations, whereas the locally entropy-preserving techniques effectively suppress these numerical artifacts, ensuring a more reliable and accurate solution.

	\begin{figure}[!htb]
		\centering
		\begin{subfigure}[t]{.32\textwidth}
			\centering
			\includegraphics[width=1\textwidth]{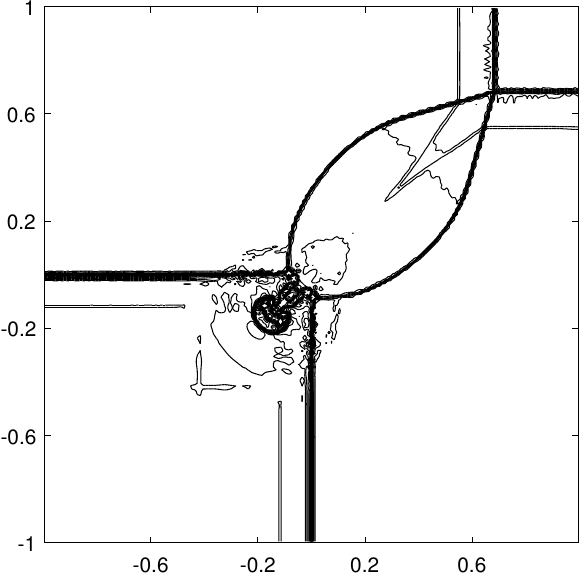}
			\caption{Without entropy-preserving limiter.}
		\end{subfigure}
		\begin{subfigure}[t]{.32\textwidth}
			\centering
			\includegraphics[width=1\textwidth]{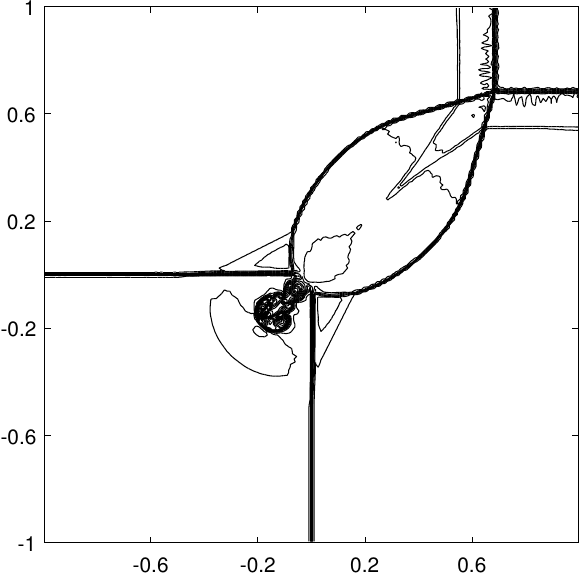}
			\caption{With global entropy-preserving limiter.}
		\end{subfigure}
		\begin{subfigure}[t]{.32\textwidth}
			\centering
			\includegraphics[width=1\textwidth]{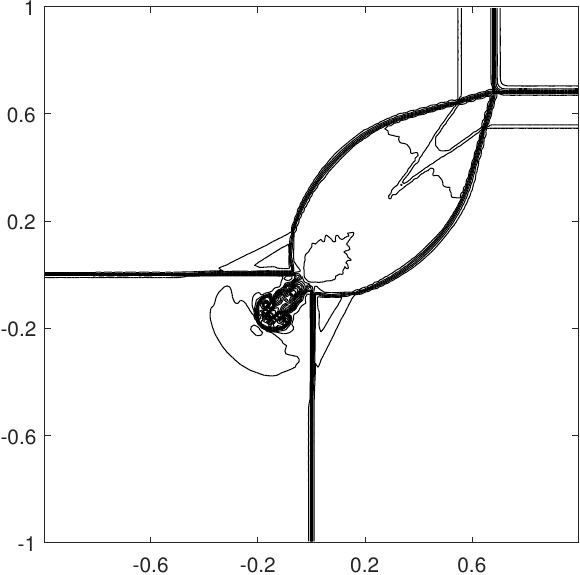}
			\caption{Local entropy-preserving with {\tt LI} \eqref{LV Final star}.}
		\end{subfigure}
		\begin{subfigure}[t]{.32\textwidth}
			\centering
			\includegraphics[width=1\textwidth]{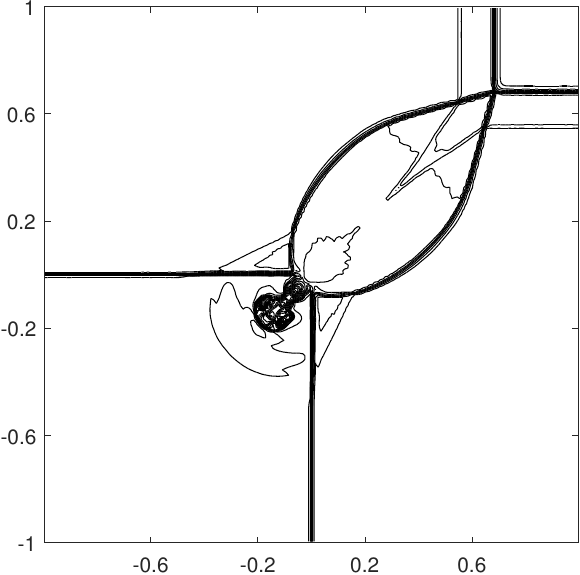}
			\caption{Local entropy-preserving with {\tt CJK} \eqref{CJK}.}
		\end{subfigure}
		\begin{subfigure}[t]{.32\textwidth}
			\centering
			\includegraphics[width=1\textwidth]{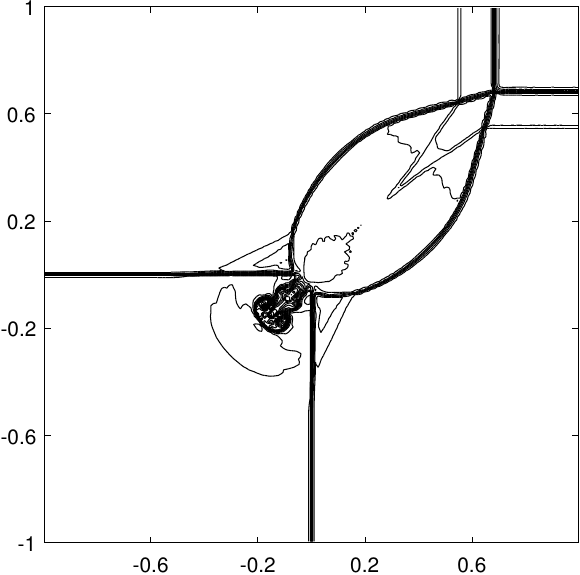}
			\caption{Local entropy-preserving with {\tt New I} \eqref{2D_CUI}.}
		\end{subfigure}
		\begin{subfigure}[t]{.32\textwidth}
			\centering
			\includegraphics[width=1\textwidth]{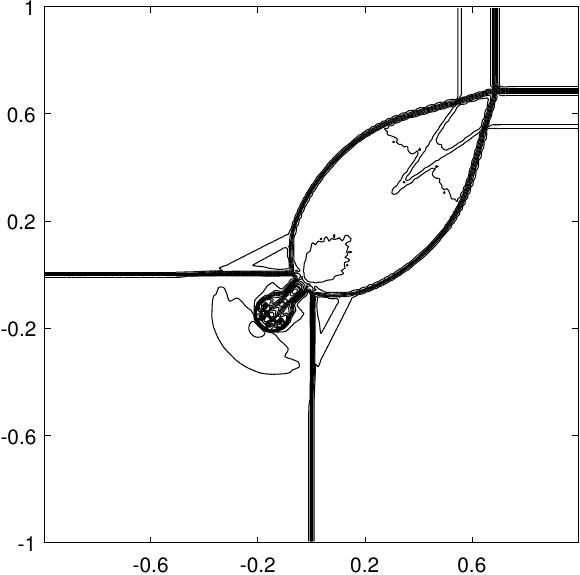}
			\caption{Local entropy-preserving with {\tt New II} \eqref{2DMPrelax}.}
		\end{subfigure}
		\caption{Example \ref{Ex5.2.4}: Contour plots of $\log_{10}\rho$ at $t=0.8$ obtained using $\mathbb{P}^3$-based DG methods without or with entropy-preserving limiters on a uniform mesh with $\dx=\dy=1/100$. 20 equally spaced contour lines are displayed.}\label{Fig:2D_RP3_IPEOS}
	\end{figure}

\end{example}

\begin{example}{(Shock-Bubble Interaction Problem).}\label{Ex5.2.5}  
	This example simulates the interaction between a shock wave and a bubble within the rectangular computational domain $[0,325]\times[-45,45]$. Reflective boundary conditions are imposed at $y=-45$ and $y=45$, while inflow and outflow boundary conditions are applied at $x=325$ and $x=0$, respectively. The problem setup follows \cite{he2012adaptive1}, except that we consider the {RC-EOS} \eqref{hEOS1} in this study.
	
	At the initial time $t = 0$, a shock wave is positioned along $x = 265$, with the left and right states given by
	\begin{equation*}
		\mathbf{V}(x,y,0) = 
		\begin{cases}
			(1, 0, 0, 0.05)^\top, & x<265,\\
			(1.941272902134272,-0.200661045980881, 0, 0.15)^\top, & x>265.
		\end{cases}
	\end{equation*}
	On the left of the shock, a bubble with equal pressure but lower density is initially centered at $(215,0)^\top$:
	\begin{equation*}
		\textbf{V}(x,y,0) = (0.38, 0, 0, 0.05)^\top, \quad \textrm{for} \quad \sqrt{(x-215)^2+y^2} \leq 25.
	\end{equation*}
	
	To investigate the evolution of the shock-bubble interaction, we employ the $\mathbb{P}^2$-based locally entropy-preserving DG scheme to compute the numerical solution up to $t = 450$. The schlieren images of the rest-mass density $\rho$ are displayed in Figure \ref{Fig:2D_shockbb1_RCEOS_035b}. As observed, the proposed locally entropy-preserving DG scheme successfully captures the intricate flow structures without introducing spurious oscillations, while also resolving the Kelvin-Helmholtz instability along the bubble interface with high resolution.
	
	For comparison, the results obtained using the $\mathbb{P}^2$-based DG scheme without entropy-preserving techniques are shown in Figure \ref{Fig:2D_shockbb1_RCEOS_035a}. The absence of entropy-preserving mechanisms leads to severe nonphysical oscillations, particularly in the regions affected by the shock-bubble interaction. These results highlight the crucial role of MEP preservation in maintaining the physical fidelity of the numerical solution.
	
\end{example}

		\begin{figure}[!htb]
			\centering
			\begin{subfigure}[t]{0.32\textwidth}
				\centering
				\includegraphics[width=\textwidth, trim = 0 0 0 0, clip]{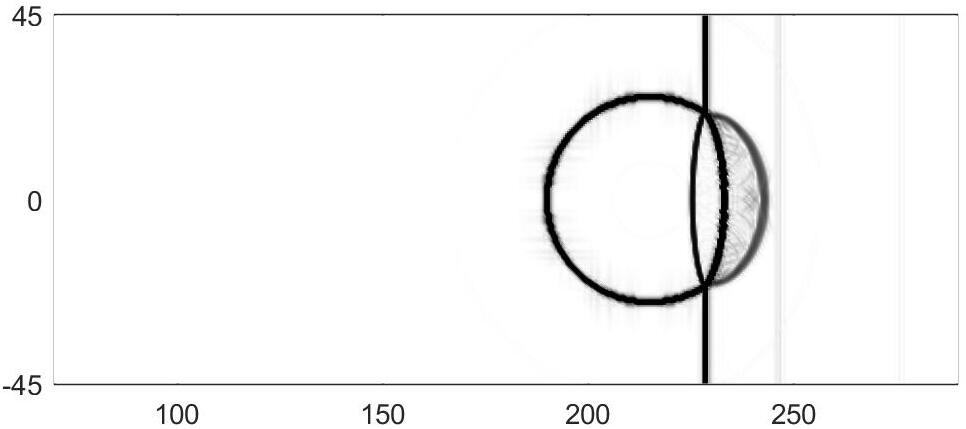}
			\end{subfigure}
			\begin{subfigure}[t]{0.32\textwidth}
				 		\centering
				\includegraphics[width=\textwidth, trim = 0 0 0 0, clip]{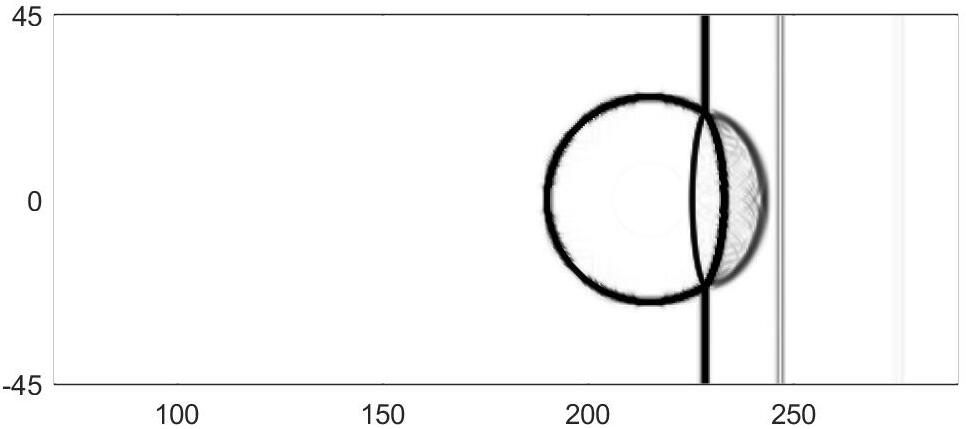}
			\end{subfigure}
			\begin{subfigure}[t]{0.32\textwidth}
				 		\centering
				\includegraphics[width=\textwidth, trim = 0 0 0 0, clip]{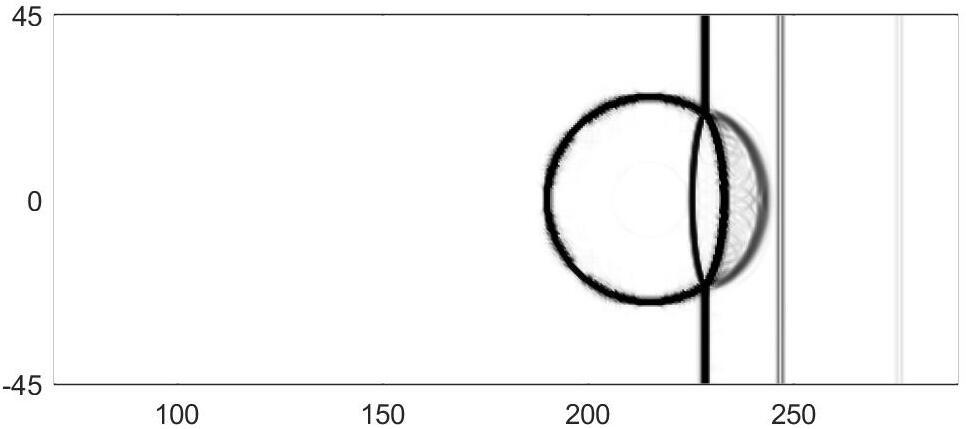}
			\end{subfigure}
			\begin{subfigure}[t]{0.32\textwidth}
				\centering
				\includegraphics[width=\textwidth, trim = 0 0 0 0, clip]{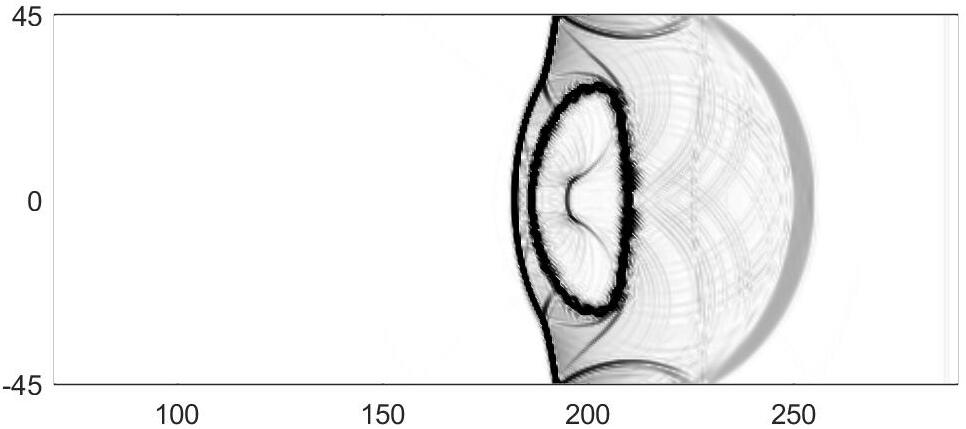}
			\end{subfigure}
			\begin{subfigure}[t]{0.32\textwidth}
						\centering
				\includegraphics[width=\textwidth, trim = 0 0 0 0, clip]{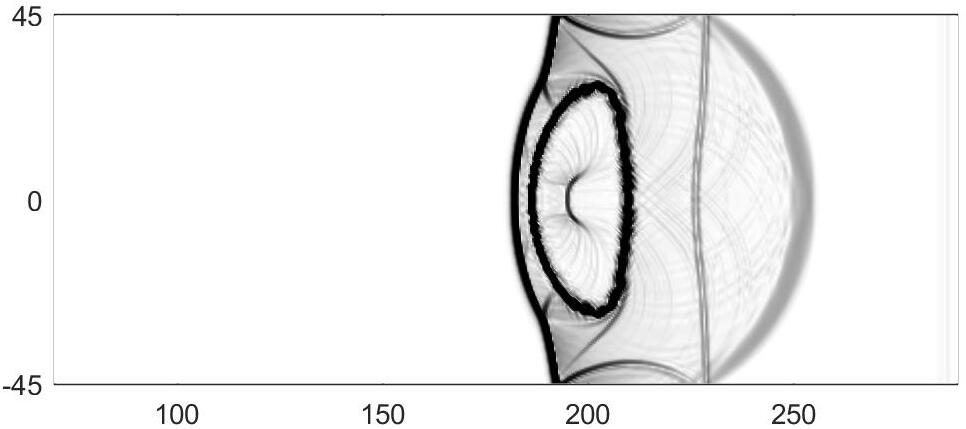}
			\end{subfigure}
			\begin{subfigure}[t]{0.32\textwidth}
						\centering
				\includegraphics[width=\textwidth, trim = 0 0 0 0, clip]{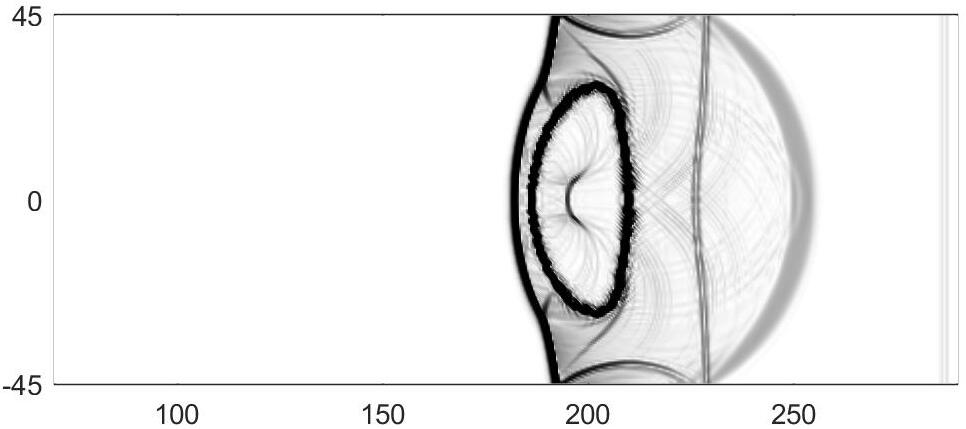}
			\end{subfigure}
			\begin{subfigure}[t]{0.32\textwidth}
				\centering
				\includegraphics[width=\textwidth, trim = 0 0 0 0, clip]{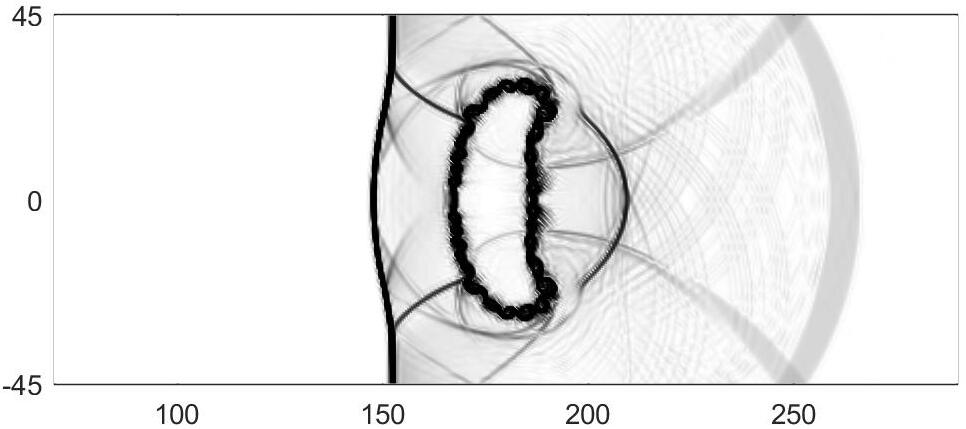}
			\end{subfigure}
			\begin{subfigure}[t]{0.32\textwidth}
				 		\centering
				\includegraphics[width=\textwidth, trim = 0 0 0 0, clip]{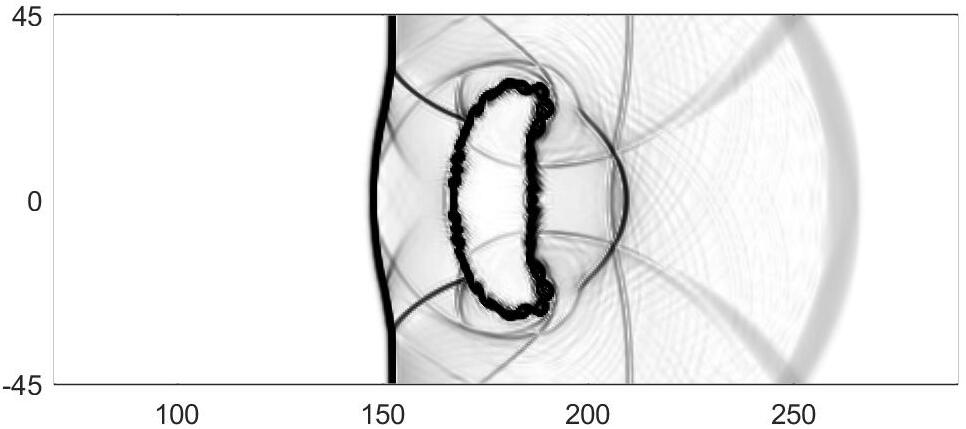}
			\end{subfigure}
			\begin{subfigure}[t]{0.32\textwidth}
						\centering
				\includegraphics[width=\textwidth, trim = 0 0 0 0, clip]{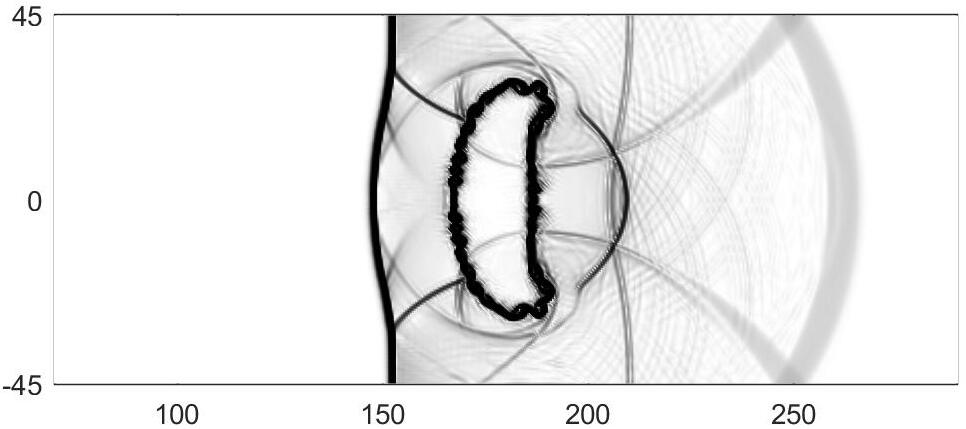}
			\end{subfigure}
			\begin{subfigure}[t]{0.32\textwidth}
				\centering
				\includegraphics[width=\textwidth, trim = 0 0 0 0, clip]{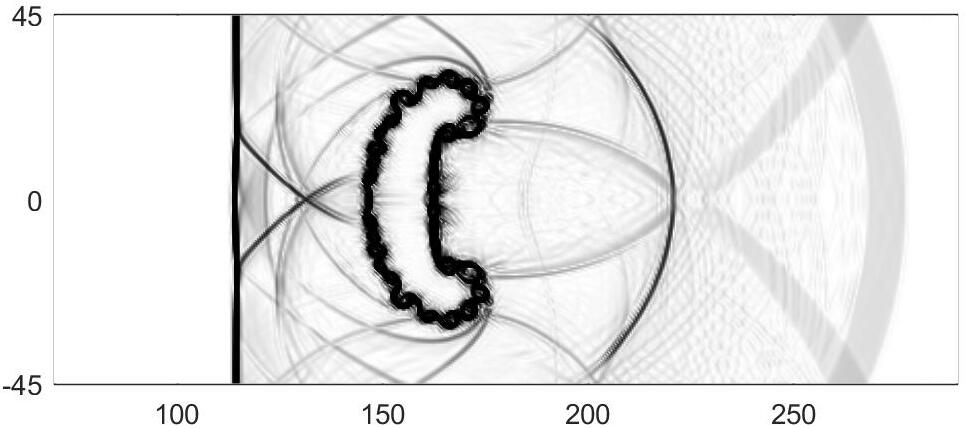}
			\end{subfigure}
			\begin{subfigure}[t]{0.32\textwidth}
				 		\centering
				\includegraphics[width=\textwidth, trim = 0 0 0 0, clip]{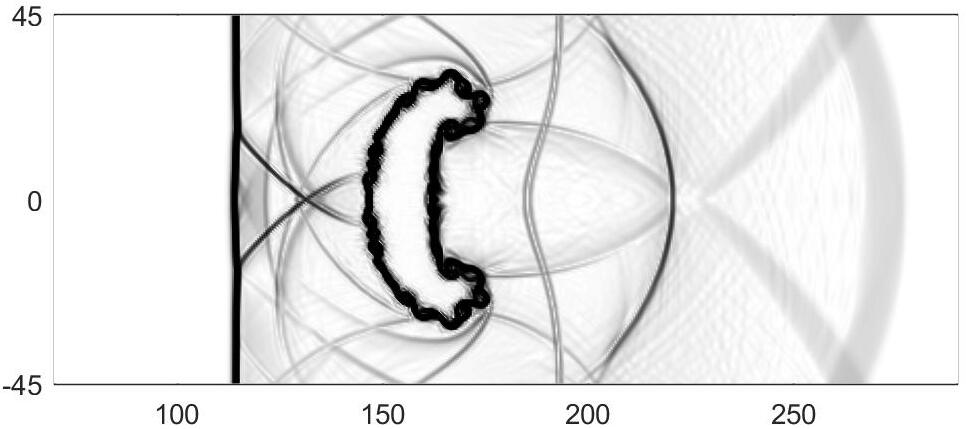}
			\end{subfigure}
			\begin{subfigure}[t]{0.32\textwidth}
				 		\centering
				\includegraphics[width=\textwidth, trim = 0 0 0 0, clip]{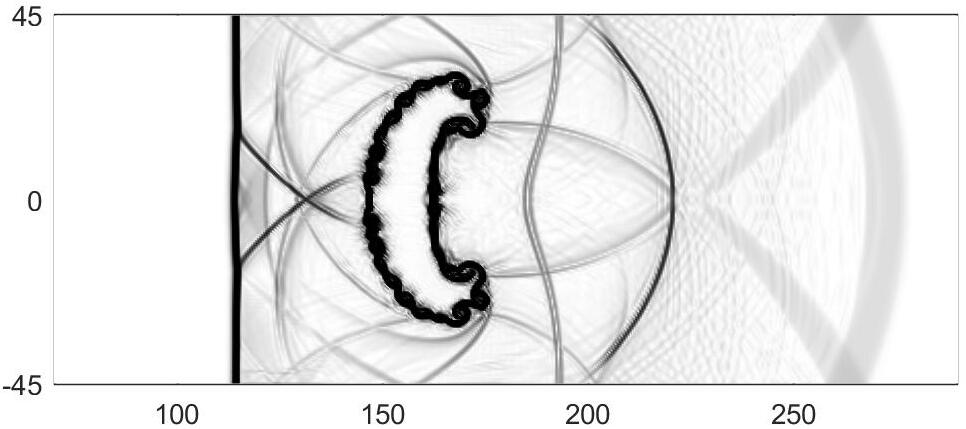}
			\end{subfigure}
			\begin{subfigure}[t]{0.32\textwidth}
				\centering
				\includegraphics[width=\textwidth, trim = 0 0 0 0, clip]{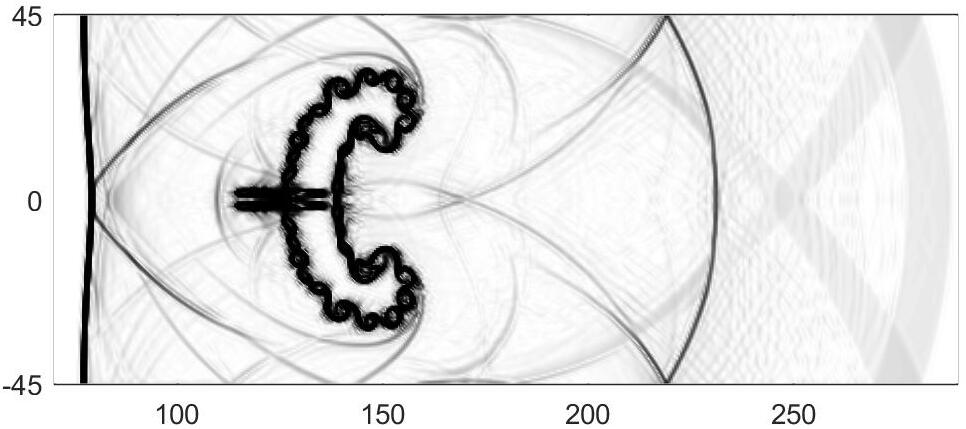}
				\caption{Without entropy-preserving limiter.}\label{Fig:2D_shockbb1_RCEOS_035a}
			\end{subfigure}
			\begin{subfigure}[t]{0.32\textwidth}
						\centering
				\includegraphics[width=\textwidth, trim = 0 0 0 0, clip]{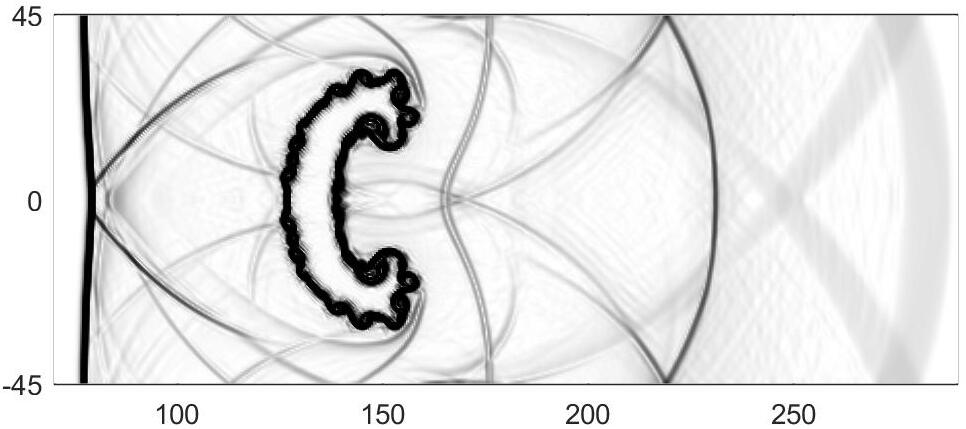}
				\caption{Local entropy-preserving with {\tt New I} \eqref{2D_CUI}.}
				\label{Fig:2D_shockbb1_RCEOS_035b}
			\end{subfigure}
			\begin{subfigure}[t]{0.32\textwidth}
						\centering
				\includegraphics[width=\textwidth, trim = 0 0 0 0, clip]{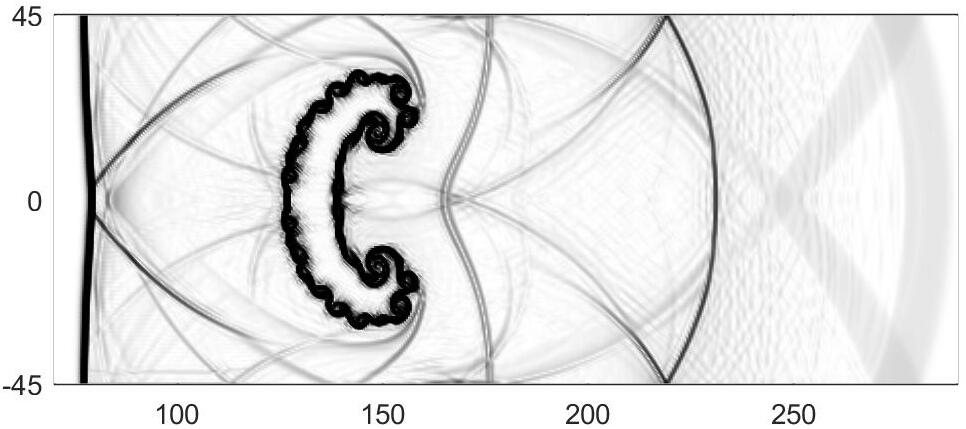}
				\caption{Local entropy-preserving with {\tt New II} \eqref{2DMPrelax}.}
				\label{Fig:2D_shockbb1_RCEOS_035c}
			\end{subfigure}
			\caption{Example \ref{Ex5.2.5}: Schlieren images of $\rho$ at $t=90,180,270,360,450$ (from top to bottom), obtained using the $\mathbb{P}^2$-based DG scheme without an entropy-preserving limiter and with our locally entropy-preserving limiters. Computations are performed on a uniform grid with $\dx = \dy = 1/2$.}
		\end{figure}

\begin{example}{(Shock-Vortex Interaction)} \label{EX:SV}
	This example examines the interaction between a shock wave and a vortex, a problem previously studied for the ideal EOS \eqref{ID-EOS} in \cite{duan2019high,chen2022physical}. 
	Here, we extend this test to the relativistic Euler system with the IP-EOS \eqref{hEOS2}. 
	
	The initial condition is given by
	\begin{equation*}
		\mathbf{V}(x,y,0) = 
		\begin{dcases}
			(33.703677019208184,-0.462656511475986,0,42.706741066043749)^\top =:\mathbf{V}_{\rm L}, & x < -0.6, \\[5pt]
			\left(
			h\big(\hat{\theta}(x,y)\big)^2 \hat{\theta}(x,y),
			\frac{w_1(x,y)-0.7}{1-0.7 \, w_1(x,y)},
			\frac{\sqrt{0.51}w_2(x,y)}{1-0.7 \, w_1(x,y)},
			h\big(\hat{\theta}(x,y)\big)^2 \hat{\theta}^2(x,y)
			\right)^\top, & x \ge -0.6,
		\end{dcases}
	\end{equation*}
	where $h$ is defined in \eqref{hEOS2}, and the functions $\hat{\theta}$, $w_1$, and $w_2$ are given by
	\begin{equation*}
		\hat{\theta}(x,y) = 1-0.02 \, \exp\left(1-\frac{x^2}{0.51}-y^2\right), 
		\quad
		w_1(x,y) = -y \hat{f}(x,y), 
		\quad
		w_2(x,y) = \frac{x}{\sqrt{0.51}}\hat{f}(x,y),
	\end{equation*}
	with
	\begin{equation*}
		\hat{f}(x,y) = \left(
		\frac
		{h\big(\hat{\theta}(x,y)\big)^2}
		{
			4 \big(1-\hat{\theta}(x,y)\big)
			\Big[
			h\big(\hat{\theta}(x,y)\big)
			+
			h'\big( \hat{\theta}(x,y)\big) \cdot \hat{\theta}(x,y)
			\Big]
		}
		+\frac{x^2}{0.51}+y^2
		\right)^{-1/2}.
	\end{equation*}
	This setup ensures that
	\begin{equation*}
		\lim_{x\rightarrow -0.6^-} \mathbf{V}(x,y,0) = \mathbf{V}_{\rm L}, 
		\quad
		\lim_{x\rightarrow -0.6^+} \mathbf{V}(x,y,0) 
		\approx 
		\left((2+\sqrt{5})^2,-0.7, 0, (2+\sqrt{5})^2\right)^\top
		=:\mathbf{V}_{\rm R},
	\end{equation*}
	which represents a stationary shock at $x = -0.6$. On the right side of the shock, an initially centered vortex at $(0,0)$ moves leftward with a velocity of $0.7$. This initial condition is deduced by following the analytical method proposed in \cite{duan2022analytical}. 
	The computational domain is $[-17,3] \times [-5,5]$ with reflective boundary conditions at $y = \pm5$, inflow at $x=3$, and outflow at $x=-17$.
	
	We solve this problem up to $t = 20$ using the $\mathbb{P}^3$-based DG method with and without the entropy-preserving limiter on a uniform grid with $\Delta x = \Delta y = 1/40$. 
	The contour plots of $\log_{10}(1+|\nabla \rho|)$ are presented in Figure \ref{Fig:2D_SV_IPEOS}. 
	As the vortex interacts with the stationary shock wave, intricate wave structures emerge. The locally entropy-preserving DG schemes, incorporating either our {\tt New I} or {\tt New II} approach for entropy bound estimation, effectively capture these wave patterns. 
	In contrast, the results obtained without enforcing the MEP exhibit significant nonphysical oscillations, severely compromising physical reliability.

\begin{figure}[!htb]
	\centering
	\begin{subfigure}[t]{.32\textwidth}
		\centering
		\includegraphics[width=1\textwidth]{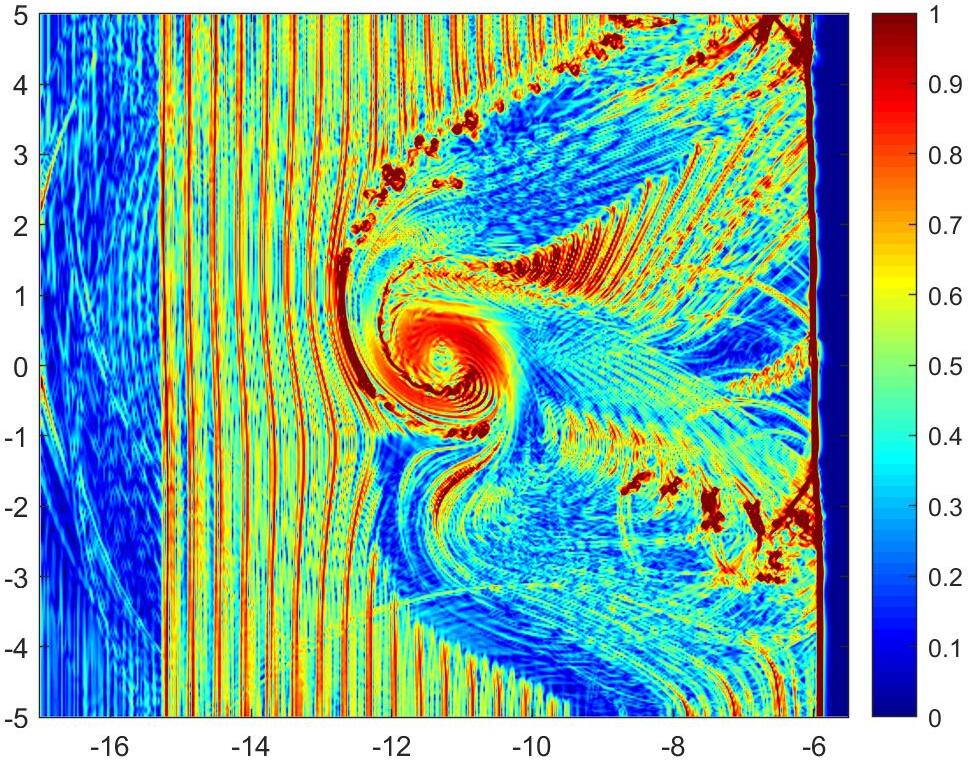}
		\caption{Without entropy-preserving limiter.}
	\end{subfigure}
	\begin{subfigure}[t]{.32\textwidth}
		\centering
		\includegraphics[width=1\textwidth]{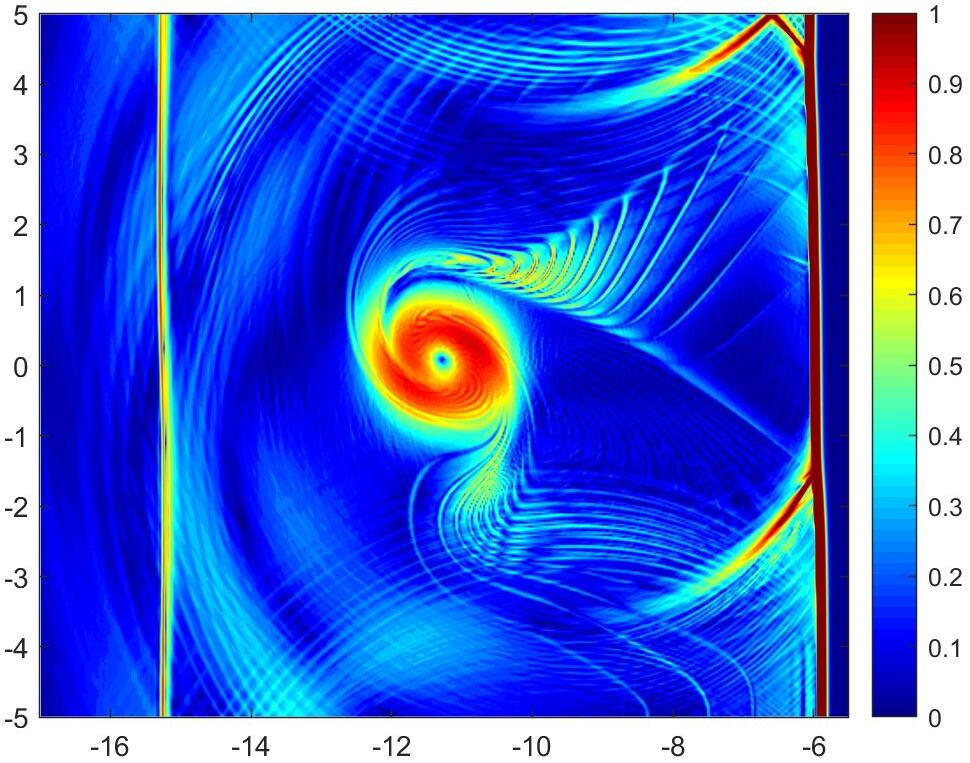}
		\caption{Local entropy-preserving with {\tt New I} \eqref{2D_CUI}.}
	\end{subfigure}
    \begin{subfigure}[t]{.32\textwidth}
		\centering
		\includegraphics[width=1\textwidth]{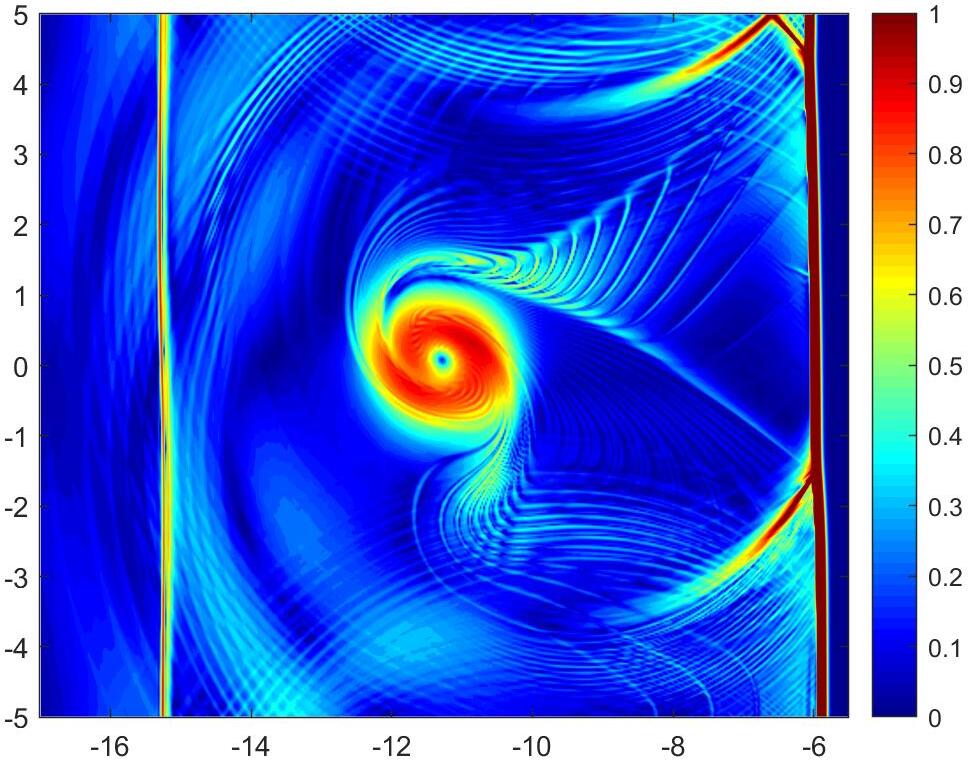}
		\caption{Local entropy-preserving with {\tt New II} \eqref{2DMPrelax}.}
	\end{subfigure}
	\caption{Example \ref{EX:SV}: Contour plots of ${\rm log}_{10}(1+|\nabla \rho|)$ at $t = 20$, obtained using the $\mathbb{P}^3$-based DG scheme without an entropy-preserving limiter and with our locally entropy-preserving limiters. Computations are performed on a uniform grid with $\Delta x = \Delta y = 1/40$.}\label{Fig:2D_SV_IPEOS}
\end{figure}

\end{example}


\section{Conclusions}
In this work, we have established the minimum entropy principle (MEP) for the relativistic Euler equations with a broad class of general equations of state (EOSs) that satisfy relativistic causality.  
At the continuous level, we identified a family of entropy pairs for the relativistic Euler equations and rigorously proved the strict convexity of entropy under a necessary and sufficient condition. This theoretical foundation allowed us to establish the relativistic MEP, overcoming the difficulties posed by the lack of explicit formulas for entropy and flux functions in terms of conservative variables. 

At the numerical level, we developed a rigorous framework for constructing provably entropy-preserving high-order schemes. One of the central challenges was the highly nonlinear and implicit dependence of entropy on the conservative variables, making it particularly difficult to enforce entropy preservation. To overcome this, we established a series of auxiliary theories through highly technical estimates, which provided key mathematical tools for proving entropy preservation in the relativistic setting. Additionally, we leveraged the geometric quasi-linearization (GQL) technique to reformulate nonlinear entropy constraints into equivalent linear ones, by introducing additional free parameters. 
Another key contribution of our work is the development of novel, robust, locally entropy-preserving high-order frameworks. In particular, we introduced two new approaches for estimating local lower bounds of specific entropy, which proved effective in both smooth and discontinuous cases, even in the presence of shock waves at unknown locations. Numerical experiments demonstrated that our entropy-preserving methods maintain high-order accuracy while effectively suppressing spurious oscillations, outperforming existing local entropy minimum estimation techniques.
Moreover, our approach is not limited to Euler equations but provides a versatile framework applicable to other models that admit an MEP. By establishing a rigorous theoretical foundation and developing robust numerical techniques, this study represents a significant step forward in the understanding and preservation of entropy in relativistic hydrodynamics. 

Future research will explore extending these methods to more complex relativistic systems, such as magnetohydrodynamics (MHD) and general relativistic flows, further broadening the applicability of our entropy-preserving framework.

\renewcommand\baselinestretch{0.86}

	\bibliographystyle{siamplain}
	\bibliography{references_short}

\end{document}